\documentclass{amsart}
\usepackage{amssymb}
\usepackage{amsmath}
\usepackage{amsfonts}
\usepackage{geometry}

\newtheorem{theorem}{Theorem}[section]

\newtheorem{corollary}[theorem]{Corollary}

\newtheorem{definition}[theorem]{Definition}
\newtheorem{example}[theorem]{Example}

\newtheorem{lemma}[theorem]{Lemma}

\newtheorem{proposition}[theorem]{Proposition}
\newtheorem{remark}[theorem]{Remark}

\input{tcilatex}
\begin{document}
\title[Asymptotics of the Yang-Mills Flow]{ Asymptotics of the Yang-Mills
Flow for Holomorphic Vector Bundles Over K\"{a}hler Manifolds: The Canonical
Structure of the Limit}
\author{Benjamin Sibley}
\address{Department of Mathematics, University of Maryland, College Park, MD
20742, USA }
\email{bsibley@math.umd.edu}
\maketitle

\begin{abstract}
In the following article we study the limiting properties of the Yang-Mills
flow associated to a holomorphic vector bundle $E$ over an arbitrary compact
K\"{a}hler manifold $(X,\omega )$. In particular we show that the flow is
determined at infinity by the holomorphic structure of $E$. Namely, if we
fix an integrable unitary reference connection $A_{0}$ defining the
holomorphic structure, then the Yang-Mills flow with initial condition $%
A_{0} $, converges (away from an appropriately defined singular set) in the
sense of the Uhlenbeck compactness theorem to a holomorphic vector bundle $%
E_{\infty }$, which is isomorphic to the associated graded object of the
Harder-Narasimhan-Seshadri filtration of $(E,A_{0})$. Moreover, $E_{\infty }$
extends as a reflexive sheaf over the singular set as the double dual of the
associated graded object. This is an extension of previous work in the cases
of $1$ and $2$ complex dimensions and proves the general case of a
conjecture of Bando and Siu.
\end{abstract}

\tableofcontents

\section{Introduction}

This paper is a study of the Yang-Mills flow, the $L^{2}$-gradient flow of
the Yang-Mills functional; and in particular its convergence properties at
infinity. The flow is (after imposing the Coulomb gauge condition) a
parabolic equation for a connection on a holomorphic vector bundle. Very
soon after the introduction of the flow equations, Donaldson and Simpson
proved that in the case of a stable bundle the gradient flow converges
smoothly at infinity (see \cite{DO1},\cite{DO2},\cite{SI}). In the unstable
case the behaviour of the flow is more ambiguous. Nevertheless, even in the
general case there is an appropriate notion of convergence (a version of
Uhlenbeck's compactness theorem) that is always satisfied. The goal of this
article is to prove that this notion depends only on the holomorphic
structure of the original bundle.

We follow up on work whose origin lies in two principal directions, both
related to stability properties of holomorphic vector bundles over compact K%
\"{a}hler manifolds. The first strain is the seminal work of Atiyah and Bott 
\cite{AB}, in which the authors study the moduli space of stable holomorphic
bundles over Riemann surfaces. In particular, they computed the $\mathcal{G}%
^{%
\mathbb{C}
}$-equivariant Betti numbers of this space in certain cases, where $\mathcal{%
G}^{%
\mathbb{C}
}$ is the complex gauge group of a holomorphic vector bundle $E$ (over a
Riemann surface $X$) acting on the space $\mathcal{A}_{\limfunc{hol}}$ of
holomorphic structures of $E$. Their approach was to stratify $\mathcal{A}_{%
\limfunc{hol}}$ by Harder-Narasimhan type. The type is a tuple of rational
numbers $\mu =(\mu _{1},\cdots ,\mu _{R})$ associated to a holomorphic
structure $(E,\bar{\partial}_{E})$, defined using a filtration of $E$ by
analytic subsheaves whose successive quotients are semi-stable, called the
Harder-Narasimhan filtration. One of the resulting strata of $\mathcal{A}_{%
\limfunc{hol}}$ consists of the semi-stable bundles. Furthermore the action
of $\mathcal{G}^{%
\mathbb{C}
}$ preserves the stratification, and the main result that yields the
computation of the equivariant Betti numbers is that the stratification by
Harder-Narasimhan type is equivariantly perfect under this action.

Atiyah and Bott also noticed that the problem might be amenable to a more
analytic approach. Specifically they considered the Yang-Mills functional $%
YM $ on the space $\mathcal{A}_{h}$ of integrable, unitary connections with
respect to a fixed hermitan metric on $E$. The space $\mathcal{A}_{h}$ may
be identified with $\mathcal{A}_{\limfunc{hol}}$ by sending a connection $%
\nabla _{A}$ to its $(0,1)$ part $\bar{\partial}_{A}$. The Yang-Mills
functional is defined by taking the $L^{2}$ norm of the curvature of $\nabla
_{A}$, and is a Morse function on $\mathcal{A}_{h}/\mathcal{G}$, where $%
\mathcal{G}$ is the unitary gauge group. Therefore this functional induces
the usual stable-unstable manifold stratification on $\mathcal{A}_{h}$ (or
equivalently $\mathcal{A}_{\limfunc{hol}}$) familiar from Morse theory. It
is natural to conjecture that this analytic stratification is in fact the
same as the algebraic stratification given by the Harder-Narasimhan type.
The authors of \cite{AB} stopped short of proving this statement, instead
leaving it at the conjectural level, and working directly with the algebraic
stratification. They noted however that a key technical point in proving the
equivalence was to show the convergence of the gradient flow of the
Yang-Mills functional at infinity. This was proven in \cite{D} by
Daskalopoulos (see also \cite{R}). Specifically, in the case of Riemann
surfaces, Daskalopoulos showed the asymptotic convergence of the Yang-Mills
Flow, that there is indeed a well-defined stratification in the sense of
Morse theory in this case, and that it coincides with the algebraic
stratification (which makes sense in all dimensions).

When $(X,\omega )$ is a higher dimensional K\"{a}hler manifold, the
Yang-Mills flow fails to converge in the usual sense. This brings us to the
second strain of ideas of which the present paper is a continuation: the
so-called \textquotedblleft Kobayashi-Hitchin
correspondences\textquotedblright . These are statements (in various levels
of generality) relating the existence of Hermitian-Einstein metrics on a
holomorphic bundle $E$, to the stability of $E$. Namely, $E$ admits an
Hermitian-Einstein metric if and only if $E$ is polystable. This was first
proven in the case of a Riemann surface by Narasimhan and Seshadri. Their
proof did not use differential geometry, and the condition that the bundle
admits an Hermitian-Einstein connection was originally formulated purely in
terms of representations of the fundamental group of the Riemann surface. It
was Donaldson who gave the first proof using gauge theory, reformulating the
statement in terms of a metric of constant central curvature. He initially
did this in the case of a Riemann surface in \cite{DO3} by considering
sequences of connections in a complex gauge orbit that are minimising for a
certain functional, which is analogous to our $HYM_{\alpha }$ functionals
defined in Section $3.2$. Shortly after this, Donaldson extended the result
to the case of algebraic surfaces in \cite{DO1}, and later to the case of
projective complex manifolds in \cite{DO2}. In both \cite{DO1} and \cite{DO2}
the idea of the proof was to reformulate the flow as an equivalent parabolic 
$PDE$, show long-time existence of the equation, and then prove that for a
stable bundle, this modified flow indeed converges, the solution being the
desired Hermitian-Einstein metric. This was generalised by Uhlenbeck and Yau
in \cite{UY} in the case of a compact K\"{a}hler manifold using different
methods. Finally, in \cite{BS}, Bando and Siu extended the correspondence to
coherent analytic sheaves on K\"{a}hler manifolds by considering what they
called \textquotedblleft admissible\textquotedblright\ hermitian metrics,
which are metrics on the locally free part of the sheaf having controlled
curvature. They also conjectured that there should also be a correspondence
(albeit far less detailed) between the Yang-Mills flow and the
Harder-Narasimhan filtration in higher dimensions despite the absence of a
Morse theory for the Yang-Mills functional.

There are two main features that distinguish the higher dimensional case
from the case of Riemann surfaces. As previously mentioned, the flow does
not converge in general. However, the only obstruction to convergence is
bubbling phenomena. Specifically, one of Uhlenbeck's compactness results
(see \cite{UY} Theorem $5.2$) applies to the flow, which means that there
are always subsequences that converge (in a certain Sobolev norm) away from
a singular set of Hausdorff codimension at least $4$ inside $X$ (which we
will denote by $Z_{\limfunc{an}}$), to a connection on a possibly different
vector bundle $E_{\infty }$. A priori, this pair of a limiting connection
and bundle depends on the subsequence. In the case of two complex
dimensions, the singular set is a locally finite set of points (finite in
the compact case) and by Uhlenbeck's removable singularities theorem $%
E_{\infty }$ extends over the singular set as a vector bundle with a
Yang-Mills connection. In higher dimensions, again due to a result of Bando
and Siu, $E_{\infty }$ extends over the singular set, but only as a
reflexive sheaf. Although we will not use their result, Hong and Tian have
proven in \cite{HT} that in fact the convergence is in $C^{\infty }$ on the
complement of $Z_{\limfunc{an}}$ and that $Z_{\limfunc{an}}$ is a
holomorphic subvariety.

A separate, but intimately related issue is the Harder-Narasimhan
filtration. In the case of a Riemann surface the filtration is given by
subbundles. In higher dimensions, it is only a filtration by subsheaves.
Again however, away from a singular set $Z_{\limfunc{alg}}$, which is a
complex analytic subset of $X$ of complex codimension at least $2$, the
filtration is indeed given by subbundles. Once more, in the case of a K\"{a}%
hler surface this is a locally finite set of points (finite in the compact
case).

The main result of this paper (the conjecture of Bando and Siu), describes
the relationship between the analytic and algebraic sides of the above
picture. To state it, we recall that there is a refinement of the
Harder-Narasimhan filtration called the Harder-Narasimhan-Seshadri
filtration, which is a double filtration whose successive quotients are
stable rather than merely semi-stable. Then if $(E,\bar{\partial}_{E})$ is a
holomorphic vector bundle where the operator $\bar{\partial}_{E}$ denotes
the holomorphic structure, write $Gr_{\omega }^{HNS}(E,\bar{\partial}_{E})$
for the associated graded object (the direct sum of the stable quotients) of
the Harder-Narasimhan-Seshadri filtration. Notice that by the
Kobayashi-Hitchin correspondence, $Gr_{\omega }^{HNS}(E,\bar{\partial}_{E})$
also carries a natural Yang-Mills connection on its locally free part, given
by the direct sum of the Hermitian-Einstein connections on each of the
stable factors, and this connection is unique up to gauge. The main theorem
says in particular that the limiting bundle along the flow is in fact
independent of the subsequence chosen in order to employ Uhlenbeck
compactness, and is determined entirely by the holomorphic structure $\bar{%
\partial}_{E}$ of $E$. Furthermore, the limiting connection is precisely the
connection on $Gr_{\omega }^{HNS}(E,\bar{\partial}_{E})$.

\begin{theorem}
\label{THM:MainTheorem}Let $(X,\omega )$ be a compact K\"{a}hler manifold,
and $E\rightarrow X$ be an hermitian vector bundle. Let $A_{0}$ denote an
integrable, unitary connection endowing $E$ with a holomorphic structure $%
\bar{\partial}_{E}=\bar{\partial}_{A_{0}}$. Let $A_{\infty }$ denote the
Yang-Mills connection on $Gr_{\omega }^{HNS}(E,\bar{\partial}_{E})$
restricted to $X-Z_{\limfunc{alg}}$ induced from the Kobayashi-Hitchin
correspondence. Let $A_{t}$ be the time $t$ solution of the flow with
initial condition $A_{0}$. Then as $t\rightarrow \infty $, $A_{t}\rightarrow
A_{\infty }$ in the sense of Uhlenbeck, and on $X-Z_{\limfunc{alg}}\cup Z_{%
\limfunc{an}}$, the vector bundles $Gr_{\omega }^{HNS}(E,\bar{\partial}_{E})$
and the limiting bundle $E_{\infty }$ are holomorphically isomorphic.
Moreover, $E_{\infty }$ extends over $Z_{\limfunc{an}}$ as a reflexive sheaf
to $\left( Gr_{\omega }^{HNS}(E,\bar{\partial}_{E})\right) ^{\ast \ast }$.
\end{theorem}

This theorem was proven in \cite{DW1} by Daskalopoulos and Wentworth in the
case when $\dim X=2$. In this case, the filtration consists of vector
bundles, whose successive quotients may have point singularities. As stated
earlier, this means $E_{\infty }$ extends as a vector bundle and \cite{DW1}
proves that this bundle is isomorphic to the vector bundle $\left(
Gr_{\omega }^{HNS}(E,\bar{\partial}_{E})\right) ^{\ast \ast }$.

We now give an overview of our proof, pointing out what goes through
directly from \cite{DW1} and where we require new arguments. Section $2$
consists of the basic definitions we need from sheaf theory, including the
Harder-Narasimhan and Harder-Narasimhan-Seshadri filtrations and their
associated graded objects, as well as the corresponding types. We also
discuss the Yang-Mills functional, the Hermitian-Yang-Mills functional and
the version of the Uhlenbeck compactness result that we will need. Although
we will primarily be concerned with the flow, the proof of Theorem \ref%
{THM:MainTheorem} is set up to work for slightly more general sequences of
connections, so we state the compactness theorem in this generality first,
and specialise to the flow when appropriate. Lastly, we recall the notion of
a weakly holomorphic projection operator associated to a subsheaf first
introduced in \cite{UY}, the Chern-Weil formula, and a lemma on the
boundedness of second fundamental forms from \cite{DW1}.

In Section $3$ we introduce the Yang-Mills flow and its basic properties. We
recast Uhlenbeck compactness in the context of the flow, which satisfies the
boundedness conditions required to apply the general theorem. We recall one
of the main results of \cite{DW1}, that the Harder-Narasimhan type of an
Uhlenbeck limit is bounded from below by the type of the initial bundle with
respect to the partial ordering on types. Finally, Section $3$ ends with a
discussion of Yang-Mills type functionals associated to $\limfunc{Ad}$%
-invariant convex functions on the Lie algebra of the unitary group.

Section $4$ details the main results we will need about resolution of
singularities.\ This is the first place in which our presentation differs
fundamentally from that of \cite{DW1}. The main strategy of the proof is to
eliminate the singular set of the Harder-Narasimhan-Seshadri filtration by
blowing up, and doing all the necessary analysis on the blowup. In the
two-dimensional case, since the singularities consist only of points, this
can be done directly by hand as in \cite{DW1} see also \cite{BU1}. In the
general case we must appeal to the resolution of singularities theorem of
Hironaka see \cite{H1} and \cite{H2}. We consider the filtration as a
rational section of a flag bundle, and apply the resolution of indeterminacy
theorem for rational maps. If we write $\pi :\tilde{X}\rightarrow X$ for the
composition of the blowups involved in resolution, the result is that the
pullback bundle $\pi ^{\ast }E\rightarrow \tilde{X}$ has a filtration by
subbundles, which away from the exceptional divisor $\mathbf{E}$ is
precisely the filtration on $X$.

We will need to consider a natural family of K\"{a}hler metrics $\omega
_{\varepsilon }$ on $\tilde{X}$, which are perturbations of the pullback
form $\pi ^{\ast }\omega $ by the irreducible components of the exceptional
divisor, and which are introduced in order to compensate for the fact that $%
\pi ^{\ast }\omega $ fails to be a metric on $\mathbf{E}$. The filtration of 
$\pi ^{\ast }E$ by subbundles is not quite the Harder-Narasimhan-Seshadri
filtration with respect to $\omega _{\varepsilon }$ but is closely related.
In particular, the main result of this section is that the Harder-Narasimhan
type of $\pi ^{\ast }E$ with respect to $\omega _{\varepsilon }$ converges
to the type of $E$ with respect to $\omega $. This was proven in the surface
case in \cite{DW1} using an argument of Buchdahl from \cite{BU1}. The proof
contained in \cite{DW1} seems to be insufficient in the higher dimensional
case, so we give a rather different proof of this result. The main
ingredient is a bound on the $\omega _{\varepsilon }$ degree of a subsheaf
of $\pi ^{\ast }E$ with torsion-free quotient in terms of its pushforward
sheaf that is uniform as $\varepsilon \rightarrow 0$. To prove this we use
standard algebro-geometric facts together with a modification of an argument
of Kobayashi \cite{KOB} first used to prove the uniform boundedness of the
degree of subsheaves of a vector bundle with respect to a fixed K\"{a}hler
metric. In particular we prove the following theorem:

\begin{theorem}
Let $(X,\omega )$ be a compact K\"{a}hler manifold and $\tilde{S}$ be a
subsheaf (with torsion free quotient $\tilde{Q}$) of a holomorphic vector
bundle $\tilde{E}$ on $\tilde{X}$, where $\pi :\tilde{X}\rightarrow X$ is
given by a sequence of blowups along complex submanifolds of $\limfunc{codim}%
\geq 2$. Then then there is a uniform constant $M$ such that the degrees of $%
\tilde{S}$ and $\tilde{Q}$ with respect to $\omega _{\varepsilon }$ satisfy: 
$\deg (\tilde{S},\omega _{\varepsilon })\leq \deg (\pi _{\ast }\tilde{S}%
)+\varepsilon M$, and $\deg (\tilde{Q},\omega _{\varepsilon })\geq \deg (\pi
_{\ast }\tilde{Q})-\varepsilon M$.
\end{theorem}

Similar statements are proven in the case of a surface by Buchdahl \cite{BU1}
and for projective manifolds by Daskalopoulos and Wentworth see \cite{DW3}.

Section $5$ is the technical heart of the proof. An essential fact needed to
complete the proof of Theorem \ref{THM:MainTheorem} is that the
Harder-Narasimhan type of the limiting sheaf is in fact equal to the type of
the initial bundle. This fact seems to be closely related to the existence
of what is called an $L^{p}$-approximate critical hermitian structure. In
rough terms this is an hermitian metric on a holomorphic vector bundle whose
Hermitian-Einstein tensor is $L^{p}$-close to that of a Yang-Mills
connection (a critical value) determined by the Harder-Narasimhan type of
the bundle (see Definition \ref{Def6}). Since any connection on $E$ has
Hermitian-Yang-Mills energy bounded below by the type of $E$, and we have a
monotonicity property along the flow, the result of Section $3$ implies that
the existence of an approximate structure then ensures that the flow
starting from this initial condition realises the correct type in the limit.
Then one shows that \textit{any} initial condition flows to the correct
type, essentially by proving that the set of such metrics is open and closed
(and non-empty by the existence of an approximate structure) in the space of
smooth metrics, and applying the connectivity of the latter space. This last
argument appears in detail in \cite{DW1} and we do not repeat it. The main
theorem of Section $5$ is the following: \bigskip

\begin{theorem}
Let $E\rightarrow X$ be a holomorphic vector bundle over a K\"{a}hler
manifold with K\"{a}hler form $\omega $. Then given $\delta >0$ and any $%
1\leq p<\infty $, $E$ has an $L^{p}\ \delta $-approximate critical hermitian
structure.
\end{theorem}

The method does not extend to $p=\infty $. This is straightforward in the
case when the filtration is given by subbundles (even for $p=\infty $).
Given an exact sequence of holomorphic vector bundles:%
\begin{equation*}
0\longrightarrow S\longrightarrow E\longrightarrow Q\longrightarrow 0
\end{equation*}%
and hermitian metrics on $S$ and $Q$, one can scale the second fundamental
form $\beta \mapsto t\beta $ to obtain an isomorphic bundle whose
Hermitian-Einstein tensor is close to the direct sum of those of $S$ and $Q$%
. In general it seems difficult to do this directly. The problem here is
that the filtration is not in general given by subbundles, and so the vast
majority Section $5$ is an argument needed to address this point. This is
precisely where we need the resolution of the filtration obtained in Section 
$4$. We first take the direct sum of the Hermitian-Einstein metrics on the
stable quotients in the resolution by subbundles, which sits inside the
pullback $\pi ^{\ast }E$ under the blowup map $\pi :\tilde{X}\rightarrow X$.
Then the argument above shows that after modifying this metric by a gauge
transformation, its Hermitian-Einstein tensor becomes close to the type in
the $L^{p}$ norm. We complete the proof by pushing this metric down to $%
E\rightarrow X$ using a cutoff argument.

In broad outline our discussion in Section $5$ follows the ideas in \cite%
{DW1}. The principal difference is that the authors of \cite{DW1} were able
to rely on the fact that the singular set was given by points when applying
the cutoff argument, in particular they knew that there were uniform bounds
on the derivatives of the cutoff function. We must allow for the fact that
the singular set is higher dimensional, and therefore need to replace their
arguments involving coverings of the singular set by disjoint balls of
arbitrarily small radius by calculations in a tubular neighbourhood. We
first assume $Z_{\limfunc{alg}}$ is smooth and that blowing up once along $%
Z_{\limfunc{alg}}$ resolves the singularities. The essential point is that
the Hausdorff codimension of $Z_{\limfunc{alg}}$ is large enough to allow
the arguments of \cite{DW1} to go through in this case. We then reduce the
general theorem to this case by applying an inductive argument on the number
of blowups required to resolve the filtration. It is here that we crucially
use the convergence of the Harder-Narasimhan type proven in section $4$.

In Section $6$, following Bando and Siu, we introduce a degenerate
Yang-Mills flow on the composition of blowups $\tilde{X}$ with respect to
the degenerate metric $\pi ^{\ast }\omega $. We review some basic properties
of this flow that are necessary for the proof of Theorem \ref%
{THM:MainTheorem}. In particular we show that a solution of this degenerate
flow is in fact an hermitian metric, and solves the ordinary flow equations
with respect to the metric $\pi ^{\ast }\omega $ away from the exceptional
divisor $\mathbf{E}$.

Section $7$ completes the proof of the main theorem by showing the
isomorphism of the limit $E_{\infty }$ with $\left( Gr_{\omega }^{HNS}(E,%
\bar{\partial}_{E})\right) ^{\ast \ast }$. The basic idea follows that of 
\cite{DW1} which in turn is a generalisation of the argument of Donaldson in 
\cite{DO1}. His idea is to construct a non-zero holomorphic map to the
limiting bundle as the limit of the sequence of gauge transformations
defined by the flow. In the case that the initial bundle is stable and has
stable image, one may apply the basic fact that such a map is always an
isomorphism. In general, the idea in \cite{DW1} is simply to apply this
argument to the first factor of the associated graded object (which is
stable) and then perform an induction. The image of the first factor will be
stable because of the result in Section $5$ about the type of the limiting
sheaf. The difficulty with this method is in proving that the limiting map
is in fact non-zero. This follows directly from Donaldson's proof in the
case of a single subsheaf, but it is more complicated to construct such a
map on the entire filtration. The authors of \cite{DW1} avoid applying
Donaldson's method directly by appealing to a complex analytic argument
involving analytic extension see also \cite{BU2}. Arguing in this fashion
makes the induction rather easier.\ However, this requires the complement of
the singular set to have strictly pseudo-concave boundary, which is true in
the case of surfaces, but is not guaranteed in higher dimensions.

Therefore we give a proof of a slightly more differential geometric
character. Namely, in the case that the filtration is given by subbundles,
we follow the argument of Donaldson, which goes through with modest
corrections in higher dimensions, and does indeed suffice to complete the
induction alluded to. In the general case, we must again appeal to a
resolution of singularities of the filtration and apply the previous
strategy to the pullback bundle over the composition of blowups $\tilde{X}$.
The problem one encounters with this approach is that the induction breaks
down due to the appearance of second fundamental forms of each piece of the
filtration, which are not bounded in $L^{\infty }$ with respect to the
degenerate metric $\pi ^{\ast }\omega $. To rectify this, we apply the
degenerate flow of Section $6$ for some fixed non-zero time $t$ to each
element of the sequence of connections, and this new sequence does have the
desired bound. This is due to the key observation of Bando and Siu that the
Sobolev constant of $\tilde{X}$ with respect to the metrics $\omega
_{\varepsilon }$ is bounded away from zero. A theorem of Cheng and Li then
implies uniform control over subsolutions to the heat equation, which is
sufficient to understand the degenerate flow. One then has to show that the
limit obtained from this new sequence of connections is independent of $t$
and is the correct one. This section is an expanded and slightly modified
account of an argument contained in the unpublished preprint \cite{DW3}.

We conclude the introduction with some general comments. First of all, as
pointed out in \cite{DW1}, the proof of Theorem \ref{THM:MainTheorem} is
essentially independent of the flow, and one obtains a similar theorem by
restricting to sequences of connections which are minimising with respect to
certain Hermitian-Yang-Mills type functionals. Indeed, the statement appears
explicitly as Theorem \ref{Thm12}. Secondly, one expects that there should
be a relationship between the two singular sets $Z_{\limfunc{alg}}$ and $Z_{%
\limfunc{an}}$. Namely, in the best case $Z_{\limfunc{alg}}$ should be
exactly the set of points where bubbling occurs. One always has containment $%
Z_{\limfunc{alg}}\subset Z_{\limfunc{an}}$, and in the separate article \cite%
{DW2} Daskalopoulos and Wentworth have shown that in the surface case
equality does in fact hold. We hope to be able to clarify this issue in
higher dimensions in a future paper.

Finally, the author is aware of a recent series of preprints \cite{J1},\cite%
{J2},\cite{J3} by Adam Jacob which collectively give a proof of Theorem \ref%
{THM:MainTheorem} using different methods.

\textbf{Acknowledgement}

The author would like to thank his thesis advisor, Richard Wentworth, who
suggested this problem and gave him considerable help in solving it. He also
owes a great debt to the anonymous referee, who has given comments of
exceptional detail through three separate readings, pointed out errors,
fixed numerous typos, and whose comments have in general improved the
exposition considerably. The author was also partially supported by $NSF$
grant number 1037094 while working on this manuscript. Finally, the author
would like to thank Xuwen Chen for occasional discussions about certain
technical points in the paper.

\section{Preliminary Remarks}

\subsection{Subsheaves of Holomorphic Bundles and the $HNS$ Filtration}

We now recall some basic sheaf theory. All of this material may be found in 
\cite{KOB}. As stated in the introduction, the main obstacle we will face is
that we must consider arbitrary subsheaves of a holomorphic vector bundle.
Throughout, $X$ will be a compact K\"{a}hler manifold (unless otherwise
stated) with K\"{a}hler form $\omega $, $E$ a holomorphic vector bundle, and 
$S\subset E$ a subsheaf.

Recall that an analytic sheaf $\mathcal{E}$ on $X$ is called torsion free if
the natural map $\mathcal{E}\rightarrow $ $\mathcal{E}^{\ast \ast }$ is
injective. We call $\mathcal{E}$ reflexive if this map is an isomorphism. Of
vital importance is the fact that a torsion free sheaf is \textquotedblleft
almost a vector bundle\textquotedblright\ in the following sense. For $%
\mathcal{E}$ a sheaf on $X$ recall that its singular set is $\limfunc{Sing}(%
\mathcal{E})=\{x\in X\mid \mathcal{E}_{x}$ $is$ $not$ $free\}$. Here $%
\mathcal{E}_{x}$ is the stalk of $\mathcal{E}$ over $x$. In other words $%
\limfunc{Sing}(\mathcal{E})$ is the set of points where $\mathcal{E}$ fails
to be locally free, i.e., a vector bundle. The set $\limfunc{Sing}(\mathcal{E%
})$ is a closed complex analytic subvariety of $X$ of codimension at least $%
2 $.

Recall that the saturation of a subsheaf $S\subset E$ is defined by $%
\limfunc{Sat}_{E}(S)=\ker (E\rightarrow Q/\limfunc{Tor}(Q))$ and that $S$ is
a subsheaf of $\limfunc{Sat}_{E}(S)$ with torsion quotient, and the quotient 
$E/\limfunc{Sat}_{E}(S)$ is torsion free. We also have the following lemma
whose proof we omit.

\begin{lemma}
\label{Lemma1}Let $E$ be a holomorphic vector bundle. Suppose $S_{1}\subset
S_{2}\subset $ $E$ are subsheaves with $S_{2}/S_{1}$ torsion. Then $\limfunc{%
Sat}_{E}(S_{1})=\limfunc{Sat}_{E}(S_{2})$.
\end{lemma}

The $\omega $-\textbf{slope} of a torsion free sheaf $\mathcal{E}$ on $X$ is
defined by:%
\begin{equation*}
\mu _{\omega }(\mathcal{E)=}\frac{\deg _{\omega }(\mathcal{E)}}{\limfunc{rk}(%
\mathcal{E)}}=\frac{1}{\limfunc{rk}(\mathcal{E)}}\int_{X}c_{1}(\mathcal{%
E)\wedge \omega }^{n-1}.
\end{equation*}%
Note that the right hand side is well defined independently of the
representative for $c_{1}(\mathcal{E)}$ since $\omega $ is closed.
Throughout we will assume that the volume of $X$ with respect to $\omega $
is normalised to be $\frac{2\pi }{(n-1)!}$, where $n=\dim _{%
\mathbb{C}
}X$.

\begin{definition}
\label{Def1}We say that a torsion free sheaf $\mathcal{E}$ is $\omega $-%
\textbf{stable} ($\omega $-\textbf{semistable}) if for all proper subsheaves 
$S\subset \mathcal{E}$, $\mu _{\omega }(S)<$ $\mu _{\omega }(\mathcal{E})$ $%
\left( \mu _{\omega }(S)\leq \mu _{\omega }(\mathcal{E})\right) $.
Equivalently $\mu _{\omega }(Q)>\mu _{\omega }(\mathcal{E})$ $\left( \mu
_{\omega }(Q)\geq \mu _{\omega }(\mathcal{E}\right) )$ for every torsion
free quotient $Q$.
\end{definition}

We have the following important proposition.

\begin{proposition}
\label{Prop2}There is an upper bound on the set of slopes $\mu _{\omega }(S)$
of subsheaves of a torsion free sheaf $\mathcal{E}$, and moreover this upper
bound is realised by some subsheaf $\mathcal{E}_{1}\subset \mathcal{E}$.
Furthermore, we can choose $\mathcal{E}_{1}$ so that for any $S\subset 
\mathcal{E}$, if $\mu _{\omega }(S)=\mu _{\omega }(\mathcal{E}_{1})$ then $%
\limfunc{rk}(S)\leq \limfunc{rk}(\mathcal{E}_{1})$. Moreover such a subsheaf
is unique.
\end{proposition}

For the proof see Kobayashi \cite{KOB}. The sheaf $\mathcal{E}_{1}$ is
called the \textbf{maximal destabilising subsheaf} of $\mathcal{E}$. This
sheaf is also clearly semistable.

\begin{remark}
\label{Rmk1}If $S\subset \mathcal{E}$ is a subsheaf with torsion free
quotient $Q=\mathcal{E}/S$, then $Q^{\ast }\hookrightarrow \mathcal{E}^{\ast
}$ is a subsheaf and $\deg (Q^{\ast })=-\deg (Q)$. By the above proposition $%
\mu _{\omega }(Q^{\ast })$ is bounded from above, so $\mu _{\omega }(Q)$ is
bounded from below.
\end{remark}

\begin{remark}
\label{Rmk2}Note also that the saturation of a sheaf has slope at least as
large as the slope of the original sheaf. Therefore the maximal
destabilising subsheaf is saturated by definition.
\end{remark}

\begin{definition}
\label{Def2}We will write $\mu ^{\max }(\mathcal{E)}$ for the maximal slope
of a subsheaf, and $\mu ^{\min }(\mathcal{E)}$ for the minimal slope of a
torsion free quotient. Clearly we have the equality $\mu ^{\min }(\mathcal{%
E)=-}\mu ^{\max }(\mathcal{E}^{\ast }\mathcal{)}$.
\end{definition}

We now specialise to the case of a holomorphic vector bundle $E$, although
the following all holds also for an arbitrary torsion-free sheaf.

\begin{proposition}
\label{Prop3}There is a filtration:%
\begin{equation*}
0=E_{0}\subset E_{1}\subset \cdots \subset E_{l}=E
\end{equation*}%
such that the quotients $Q_{i}=E_{i}/E_{i-1}$ are torsion free and
semistable, and $\mu _{\omega }(Q_{i+1})<\mu _{\omega }(Q_{i})$.
Furthermore, the associated graded object: $Gr_{\omega }^{HN}(E)={\Huge %
\oplus }_{i}Q_{i}$, is uniquely determined by the isomorphism class of $E$
and is called the \textbf{Harder-Narasimhan filtration}. The
Harder-Narasimhan filtration is unique.
\end{proposition}

In the sequel we will usually abbreviate this as the $HN$ filtration, and we
will write $\mathbb{F}_{i}^{HN}(E)$ for the $i^{th}$ piece of the
filtration. The previous proposition follows from Proposition $2$. The
maximal destabilising subsheaf is $\mathbb{F}_{1}^{HN}(E)$. Then consider
the quotient $E/\mathbb{F}_{1}^{HN}(E)$ and its maximal destabilising
subshseaf. Define $\mathbb{F}_{2}^{HN}(E)$ to be the pre-image of this
subsheaf under the natural projection. Iterating this process gives the
stated filtration, and one easily checks that it has the desired properties.

Another invariant of the isomorphism class of $E$ is the collection of all
slopes of the quotients $Q_{i}$.

\begin{definition}
\label{Def3}Let $E$ have rank $K$. Then we form a $K$-tuple 
\begin{equation*}
\mu (E)=(\mu (Q_{1}),\cdots ,\mu (Q_{1}),\cdots ,\mu (Q_{i}),\cdots ,\mu
(Q_{i}),\cdots \mu (Q_{l}),\cdots \mu (Q_{l}))
\end{equation*}%
where $\mu (Q_{i})$ is repeated $\limfunc{rk}(Q_{i})$ times. Then $\mu (E)$
is called the\textbf{\ Harder-Narasimhan} (or $HN$) \textbf{type} of $E$.
\end{definition}

We will also need a result describing the $HN$ filtration of $E$ in terms of
the $HN$ filtration of a subsheaf $S$ and its quotient $Q$. The following
lemma and its corollary are elementary and we omit the proofs.

\begin{proposition}
\label{Prop4}Let $0\rightarrow S\rightarrow E\rightarrow Q\rightarrow 0$ be
an exact sequence of torsion free sheaves with $E$ a holomorphic vector
bundle such that $\mu ^{\min }(S)>\mu ^{\max }(Q)$. Then the $HN$ filtration
of $E$ is given by: 
\begin{equation*}
0\subset \mathbb{F}_{1}^{HN}(S)\subset \cdots \subset \mathbb{F}%
_{k}^{HN}(S)=S\subset \mathbb{F}_{k+1}^{HN}(E)\subset \cdots \subset \mathbb{%
F}_{l}^{HN}(E)=E,
\end{equation*}%
where $\mathbb{F}_{k+i}^{HN}(E)=\ker (E\rightarrow Q/\mathbb{F}_{i}^{HN}(Q))$%
, for $i=0,1,\cdots ,l-k$. In particular, this means that $Q_{i}=\mathbb{F}%
_{k+i}^{HN}(E)/\mathbb{F}_{k+i-1}^{HN}(E)=\mathbb{F}_{i}^{HN}(Q)$ and
therefore $Gr^{HN}(E)=Gr^{HN}(S)\oplus Gr^{HN}(Q)$.
\end{proposition}

\begin{corollary}
\label{Cor1}Suppose that $0\subset E_{1}\subset \cdots \subset
E_{l-1}\subset E_{l}=E$ is a filtration of $E$ by subbundles, and suppose
that for each $i$, $\mu ^{\min }(E_{i})>\mu ^{\max }(E/E_{i})$. Then the
Harder-Narasimhan filtration of $E$ is given by:%
\begin{eqnarray*}
0 &\subset &\mathbb{F}_{1}^{HN}(E_{1})\subset \cdots \subset \mathbb{F}%
_{k_{1}}^{HN}(E_{1})=E_{1}\subset \cdots \subset \mathbb{F}_{k_{1}+\cdots
+k_{l-1}}^{HN}(E_{l-1})=E_{l-1} \\
&\subset &\mathbb{F}_{k_{1}+\cdots +k_{l-1}+1}^{HN}(E)\subset \cdots \subset 
\mathbb{F}_{k_{1}+\cdots +k_{l}}^{HN}(E)=E.
\end{eqnarray*}
\end{corollary}

Now we will define the double filtration that appears in the statement of
the Main Theorem. Its existence follows from the existence of the $HN$
filtration and the following proposition.

\begin{proposition}
\label{Prop5}Let $Q$ be a semi-stable torsion free sheaf on $X$. Then there
is a filtration:%
\begin{equation*}
0\subset F_{1}\subset \cdots \subset F_{l}=Q
\end{equation*}%
such that $F_{i}/F_{i-1}$ is stable and torsion-free. Also, for each $i$ we
have $\mu \left( F_{i}/F_{i-1}\right) =\mu (Q)$. The associated graded
object:%
\begin{equation*}
Gr_{\omega }^{S}(Q)={\Huge \oplus }_{i}F_{i}/F_{i-1}
\end{equation*}%
is uniquely determined by the isomorphism class of $Q$, though the
filtration itself is not. Such a filtration is called a \textbf{Seshadri
filtration} of $Q$.
\end{proposition}

\begin{proposition}
\label{Prop6}Let $E$ be a holomorphic vector bundle on $X$. Then there is a
double filtration $\left\{ E_{i,j}\right\} $ with the following properties.
If the $HN$ filtration is given by:%
\begin{equation*}
0\subset E_{1}\subset \cdots \subset E_{l-1}\subset E_{l}=E,
\end{equation*}%
then $E_{i-1}=E_{i,0}\subset E_{i,1}\subset \cdots \subset E_{i,l_{i}}=E_{i}$%
, where the successive quotients $Q_{i,j}=E_{i,j}/E_{i,j-1}$ are stable and
torsion-free. Furthermore: $\mu _{\omega }(Q_{i,j})=\mu _{\omega
}(Q_{i,j+1}) $ for $j>0$, and $\mu _{\omega }(Q_{i,j})>\mu _{\omega
}(Q_{i+1,j})$. The associated graded object $Gr_{\omega }^{HNS}(E)={\Huge %
\oplus }_{i}{\Huge \oplus }_{j}Q_{i,j}$ is uniquely determined by the
isomorphism class of $E$. This double filtration is called the \textbf{%
Harder-Narasimhan-Seshadri filtration }(or $HNS$ filtration) of $E$.
\end{proposition}

\bigskip Similarly, we can define an $K$-tuple:%
\begin{equation*}
\mu =\left( \mu (Q_{1,1}),\cdots ,\mu (Q_{i,j}),\cdots ,\mu
(Q_{l,k_{l}})\right)
\end{equation*}%
where each $\mu (Q_{i,j})$ is repeated according to $\limfunc{rk}(Q_{i,j})$.
Note that this vector is exactly the same as the Harder-Narasimhan type of $%
E $ (the slopes of a Seshadri filtration are all equal). Since each of the
quotients $Q_{i,j}$ is torsion-free, $\limfunc{Sing}(Q_{ij})$ lies in
codimension at least $2$. We will write:%
\begin{equation*}
Z_{\limfunc{alg}}={\Huge \cup }_{i,j}\limfunc{Sing}(E_{i,j})\cup \limfunc{%
Sing}(Q_{i,j}).
\end{equation*}%
This is a complex analytic subset of codimension at least two, and
corresponds exactly to the set of points at which the $HNS$ filtration fails
to be given by subbundles. We will refer to it as the algebraic singular set
of the filtration.

\subsection{The Yang-Mills Functional and Uhlenbeck Compactness}

Recall that for $E\rightarrow X$ a complex vector bundle, the set of
holomorphic structures on $E$ may be identified with the set of operators $%
\bar{\partial}_{E}$ satisfying the Leibniz rule and the integrability
condition $\bar{\partial}_{E}\circ \bar{\partial}_{E}=0$. When we wish to
make the holomorphic structure explicit we will sometimes write $(E,\bar{%
\partial}_{E})$.

In general we will represent a connection either abstractly by its covariant
derivative $\nabla _{A}$ or in local coordinates by its connection $1$-form $%
A$. We will be careless about this distinction and use whichever notation is
more convenient. We will write $\bar{\partial}_{A}$ and $\partial _{A}$ for
the $(0,1)$ and $(1,0)$ parts of $\nabla _{A}$ respectively. If $(E,\bar{%
\partial}_{E})$ is equipped with a smooth hermitian metric $h$, then there
is a unique $h$-unitary connection $\nabla _{A}$ on $E$ called the \textbf{%
Chern connection }that satisfies $\bar{\partial}_{A}=$ $\bar{\partial}_{E}$.
More specifically the local form of this connection in terms of $h$ is: $A=%
\bar{h}^{-1}\partial \bar{h}$, with curvature $F_{A}=\bar{\partial}\left( 
\bar{h}^{-1}\partial \bar{h}\right) $. Sometimes we will denote this
connection by $\nabla _{A}=(\bar{\partial}_{E},h)$. Conversely, if we have
in hand a unitary connection $\nabla _{A}$ whose curvature $F_{A}=\nabla
_{A}\circ \nabla _{A}$ is of type $(1,1)$ (i.e. $F_{A}^{0,2}=0$), then $\bar{%
\partial}_{A}$ defines a holomorphic structure on $E$ by the
Newlander-Nirenberg theorem, and $\nabla _{A}=(\bar{\partial}_{A},h)$.

Let $\mathcal{A}_{h}$ denote the space of $h$-unitary connections $\nabla
_{A}$ on $E$, and write $\mathcal{A}_{h}^{1,1}$ for the subset consisting of
those with $F_{A}^{0,2}=0$. The above discussion translates to the statement
that there is a bijection between $\mathcal{A}_{h}^{1,1}$ and the space $%
\mathcal{A}_{\limfunc{hol}}$ of integrable $\bar{\partial}_{E}$ operators.
We will write $\mathcal{G}$ for the set of unitary gauge transformations of $%
E$. The set $\mathcal{G}$ is a bundle of groups whose fibres are copies of $%
U(n)$, and $\mathcal{G}$ acts on $\mathcal{A}_{h}$ by the usual conjugation $%
g\cdot \nabla _{A}=g^{-1}\circ \nabla _{A}\circ g$. Moreover this induces an
action on $F_{A}$, which is also by conjugation, so the subspace $\mathcal{A}%
_{h}^{1,1}$ is preserved. We will write:%
\begin{equation*}
\mathcal{B}_{h}=\mathcal{A}_{h}/\mathcal{G}\ ,\ \mathcal{B}_{h}^{1,1}=%
\mathcal{A}_{h}^{1,1}/\mathcal{G}
\end{equation*}%
for the quotients.

Finally there is also an action of the full complex gauge group $\mathcal{G}%
^{%
\mathbb{C}
}$ (the set of all complex gauge transformations of $E$) on $\mathcal{A}_{%
\limfunc{hol}}$ again by conjugation, i.e. $g\cdot \bar{\partial}%
_{E}=g^{-1}\circ \bar{\partial}_{E}\circ g$. The set of isomorphism classes
of holomorphic structures on $E$ is precisely the quotient $\mathcal{A}_{%
\limfunc{hol}}/\mathcal{G}^{%
\mathbb{C}
}$, and via the bijection $\mathcal{A}_{\limfunc{hol}}\simeq \mathcal{A}%
_{h}^{1,1}$ we see that $\mathcal{G}^{%
\mathbb{C}
}$ also acts on $\mathcal{A}_{h}^{1,1}$, extending the action of $\mathcal{G}
$. Moreover, $\mathcal{G}^{%
\mathbb{C}
}$ also acts on the space of hermitian metrics via $h\mapsto g\cdot h$ where 
$g\cdot h(s_{1},s_{2})=h(g(s_{1}),g(s_{2}))$. In matrix form this reads $%
g\cdot h=\bar{g}^{T}hg$.

Now, starting from a holomorphic bundle $E$ with hermitian metric $h$ and
Chern connection $(\bar{\partial}_{E},h)$, we may use a complex gauge
transformation to perturb this connection in two different ways. We may
either let $g$ act on $\bar{\partial}_{E}$ or on $h$. If we write $g^{\ast }$
for the adjoint of $g$ with respect to $h$, then $g\cdot
h(s_{1},s_{2})=h(g^{\ast }g(s_{1}),s_{2})$. If we set $k=g^{\ast }g$, then
the connection corresponding to $h$ and $g\cdot h$ are related by:%
\begin{equation*}
\bar{\partial}_{h}=\bar{\partial}_{g\cdot h}\text{ }and\text{ }\partial
_{g\cdot h}=k\circ \partial _{h}\circ k^{-1}.
\end{equation*}%
Now note that the action of a complex gauge transformation $g$ on a
connection $\nabla _{A}$ is 
\begin{equation*}
g\cdot \nabla _{A}=g^{\ast }\circ \partial _{A}\circ (g^{\ast
})^{-1}+g^{-1}\circ \bar{\partial}_{A}\circ g,
\end{equation*}%
so $g\circ \nabla _{g\cdot A}\circ g^{-1}=k\circ \partial _{A}\circ k^{-1}+%
\bar{\partial}_{A}=(\bar{\partial}_{E},g\cdot h)$ or%
\begin{equation*}
\nabla _{g\cdot A}=(g\cdot \bar{\partial}_{E},h)=g^{-1}\circ (\bar{\partial}%
_{E},g\cdot h)\circ g.
\end{equation*}%
Taking the square of this formula also gives the relation between the
respective curvatures:%
\begin{equation*}
F_{(g\cdot \bar{\partial}_{E},h)}=g^{-1}\circ F_{(\bar{\partial}_{E},g\cdot
h)}\circ g.
\end{equation*}

If we denote by $\mathfrak{u}((E,h))\subset End(E)$ the subbundle of
skew-hermitian endomorphisms, then for a section $\sigma $ of $\mathfrak{u}%
(E)$, we will write $\left\vert \sigma \right\vert $ for its pointwise norm.
This is defined as usual by%
\begin{equation*}
\left\vert \sigma \right\vert =\left( \sum_{i=1}^{K}\left\vert \lambda
_{i}\right\vert ^{2}\right) ^{\frac{1}{2}}
\end{equation*}%
where the $\lambda _{i}$ are the eigenvalues of $\sigma $ at a given point
and $K$ is the rank of $E$. Now we may define the \textbf{Yang-Mills
functional }($YM$ functional) by:%
\begin{equation*}
YM(\nabla _{A})=\int_{X}\left\vert F_{A}\right\vert ^{2}dvol.
\end{equation*}%
If we assume that $X$ is K\"{a}hler, we have:%
\begin{equation*}
YM(\nabla _{A})=\int_{X}\left\vert F_{A}\right\vert ^{2}\frac{\omega ^{n}}{n!%
}.
\end{equation*}%
This functional is gauge invariant and so defines a map $YM:\mathcal{B}_{h}%
\mathcal{\rightarrow 
\mathbb{R}
}$. Its critical points are the so called \textbf{Yang-Mills connections }%
and satisfy the Euler-Lagrange equations for $YM$: $d_{A}^{\ast }F_{A}=0$,
where $d_{A}$ is the covariant derivative induced on $End(E)$ valued $2$%
-forms by $\nabla _{A}$. If $\nabla _{A}\in $ $\mathcal{A}_{h}^{1,1}$ then
we may also define the \textbf{Hermitian-Yang-Mills functional}:%
\begin{equation*}
HYM(\nabla _{A})=\int_{X}\left\vert \Lambda _{\omega }F_{A}\right\vert ^{2}%
\frac{\omega ^{n}}{n!},
\end{equation*}%
where $\Lambda _{\omega }$ is, as usual the adjoint of the Lefschetz
operator, (which is given by wedging with the K\"{a}hler form). For a $(1,1)$
form $G=\sum G_{i,j}dz_{i}\wedge d\bar{z}_{j}$ this can be written
explicitly in coordinates as%
\begin{equation*}
\Lambda _{\omega }G=-2\sqrt{-1}\left( g^{ij}G_{ij}\right)
\end{equation*}%
where $g^{ij}$ denotes the inverse of the metric. The quantity $\Lambda
_{\omega }F_{A}$ is called the \textbf{Hermitian-Einstein tensor }of $A$.
Again $HYM$ is gauge invariant and so defines a functional $HYM:\mathcal{B}%
_{h}^{1,1}\rightarrow 
\mathbb{R}
$. Critical points of the functional satisfy the Euler-Lagrange equations: $%
d_{A}\Lambda _{\omega }F_{A}=0$. On the other hand, just as in the preceding
discussion, we may regard the holomorphic stucture as being fixed and
consider the space of $(1,1)$ connections as being the set of pairs $(\bar{%
\partial}_{E},h)$ where $h$ varies over all hermitian metrics. We may
therefore think of $HYM$ as a functional $HYM(h)=HYM(\bar{\partial}_{E},h)$
on the space of hermitian metrics on $E$. A critical metric of $HYM$ is
referred to a \textbf{critical hermitian structure} on $(E,\bar{\partial})$.

An important fact that we will use is that when $X$ is compact, there is a
relation between the two functionals $YM$ and $HYM$. Explicitly:%
\begin{equation*}
YM(\nabla _{A})=HYM(\nabla _{A})+\frac{4\pi ^{2}}{\left( n-2\right) !}%
\int_{X}\left( 2c_{2}(E)-c_{1}^{2}(E)\right) \wedge \omega ^{n-2}
\end{equation*}%
for any $A\in \mathcal{A}_{h}^{1,1}$. The second term depends only on the
topology of $E$ and the form $\omega $, so $YM$ and $HYM$ have the same
critical points on $\mathcal{A}_{h}^{1,1}$. Furthermore, $\nabla _{A}$ is a
critical point of $YM$ and $HYM$, if and only if $h$ is a critical hermitian
structure for the holomorphic stucture on $E$ given by $A$.

For a Yang-Mills connection we have the following proposition.

\begin{proposition}
\label{Prop7}Let $\nabla _{A}\in \mathcal{A}_{h}^{1,1}$ be a Yang-Mills
connection on an hermitian vector bundle $(E,h)$ over a K\"{a}hler manifold $%
X$. Then $\nabla _{A}=\oplus _{i=1}^{l}\nabla _{A_{i}}$ where $E=\oplus
_{i=1}^{l}Q_{i}$ is an orthogonal splitting of $E$, and where $\sqrt{-1}%
\Lambda _{\omega }F_{A_{i}}=\lambda _{i}\mathbf{Id}_{Q_{i}}$, where $\lambda
_{i}$ are constant. If $X$ is compact, then $\lambda _{i}=\mu (Q_{i})$.
\end{proposition}

The proof is simply the observation (stated above) that the
Hermitian-Einstein tensor of a Yang-Mills connection is covariantly
constant, and so has constant eigenvalues and eigenspaces of constant rank.
Therefore $E$ breaks up into a direct sum of the eigenspaces of this
operator.

\begin{definition}
\label{Def4}Let $E\rightarrow (X,\omega )$ be a holomorphic bundle.\ Then a
connection $\nabla _{A}$ such that there exists a constant $\lambda $ with:%
\begin{equation*}
\sqrt{-1}\Lambda _{\omega }F_{A}=\lambda \mathbf{Id}_{E}
\end{equation*}%
is called an \textbf{Hermitian-Einstein connection}. If $A$ is the Chern
connection of $(\bar{\partial}_{E},h)$ for some hermitian metric $h$, then $%
h $ is called an \textbf{Hermitian-Einstein metric}.
\end{definition}

The existence of such a metric is related to stability properties of $E$.
This is the \textbf{Kobayashi-Hitchin correspondence} (or
Donaldson-Uhlenbeck-Yau theorem).

\begin{theorem}
\label{Thm2}A holomorphic vector bundle $E$ on a compact K\"{a}hler manifold 
$(X,\omega )$, admits an Hermitian-Einstein metric if and only if $E$ is
polystable, i.e. $E$ splits holomorphically into a direct sum of $\omega $%
-stable bundles of the same $\omega $-slope $\mu _{\omega }(E)$. Such a
metric is unique up to a positive constant.
\end{theorem}

For the proof in the case of projective surfaces and projective complex
manifolds see \cite{DO1} and \cite{DO2} respectively. For the proof in the
general case see \cite{UY}. From the $HYM$ equations it is clear that an
Hermitian-Einstein connection is Hermitian-Yang-Mills (and so Yang-Mills).

\begin{remark}
\label{Rmk3}Note that if $E$ is holomorphic and $\nabla _{A}=(\bar{\partial}%
_{E},h)$ for some hermitian metric $h$, then the same argument shows that $%
(E,h)=\oplus _{i=1}^{l}(Q_{i},h_{i})$ where the $h_{i}$ are
Hermitian-Einstein metrics and the splitting is orthogonal with respect to $%
h $. Since the splitting is preserved by the Chern connection $\nabla _{A}$,
it is also holomorphic with respect to the holomorphic structure on $E$
given by $\bar{\partial}_{E}$.
\end{remark}

We now give the statement of the general Uhlenbeck compactness theorem.
Although we will be primarily concerned with the theorem as it applies to
the Yang-Mills flow of the next section, the proof of the main theorem in
Section $7$ will also rely on this more general statement.

\begin{theorem}
\label{Thm3}Let $X$ be a K\"{a}hler manifold (not necessarily compact) and $%
E\rightarrow X$ a hermitian vector bundle with metric $h$. Fix any $p>n$.
Let $\nabla _{A_{j}}$ be a sequence of integrable, unitary connections on $E$
such that $\left\Vert F_{A_{j}}\right\Vert _{L^{2}(X)}$ and $\left\Vert
\Lambda _{\omega }F_{A_{j}}\right\Vert _{L^{\infty }(X)}$ are uniformly
bounded. Then there is a subsequence (still denoted $A_{j}$), a closed
subset $Z_{\limfunc{an}}\subset X$ with Hausdorff codimension at least $4$,
and a smooth hermitian vector bundle $(E_{\infty },h_{\infty })$ defined on
the complement $X-Z_{\limfunc{an}}$ with a finite action Yang-Mills
connection $\nabla _{A_{\infty }}$ on $E_{\infty }$, such that $\nabla
_{A_{j}\mid X-Z_{\limfunc{an}}}$ is gauge equivalent to a sequence of
connections that converges to $\nabla _{A_{\infty }}$ weakly in $%
L_{1,loc}^{p}(X-Z_{\limfunc{an}})$.
\end{theorem}

The statement of this version of Uhlenbeck compactness may be found for
example in Uhlenbeck-Yau (\cite{UY} Theorem $5.2$). The proof is essentially
contained in \cite{U2} and the statement about the singular set follows from
the arguments in \cite{NA}. We will call such a limit $\nabla _{A_{\infty }}$
an \textbf{Uhlenbeck limit}. Furthermore, we have the following crucial
extension of this theorem due essentially to Bando and Siu.

\begin{corollary}
\label{Cor2}If in addition to the assumptions in the previous theorem, we
also require that:%
\begin{equation*}
\left\Vert d_{A_{j}}\Lambda _{\omega }F_{Aj}\right\Vert
_{L^{2}(X)}\longrightarrow 0,
\end{equation*}%
then any Uhlenbeck limit $\nabla _{A_{\infty }}$ is Yang-Mills. On $X-Z_{%
\limfunc{an}}$ we therefore have a holomorphic, orthogonal, splitting:%
\begin{equation*}
(E_{\infty },h_{\infty },\nabla _{A_{\infty }})={\Huge \oplus }%
_{i=1}^{l}(Q_{\infty ,i},h_{\infty ,i},\nabla _{A_{\infty ,i}})
\end{equation*}%
Moreover $E_{\infty }$ extends to a reflexive sheaf (still denoted $%
E_{\infty }$) on all of $X$.
\end{corollary}

Most of the content of this corollary resides in the last statement, which
may be found in \cite{BS} as Corollary $2$. The proof presented there is
based on results in the papers \cite{B} and \cite{SIU}. The statement about
the splitting follows directly from the fact that an Uhlenbeck limit is
Yang-Mills and $\limfunc{Proposition}$ \ref{Prop7}. Therefore it only
remains to prove that the stated condition implies the limit is Yang-Mills.
For a proof of this see for example \cite{DW1}.

We will need the following simple corollary of Uhlenbeck compactness, which
we will use repeatedly.

\begin{corollary}
\label{Cor4}With the same assumptions as in Theorem \ref{Thm3}, $\Lambda
_{\omega }F_{A_{j}}\rightarrow \Lambda _{\omega }F_{A_{\infty }}$ in $%
L^{p}(X-Z_{\limfunc{an}})$ for all $1\leq p<\infty $.
\end{corollary}

For the proof see \cite{DW1}.

In general, if $\mathcal{E}$ is only a reflexive sheaf, Bando and Siu (\cite%
{BS}) defined the notion of an \textbf{admissible hermitian metric}. This is
an hermitian metric $h$ on the locally free part of $\mathcal{F}$ such that:

$\cdot $ $\Lambda _{\omega }F_{h}\in L^{\infty }(X,\omega )$

$\cdot $ $F_{h}\in L^{2}(X,\omega ).$

Corollary \ref{Cor2} says that the limiting metric is an admissible
hermitian metric on the reflexive sheaf $E_{\infty }$ that is a direct sum
of admissible Hermitian-Einstein metrics. We also point out the version of
the Kobayashi-Hitchin correspondence for reflexive sheaves, due to Bando and
Siu \cite{BS}.

\begin{theorem}
\label{THM Bando-Siu}(Bando-Siu) A reflexive sheaf $\mathcal{E}$ on a
compact K\"{a}hler manifold $(X,\omega )$ admits an admissible
Hermitian-Einstein metric if and only if it is polystable. Such a metric is
unique up to a positive constant.
\end{theorem}

Note that this theorem says the $(Gr_{\omega }^{HNS}(E))^{\ast \ast }$
carries an admissible Yang-Mills connection (where admissible has the same
meaning for connections), which is unique up to gauge.

\subsection{Weakly Holomorphic Projections/Second Fundamental Forms}

Let $S\subset E$ be a subsheaf with quotient $Q$. Then away from $\limfunc{%
Sing}(S)\cup \limfunc{Sing}(Q)$, $S$ is a subbundle. If we fix an hermitian
metric $h$ on $E$, then we may think of $S$ as a direct summand away from
the singular set, and there is a corresponding smooth projection operator $%
\pi :E\rightarrow S$ depending on $h$. The condition of being a holomorphic
subbundle almost everywhere can be shown to be equivalent to the condition: $%
\left( \mathbf{Id}_{E}-\pi \right) \bar{\partial}_{E}\pi =0$. Since $\pi $
is a projection operator we also have $\pi ^{2}=\pi =\pi ^{\ast }$.
Furthermore it can be shown that $\pi $ extends to an $L_{1}^{2}$ section of 
$\limfunc{End}E$. Conversely it turns out that an operator with these
properties determines a subsheaf.

\begin{definition}
\label{Def5}An element $\pi \in L_{1}^{2}(\limfunc{End}E)$ is called a
weakly holomorphic projection operator if the conditions 
\begin{equation*}
\left( \mathbf{Id}_{E}-\pi \right) \bar{\partial}_{E}\pi =0\text{ }and\text{ 
}\pi _{j}^{2}=\pi _{j}=\pi _{j}^{\ast }\text{ \ \ \ \ \ \ }\ast
\end{equation*}%
hold almost everywhere.
\end{definition}

\begin{theorem}
\label{Thm4}(Uhlenbeck-Yau) A weakly holomorphic projection operator $\pi $
of a holomorphic vector bundle $(E,h)$ with a smooth hermitian metric over a
compact K\"{a}hler manifold $(X,\omega )$ determines a coherent subsheaf of $%
E$. That is, there exists a coherent subsheaf $S$ of $E$ together with a
singular set $V\subset X$ with the following properties:

$\cdot \limfunc{Codim}V\geq 2,$

$\cdot \pi _{\mid X-V}$ is $C^{\infty }$ and satisfies $\ast $,

$\cdot S_{\mid X-V}=$ $\pi _{\mid X-V}(E_{\mid X-V})\hookrightarrow $ $%
E_{\mid X-V}$ is a holomorphic subbundle.
\end{theorem}

The proof of this theorem is contained in \cite{UY}. From here on out we
will identify a subsheaf with its weakly holomorphic holomorphic projection.

If $S\subset E$ is a subsheaf, then away from $\limfunc{Sing}(S)\cup 
\limfunc{Sing}(Q)$ there is an orthogonal splitting $E=S\oplus Q$. In
general we may write the Chern connection $\nabla _{(\bar{\partial}_{E},h)}$
connection on $E$ as:%
\begin{equation*}
\nabla _{(\bar{\partial}_{E},h)}=%
\begin{pmatrix}
\nabla _{\left( \bar{\partial}_{S},h_{S}\right) } & \beta \\ 
-\beta ^{\ast } & \nabla _{\left( \bar{\partial}_{Q},h_{Q}\right) }%
\end{pmatrix}%
\end{equation*}%
where $\nabla _{\left( \bar{\partial}_{S},h_{S}\right) }$ and $\nabla
_{\left( \bar{\partial}_{Q},h_{Q}\right) }$ are the induced Chern
connections on $S$ and $Q$ respectively, and $\beta $ is the second
fundamental form. Recall that $\beta \in \Omega ^{0,1}(Hom(Q,S))$. More
specifically, in terms of the projection operator, we have$\ \bar{\partial}%
_{E}\pi =\beta $ and$\ \partial _{E}\pi =\beta ^{\ast }$. Also $\beta $
extends to an $L^{2}$ section of $\Omega ^{0,1}(Hom(Q,S))$ everywhere as $%
\bar{\partial}_{E}\pi $ since $\pi $ is $L_{1}^{2}$. We also have the
following well-known formula for the degree of a subsheaf in terms of its
weakly holomorphic projection.

\begin{theorem}
\label{Thm5}(Chern-Weil Formula) Let $S\subset E$ be a saturated subsheaf of
a holomorphic vector bundle with hermitian metric $h$, and $\pi $ the
associated weakly holomorphic projection. Let $\bar{\partial}_{E}$ denote
the holomorphic structure on $E$. Then we have:%
\begin{equation*}
\deg S=\frac{1}{2\pi n}\int_{X}\limfunc{Tr}\left( \sqrt{-1}\Lambda _{\omega
}F_{(\bar{\partial}_{E},h)}\pi \right) \omega ^{n}-\frac{1}{2\pi n}%
\int_{X}\left\vert \beta \right\vert ^{2}\omega ^{n}
\end{equation*}
\end{theorem}

The statement of this theorem as well as a sketch of the proof may be found
in \cite{SI}. This formula will also follow as a special case of our
discussion in Section $4$.

Clearly any sequence $\pi _{j}$ of such projection operators is uniformly
bounded in $L^{\infty }(X)$. As an immediate corollary of the Chern-Weil
formula we have the following.

\begin{corollary}
\label{Cor5}A sequence $\pi _{j}$ of weakly holomorphic projection operators
such that $\deg \pi _{j}$ is bounded from below is uniformly bounded in $%
L_{1}^{2}$. In particular, if $\deg \pi _{j}$ is constant then $\pi _{j}$ is
bounded in $L_{1}^{2}$.
\end{corollary}

Now suppose $\nabla _{A_{0}\text{ }}$is a reference connection, $g_{j}\in 
\mathcal{G}^{%
\mathbb{C}
}$ is a sequence of complex gauge transformations, and $\nabla _{A_{j}}$ is
the sequence of integrable unitary connections on an hermitian vector bundle 
$(E,h)$ given by $\nabla _{A_{j}}=g_{j}\cdot \nabla _{A_{0}}$, and assume as
before that $\Lambda _{\omega }F_{A_{j}}$ is uniformly bounded in $L^{\infty
}$. Let $S\subset E$ be a subbundle with quotient $Q$. We have a sequence of
projection operators $\pi _{j}$ given by orthogonal projection onto $%
g_{j}(S) $ (with respect to the metric $h$) from $E$ to holomorphic
subbundles $S_{j}$ (whose holomorphic structures are induced by the
connections $\nabla _{A_{j}} $) smoothly isomorphic to $S$. We will denote
by $Q_{j}$ the corresponding quotients. Each of these holomorphic subbundles
has a second fundamental form which we will write as $\beta _{j}$. Assume
that the $\beta _{j}$ are also uniformly bounded in $L^{2}$ (this will later
be a consequence of our hypotheses). Then with all of the above understood,
we have the following result.

\begin{lemma}
\label{Lemma4}For any $1\leq p<\infty $, the $\beta _{j}$ are bounded in $%
L_{1,loc}^{p}(X-Z_{\limfunc{an}})$, uniformly for all $j$. In particular the 
$\beta _{j}$ are uniformly bounded on compact subsets of $X-Z_{\limfunc{an}}$%
.
\end{lemma}

The proof is the same as in \cite{DW1} Section 2.2.

\section{The Yang-Mills Flow and Basic Properties}

\subsection{The Flow Equations/Lower Bound for the $HN$ Type of the Limit}

As stated in the introduction, although many of our arguments are valid for
minimising sequences of unitary connections, our primary interest will be in
sequences obtained from the \textbf{Yang-Mills flow}. This is a sequence of
integrable unitary connections $A_{t}$ obtained as solutions of the $L^{2}$%
-gradient flow equations for the $YM$ functional. Explicitly:%
\begin{equation*}
\frac{\partial A_{t}}{\partial t}=-d_{A_{t}}^{\ast }F_{A_{t}},\text{ \ }%
A_{0}\in \mathcal{A}_{h}^{1,1}.
\end{equation*}%
It follows from \cite{DO1} and \cite{SI} that the above equations have a
unique solution in $\mathcal{A}_{h}^{1,1}\times \lbrack 0,\infty )$.
Moreover, the flow preserves complex gauge orbits, that is, $A_{t}$ lies in
the orbit $\mathcal{G}^{%
\mathbb{C}
}\cdot A_{0}$. This may be seen as follows. Instead of solving for the
connection, fix $A_{0}$ so that $\bar{\partial}_{A_{0}}=\bar{\partial}_{E}$,
and consider instead the family of hermitian metrics $h_{t}$ satisfying the 
\textbf{Hermitian-Yang-Mills flow equations}:%
\begin{equation*}
h_{t}^{-1}\frac{\partial h_{t}}{\partial t}=-2\left( \sqrt{-1}\Lambda
_{\omega }F_{h_{t}}-\mu _{\omega }(E)Id_{E}\right) .
\end{equation*}%
In the above, $F_{h_{t}}$ is the curvature of $(\bar{\partial}_{E},h_{t})$.
The Yang-Mills and Hermitian-Yang-Mills flow equations are equivalent up to
gauge. If $A_{t}=g_{t}\cdot A_{0}$ is a solution of the Yang-Mills flow,
then $h_{t}=h_{0}g_{t}^{\ast }g_{t}$ is a solution of the
Hermitian-Yang-Mills flow. Conversely, if $h_{t}=h_{0}k_{t}$ (where $%
h_{0}k_{t}(s_{1},s_{2})=h_{0}(k_{t}s_{1},s_{2})$) for a positive definite
self-adjoint (with respect to $h_{0}$) endomorphism $k_{t}$, then $%
A_{t}=\left( k_{t}\right) ^{\frac{1}{2}}A_{0}$ is real gauge equivalent to a
solution of the Yang-Mills flow. To spell out the equivalence precisely, the
map:%
\begin{equation*}
g_{t}:(E,\bar{\partial}_{E},h_{0}k_{t})\longrightarrow (E,g_{t}(\bar{\partial%
}_{E}),h_{0})
\end{equation*}%
is a biholomorphism and an isometry, where $k_{t}=g_{t}^{\ast }g_{t}$. For a
detailed discussion of this see \cite{WIL} section $3.1$ for details.

\begin{lemma}
\label{Lemm5}Let $A_{t}$ be a solution of the $YM$ flow. Then:

$(1)$ 
\begin{equation*}
\frac{\partial F_{A_{t}}}{\partial t}=-\triangle _{A_{t}}F_{A_{t}}
\end{equation*}%
and therefore,%
\begin{equation*}
\frac{d}{dt}\left\Vert F_{A_{t}}\right\Vert _{L^{2}}^{2}=-2\left\Vert
d_{A_{t}}^{\ast }F_{A_{t}}\right\Vert _{L^{2}}^{2}\leq 0.
\end{equation*}%
Hence, $t\mapsto YM(A_{t})$, and $t\mapsto HYM(A_{t})$ are non-increasing.

$(2)$ $\left\vert \Lambda _{\omega }F_{A_{t}}\right\vert ^{2}$ satisfies%
\begin{equation*}
\frac{\partial \left\vert \Lambda _{\omega }F_{A_{t}}\right\vert ^{2}}{%
\partial t}+\triangle \left\vert \Lambda _{\omega }F_{A_{t}}\right\vert
^{2}=-2\left\vert d_{A_{t}}^{\ast }F_{A_{t}}\right\vert ^{2}\leq 0,
\end{equation*}%
so by the maximum principle $\sup \left\vert \Lambda _{\omega
}F_{A_{t}}\right\vert ^{2}$ is decreasing in $t$.
\end{lemma}

For the proof see \cite{DOKR} Chapter 6.

Now we may apply the Uhlenbeck compactness theorem to a sequence of
connections given by the flow.

\begin{proposition}
\label{Prop8}Let $X$ be a compact K\"{a}hler manifold. Let $A_{0}$ be any
fixed connection, and $A_{t}$ denote its evolution along the flow. Fix $p>n$%
. For any sequence $t_{j}\rightarrow \infty $ there is a subsequence (still
denoted $t_{j}$), a closed subset $Z_{\limfunc{an}}\subset X$ with Hausdorff
codimension at least $4$, and a smooth hermitian vector bundle $(E_{\infty
},h_{\infty })$ defined on the complement $X-Z_{\limfunc{an}}$ with a finite
action Yang-Mills connection $A_{\infty }$ on $E_{\infty }$, such that $%
A_{t_{j}\mid X-Z_{\limfunc{an}}}$ is gauge equivalent to a sequence of
connections that converges to $A_{\infty }$ weakly in $L_{1,loc}^{p}(X-Z_{%
\limfunc{an}})$. Away from $Z_{\limfunc{an}\text{ }}$there is a smooth
splitting: $\left( E_{\infty },A_{\infty },h_{\infty }\right) =\oplus
_{i=l}^{l}\left( Q_{\infty ,i},A_{\infty ,i},h_{\infty ,i}\right) $, where $%
A_{\infty ,i}$ is the induced connection on $Q_{i}$, and $h_{\infty ,,i}$ is
an Hermitian-Einstein metric. Furthermore, $E_{\infty }$ extends over $Z_{%
\limfunc{an}}$ as a reflexive sheaf (still denoted $E_{\infty }$), so that
the metrics $h_{\infty ,i}$ are admissible Hermitian-Einstein metrics on the
extension.
\end{proposition}

\begin{proof}
The functions $\left\Vert F_{A_{t}}\right\Vert _{L^{2}}$ and $\left\Vert
\Lambda _{\omega }F_{A_{t}}\right\Vert _{L^{\infty }}$ are uniformly bounded
by parts $(1)$ and $(2)$ of Lemma $\ref{Lemm5}$ respectively. By \cite{DOKR} 
$\limfunc{Proposition}$ $6.2.14$, $\lim_{t\rightarrow \infty }\left\Vert
\nabla _{A_{t}}\Lambda _{\omega }F_{A_{t}}\right\Vert _{L^{2}}=0$. The
remaining statements follow from Corollary $\ref{Cor2}$.
\end{proof}

Just as before we call $A_{\infty }$ an Uhlenbeck limit of the flow.

\begin{lemma}
\label{Lemma6}If $A_{\infty }$ is an Uhlenbeck limit of $A_{t_{j}}$, then $%
\Lambda _{\omega }F_{A_{j}}\rightarrow \Lambda _{\omega }F_{A_{\infty }}$ in 
$L^{p}(X-Z_{\limfunc{an}})$ for all $1\leq p<\infty $. Moreover, $%
\lim_{t\rightarrow \infty }HYM(A_{t})=HYM(A_{\infty })$.
\end{lemma}

\begin{proof}
The first part is immediate from Corollary $\ref{Cor4}$. The second
statement is immediate from the facts that $t\rightarrow HYM(A_{t})$ is
non-increasing, and $HYM(A_{t_{j}})\mapsto HYM(A_{\infty })$.
\end{proof}

\bigskip The set of all $HN$ types of holomorphic bundles on $X$ has a
partial ordering due to Shatz \cite{SH}. For a pair of $K$-tuples $\mu $ and 
$\lambda $ with $\mu _{1}\geq \mu _{2}\geq \cdots \geq \mu _{K}$ and $%
\lambda _{1}\geq \lambda _{2}\geq \cdots \geq \lambda _{K}$ and $\sum_{i}\mu
_{i}=\sum_{i}\lambda _{i}$, we write%
\begin{equation*}
\mu \leq \lambda \Longleftrightarrow \sum_{j\leq k}\mu _{j}\leq \sum_{j\leq
k}\lambda _{j}\text{ }for\text{ }all\text{ }k=1,\cdots ,K.
\end{equation*}%
This partial ordering was originally used by Shatz to stratify the space of
holomorphic structures on a complex vector bundle.

The first crucial step in \cite{DW1} is to prove that the $HN$ type of an
Uhlenbeck limit is bounded below by the $HN$ type $\mu _{0}$ of $E$. For the
proofs of this and its corollaries, we refer to \cite{DW1} as the proof is
unchanged in the general case.

\begin{proposition}
\label{Prop9}Let $A_{j}$ be a sequence of connections along the $YM$ flow on
a holomorphic vector bundle of rank $K$, with Uhlenbeck limit $A_{\infty }$.
Let $\mu _{0}$ be the $HN$ type of $E$ with holomorphic structure $\bar{%
\partial}_{A_{0}}$. Let $\lambda _{\infty }$ be the $HN$ type of $\bar{%
\partial}_{A_{\infty }}$. Then $\mu _{0}\leq \lambda _{\infty }$.
\end{proposition}

\begin{corollary}
\label{Cor3}Let $\mu =\left( \mu _{1},\cdots ,\mu _{K}\right) $ be the $HN$
type of a rank $K$ holomorphic vector bundle $(E,\bar{\partial}_{E})$ on $X$%
. Then%
\begin{equation*}
\sum_{i=1}^{K}\mu _{i}^{2}\leq \frac{1}{2\pi n}\int_{X}\left\vert \Lambda
_{\omega }F_{A}\right\vert ^{2}\omega ^{n}
\end{equation*}%
and 
\begin{equation*}
\left( \sum_{i=1}^{K}\mu _{i}^{2}\right) ^{\frac{1}{2}}\leq \frac{1}{2\pi n}%
\int_{X}\left\vert \Lambda _{\omega }F_{A}\right\vert \omega ^{n}
\end{equation*}%
for all unitary connections $\nabla _{A}$ in the $\mathcal{G}^{%
\mathbb{C}
}$ orbit of $(E,\bar{\partial}_{E})$.
\end{corollary}

\subsection{Hermitian-Yang-Mills Type Functionals}

The $YM$ and $HYM$ functionals are not sufficient to distinguish different $%
HN$ types in general. In other words there may be multiple connections with
the same $YM$ number, but which induce holomorphic structures with different 
$HN$ types. In this subsection we introduce generalisations of the $HYM$
functional that can be used to distinguish different types. This is only a
technical device, but will be used essentially in Section $5$.

Write $\mathfrak{u}(K)$ for the Lie algebra of the unitary group $U(K)$. Fix
a real number $\alpha \geq 1$. Then for $\mathbf{v}\in \mathfrak{u}(K),$ a
skew hermitian matrix with eigenvalues $\sqrt{-1}\lambda _{1},\cdots ,\sqrt{%
-1}\lambda _{K},$ let $\varphi _{\alpha }(\mathbf{v})=\sum_{i=1}^{K}\left%
\vert \lambda _{i}\right\vert ^{\alpha }$. It can be seen that there is a
family $\varphi _{\alpha ,\rho }$, $0<\rho \leq 1$, of smooth convex $%
\limfunc{Ad}$-invariant functions such that $\varphi _{\alpha ,\rho
}\rightarrow \varphi _{\alpha }$ uniformly on compact subsets of $\mathfrak{u%
}(K)$. By Atiyah-Bott (\cite{AB}), $\limfunc{Proposition}$ $12.16$, $\varphi
_{\alpha }$ is a convex function on $\mathfrak{u}(K)$. Now if $E$ is a
vector bundle of rank $K$ equipped with an hermitian metric, we may consider
a section $\sigma \in \Gamma (X,\mathfrak{u}(E))$ as collection of local
sections $\{\sigma _{\beta }\}$ such that $\sigma _{\beta }=\limfunc{Ad}%
(g_{\beta \gamma })\sigma _{\gamma }$ where $g_{\beta \rho }$ are the
transition functions for $E$. By the $\limfunc{Ad}$-invariance of $\varphi
_{\alpha }$, $\varphi _{\alpha }(\sigma _{\beta })=\varphi _{\alpha }(\sigma
_{\gamma })$, so $\varphi _{\alpha }$ induces a well-defined function $\Phi
_{\alpha }$ on $\mathfrak{u}(E)$. Then for a fixed real number $N$, define:%
\begin{equation*}
HYM_{\alpha ,N}(A)=\int_{X}\Phi _{\alpha }(\Lambda _{\omega }F_{A}+\sqrt{-1}N%
\mathbf{Id}_{E})dvol_{\omega }
\end{equation*}%
and $HYM_{\alpha }(A)=HYM_{\alpha ,0}(A)$. Note that $HYM=HYM_{2}$ is the
usual $HYM$ functional. In the sequel we will write:%
\begin{eqnarray*}
HYM_{\alpha ,N}(\mu ) &=&HYM_{\alpha }(\mu +N)=\frac{2\pi }{\left(
n-1\right) !}\Phi _{\alpha }(\sqrt{-1}\left( \mu +N\right) ), \\
where\text{ }\mu +N &=&\left( \mu _{1}+N,\cdots ,\mu _{K}+N\right)
\end{eqnarray*}%
is identified with the matrix $\limfunc{diag}\left( \mu _{1}+N,\cdots ,\mu
_{K}+N\right) $. Therefore:%
\begin{equation*}
HYM(\mu )=\frac{2\pi }{\left( n-1\right) !}\sum_{i=1}^{K}\mu _{i}^{2}.
\end{equation*}%
We have the following elementary lemma whose proof we omit.

\begin{lemma}
\label{Lemma9}The functional $\mathbf{v\to }\left( \int_{X}\Phi _{\alpha }(%
\mathbf{v})\right) ^{\frac{1}{\alpha }}$ is equivalent to the $L^{\alpha }(%
\mathfrak{u}(E))$ norm.
\end{lemma}

The following three propositions will be crucial in Section $5$. For the
proofs see \cite{DW1}.

\begin{proposition}
\label{Prop10}$(1)$ If $\mu \leq \lambda $, then $\Phi _{\alpha }(\sqrt{-1}%
\mu )\leq \Phi _{\alpha }(\sqrt{-1}\lambda )$ for all $\alpha \geq 1$.

$\ \ \ \ \ \ \ \ \ \ \ \ \ \ \ (2)$ Assume $\mu _{K}\geq 0$ and $\lambda
_{K}\geq 0$. If $\Phi _{\alpha }(\sqrt{-1}\mu )=\Phi _{\alpha }(\sqrt{-1}%
\lambda )$ for all $\alpha $ in some set

$\ \ \ \ \ \ \ \ \ \ \ \ \ \ \ \ \ \ \ A\subset \lbrack 1,\infty )$
possessing a limit point, then $\mu =\lambda $.
\end{proposition}

\begin{proposition}
\label{Prop11}Let $A_{t}$ be a solution of the $YM$ flow. Then for any $%
\alpha \geq 1$ and any $N$, $t\mapsto HYM_{\alpha ,N}(A_{t})$ is
non-increasing.
\end{proposition}

\begin{proposition}
\label{Prop12}Let $A_{\infty }$ be a subsequential Uhlenbeck limit of $A_{t}$
where $A_{t}$ is a solution of the $YM$ flow. Then for all $\alpha \geq 1$, $%
\lim_{t\to\infty }HYM_{\alpha ,N}(A_{t})=HYM_{\alpha ,N}(A_{\infty }) $.
\end{proposition}

\section{Properties of Blowups and Resolution of the $HNS$ Filtration}

In this section we discuss the properties of blowups of complex manifolds
along complex submanifolds that will be used in the subsequent discussion.
Essentially all of this material is standard, but we review it carefully now
because we will need to employ these facts often in the proofs of the main
results.

\subsection{Resolution of Singularities Type Theorems}

The $HNS$ filtration is in general only given by subsheaves, making it
difficult to do analysis. We will therefore need some way of obtaining a
filtration by subbundles, that is, a way of resolving the singularities. In
two dimensions, when the singular set consists of point singularities this
can be done by hand (see \cite{BU1}), but in higher dimensions the only
available tool seems to be the general resolution of singularities theorem
of Hironaka. Specifically:

\begin{theorem}
\label{Thm6}(Resolution of Singularities) Let $X$ be a compact, complex
space (or $\mathbb{C}$-scheme). Then there exists a finite sequence of
blowups with smooth centres:%
\begin{equation*}
\tilde{X}=X_{m}\overset{\pi _{m}}{\longrightarrow }X_{m-1}\longrightarrow
\cdots \longrightarrow X_{1}\overset{\pi _{1}}{\longrightarrow }X_{0}=X
\end{equation*}%
such that $\tilde{X}$ is compact and non-singular (a complex manifold) and
the centre $Y_{j-1}$ of each blowup $\pi _{j}$ is contained in the singular
locus of $X_{j-1}$.
\end{theorem}

For the proof see \cite{H1} and \cite{H2}. What we will actually use is the
following corollary:

\begin{corollary}
\label{Cor6}(Resolution of the Locus of Indeterminacy) Let $X$ and $Y$ be
compact, complex spaces and let $\varphi :X\dashrightarrow Y$ be a rational
(meromorphic) map. Then there exists a compact, complex space $\tilde{X}%
\overset{\pi }{\rightarrow }X$ obtained from $X$ by a sequence of blowups
with smooth centres and a holomorphic map $\psi :\tilde{X}\rightarrow Y$
such that the following diagram commutes:%
\begin{equation*}
\begin{array}{ccc}
\widetilde{X} &  &  \\ 
\downarrow & \searrow &  \\ 
X & \underset{\varphi }{\dashrightarrow } & Y%
\end{array}%
.
\end{equation*}
\end{corollary}

In our case both $X$ and $Y$ (and hence also $\tilde{X}$) will be complex
manifolds. Note that in this case a blowup with \textquotedblleft smooth
centre\textquotedblright\ is the same as the blowup along a complex
submanifold. We will apply the Corollary in the following way.

The $HNS$ filtration of a bundle $E$, which in the sequel we will abbreviate
for simplicity as:%
\begin{equation*}
0=E_{0}\subset E_{1}\subset \cdots \subset E_{l-1}\subset E_{l}=E
\end{equation*}%
(i.e. we ignore the notation indicating that it is a double filtration), as
stated previously, is in general a filtration only by subsheaves of $E$. We
may think of a subbundle $S\subset E$ of rank $k$ as a holomorphic section
of the Grassmann bundle $Gr(k,E)$, the bundle whose fibre at each point is
the set of $k$-dimensional complex subspaces of the fibre of $E$. Similarly
a filtration by subbundles corresponds to a holomorphic section of the
partial flag bundle $\mathbb{FL}(d_{1},\cdots ,d_{l},E)$, the bundle whose
fibre at each point is the set of $l$ flags of type $(d_{1},\cdots ,d_{l})$
where $d_{i}=\limfunc{rk}(E_{i})$. On the other hand a filtration by
subsheaves corresponds to a rational section $X\overset{\sigma }{%
\dashrightarrow }\mathbb{FL}(d_{1},\cdots ,d_{l},E)$. The corollary says
that by blowing up finitely many times along complex submanifolds, we obtain
an honest section $\tilde{X}\rightarrow \mathbb{FL}(d_{1},\cdots ,d_{l},\pi
^{\ast }E)$. More explicitly, we have a diagram:%
\begin{equation*}
\begin{array}{ccc}
\tilde{X} & \overset{\overset{\tilde{\sigma}}{\longrightarrow }}{%
\longleftarrow } & \mathbb{FL}(\pi ^{\ast }E) \\ 
\downarrow & \searrow & \downarrow \\ 
X & \overset{\overset{\sigma }{\dashrightarrow }}{\underset{p}{%
\longleftarrow }} & \mathbb{FL}(E)%
\end{array}%
\end{equation*}%
where $\tilde{\sigma}$ will be constructed below. The outer square is just
the pullback diagram for the map $\tilde{X}\overset{\pi }{\rightarrow }X$.
First we claim that the triangle:%
\begin{equation*}
\begin{array}{ccc}
\tilde{X} &  &  \\ 
\downarrow & \searrow &  \\ 
X & \longleftarrow & \mathbb{FL}(E)%
\end{array}%
\end{equation*}%
commutes. If we write $\psi $ for the desingularised map $\tilde{X}%
\longrightarrow \mathbb{FL}(E)$, then note that for a point $\tilde{x}\in 
\tilde{X}-\mathbf{E}$, we have $\psi (\tilde{x})=\psi (\pi ^{-1}(x))$ for $%
x\in Z_{\limfunc{alg}}$. Then we have: $p(\psi (\tilde{x}))=p(\sigma (\pi (%
\tilde{x})))=x=\pi (\tilde{x})$ since $\sigma $ is well-defined and a
section away from $Z_{\limfunc{alg}}$ and we know the diagram:%
\begin{equation*}
\begin{array}{ccc}
\tilde{X} &  &  \\ 
\downarrow & \searrow &  \\ 
X & \overset{\sigma }{\dashrightarrow } & \mathbb{FL}(E)%
\end{array}%
\end{equation*}%
commutes. In other words on $\tilde{X}-\mathbf{E}$ we have $p\circ \psi =\pi 
$. But since both of these are holomorphic maps $\tilde{X}\longrightarrow X$%
, $p\circ \psi =\pi $ on $\tilde{X}$ by the uniqueness principle for
holomorphic maps, since they agree on a non-empty open subset. Now $\mathbb{%
FL}(\pi ^{\ast }E)=\pi ^{\ast }\mathbb{FL}(E)=\{(\tilde{x},\nu )\in \tilde{X}%
\times \mathbb{FL}(E)\mid \pi (\tilde{x})=p(\nu )\}$. Now define $\tilde{%
\sigma}:\tilde{X}\longrightarrow \mathbb{FL}(\pi ^{\ast }E)$ by $\tilde{%
\sigma}(\tilde{x})=(\tilde{x},\psi (\tilde{x}))$. Since $p\circ \psi =\pi $
this is indeed a map into $\mathbb{FL}(\pi ^{\ast }E)$, and it is manifestly
a section.

In other words there is a filtration of $\pi ^{\ast }E$:%
\begin{equation*}
0=\tilde{E}_{0}\subset \tilde{E}_{1}\subset \cdots \subset \tilde{E}%
_{l-1}\subset \tilde{E}_{l}=\pi ^{\ast }E
\end{equation*}%
where the $\tilde{E}_{i}$ are subbundles.

Now note that we have the following diagram:%
\begin{equation*}
\begin{array}{ccc}
&  & \tilde{Q}_{i} \\ 
&  & \uparrow \\ 
&  & \pi ^{\ast }E \\ 
& \nearrow & \uparrow \\ 
\pi ^{\ast }E_{i} & \dashrightarrow & \tilde{E}_{i}%
\end{array}%
\end{equation*}%
where the dashed line is the rational map corresponding to the equality of $%
\pi ^{\ast }E_{i}$ and $\tilde{E}_{i}$ away from $\mathbf{E}$ (both are
equal to $E_{i}$), and $\tilde{Q}_{i}$ is the quotient of $\pi ^{\ast }E$ by 
$\tilde{E}_{i}$. Then $\tilde{Q}_{i}$ is a vector bundle and in particular
torsion free. On the other hand the image of $\pi ^{\ast }E_{i}$ under the
composition $\pi ^{\ast }E_{i}\rightarrow \pi ^{\ast }E\rightarrow \tilde{Q}%
_{i}$ is torsion since it is supported on the divisor $\mathbf{E}$, and
hence must be zero. If we write $\limfunc{Im}\pi ^{\ast }E_{i}$ for the
image of $\pi ^{\ast }E_{i}\longrightarrow \pi ^{\ast }E$, this means there
is an actual inclusion of sheaves $\limfunc{Im}\pi ^{\ast
}E_{i}\hookrightarrow \tilde{E}_{i}$. The quotient sheaf $\tilde{E}_{i}/%
\limfunc{Im}\pi ^{\ast }E_{i}$ is supported on $\mathbf{E}$, hence torsion
and so it follows from Lemma \ref{Lemma1} that $\tilde{E}_{i}=\limfunc{Sat}%
_{\pi ^{\ast }E}(\limfunc{Im}\pi ^{\ast }E_{i})$.

Since $\pi _{\ast }\tilde{E}_{i}$ is equal to $E_{i}$ away from $\limfunc{%
Sing}E_{i}$ there is a birational map $E_{i}\dashrightarrow \pi _{\ast }%
\tilde{E}_{i}$. Now consider the short exact sequence:%
\begin{equation*}
0\longrightarrow \tilde{E}_{i}\longrightarrow \pi ^{\ast }E\rightarrow 
\tilde{Q}_{i}\longrightarrow 0.
\end{equation*}%
Pushing this sequence forward, and noting that $\pi _{\ast }\tilde{Q}_{i}$
is torsion free and hence injects into its double dual, we have an exact
sequence:%
\begin{equation*}
0\longrightarrow \pi _{\ast }\tilde{E}_{i}\longrightarrow E\rightarrow
\left( \pi _{\ast }\tilde{Q}_{i}\right) ^{\ast \ast }.
\end{equation*}%
Recall that a sheaf $S$ is normal if for any analytic subset $Z$ of
codimension at least two, the map $S(U)\longrightarrow S(U-Z)$ on local
sections is an isomorphism for any open set $U$. In other words, the local
sections of a normal sheaf extend over codimension two subvarieties.
Furthermore, recall that a sheaf is reflexive if and only if it is both
torsion free and normal. Then $(\pi _{\ast }\tilde{Q}_{i})^{\ast \ast }$ and 
$E$ are in particular both normal since they are reflexive. A simple diagram
chase reveals that normality of these sheaves together with the exactness of
this last sequence implies that $\pi _{\ast }\tilde{E}_{i}$ is also normal
(and in fact reflexive, since it is also torsion free).

Because $E_{i}$ is saturated by construction, it is also reflexive and
therefore normal. It is easy to see from the definitions that a map between
normal sheaves that is defined away from a codimension two subvariety
extends to a map on all of $X$. Since $\limfunc{Sing}E_{i}$ has singular set
of codimension at least three, the map $E_{i}\dashrightarrow \pi _{\ast }%
\tilde{E}_{i}$ extends to an isomorphism $E_{i}\cong \pi _{\ast }\tilde{E}%
_{i}$.

Similarly, if $\tilde{Q}_{i}=\tilde{E}_{i}/\tilde{E}_{i-1}$, then $\pi
_{\ast }\tilde{Q}_{i}$ is equal to $Q_{i}$ away from $\limfunc{Sing}Q_{i}$
so again we have a birational map $\left( Q_{i}\right) ^{\ast \ast
}\dashrightarrow (\pi _{\ast }\tilde{Q}_{i})^{\ast \ast }$. Since the double
dual is always reflexive, these sheaves are normal, so the map extends to an
isomorphism. To summarise:

\begin{proposition}
\label{Prop13}Let 
\begin{equation*}
0=E_{0}\subset E_{1}\subset \cdots \subset E_{l-1}\subset E_{l}=E
\end{equation*}%
be a filtration of a holomorphic vector bundle $E\rightarrow X$ by saturated
subsheaves and let $Q_{i}=E_{i}/E_{i-1}$. Then there is a finite sequence of
blowups along complex submanifolds whose composition $\pi :\tilde{X}%
\rightarrow X$ enjoys the following properties. There is a filtration 
\begin{equation*}
0=\tilde{E}_{0}\subset \tilde{E}_{1}\subset \cdots \subset \tilde{E}%
_{l-1}\subset \tilde{E}_{l}=\tilde{E}=\pi ^{\ast }E
\end{equation*}%
by subbundles. If we write $\limfunc{Im}\pi ^{\ast }E_{i}$ for the image of $%
\pi ^{\ast }E_{i}\hookrightarrow \pi ^{\ast }E_{i}$, then $\tilde{E}_{i}=%
\limfunc{Sat}_{\pi ^{\ast }E}\left( \limfunc{Im}\pi ^{\ast }E_{i}\right) $.
If $\tilde{Q}_{i}=\tilde{E}_{i}/\tilde{E}_{i-1}$ then we have $\pi _{\ast }%
\tilde{E}_{i}=E_{i}$ and $Q_{i}^{\ast \ast }=(\pi _{\ast }\tilde{Q}%
_{i})^{\ast \ast }$.
\end{proposition}

We will also have occasion to consider ideal sheaves $\mathcal{I}\subset 
\mathcal{O}_{X}$ whose vanishing set is a closed complex subspace $Y\subset
X $. If $Y$ is smooth for example then we may blowup along $Y$ to obtain a
smooth manifold $\pi :\tilde{X}\longrightarrow X$. Denote by $\pi ^{\ast }%
\mathcal{I}$ is the ideal sheaf generated by pulling back local sections of $%
\mathcal{I}$, in other words the ideal sheaf in $\mathcal{O}_{\tilde{X}}$
generated by the image of $\pi ^{-1}\mathcal{I}$ under the map $\pi ^{-1}%
\mathcal{O}_{X}\longrightarrow \mathcal{O}_{\tilde{X}}$ where $\pi ^{-1}%
\mathcal{I}$ and $\pi ^{-1}\mathcal{O}_{X}$ are the inverse image sheaves.
Note that this is not necessarily equal to the usual sheaf theoretic
pullback of $\mathcal{I}$ which is given by $\pi ^{-1}\mathcal{I\otimes }%
_{\pi ^{-1}\mathcal{O}_{X}}\mathcal{O}_{\tilde{X}}$ and may for example have
torsion. The sheaf $\pi ^{\ast }\mathcal{I}$ is sometimes called the
\textquotedblleft inverse image ideal sheaf\textquotedblright . If the order
of vanishing of $\mathcal{I}$ along $Y$ is $m$, then $\pi ^{\ast }\mathcal{%
I\subset }$ $\mathcal{O}_{\tilde{X}}(-m\mathbf{E)}$, that is, every element
of $\pi ^{\ast }\mathcal{I}$ vanishes to order at least $m$ along the smooth
divisor $\mathbf{E}$. In this situation we will use this notation without
further comment. In general $Y$ is not smooth, so we appeal to the following
resolution of singularities theorem, which is sometimes referred to as
\textquotedblleft principalisation of $\mathcal{I}$\textquotedblright\ or
more specifically \textquotedblleft monomialisation of $\mathcal{I}$%
\textquotedblright\ , and results of this type are usually used to prove
resolution of singularities.

\begin{theorem}
\label{Thm7}Let $X$ be a complex manifold and $Y$ a closed complex subspace.
Then there is a finite sequence of blowups along smooth centres whose
composition yields a map $\pi :\tilde{X}\rightarrow X$ such that $\pi :%
\tilde{X}-\mathbf{E}\rightarrow X-W$ is biholomorphic, $\mathbf{E=\pi }%
^{-1}(W)$ is a normal crossings divisor, and $\pi ^{\ast }\mathcal{I=O}_{%
\tilde{X}}(-\sum_{i}m_{i}\mathbf{E}_{i})$ where the $\mathbf{E}_{i}$ are the
irreducible components of $\mathbf{E}$. Moreover, $\pi ^{\ast }\mathcal{I}$
is locally principal (monomial) in the following sense: for any $x\in X$
there is a local coordinate neighbourhood $U\subset X$ containing $x$ and a
local section $f_{0}$ of $\mathcal{O}_{\tilde{X}}(-\sum_{i}m_{i}\mathbf{E}%
_{i})$ over $\mathbf{\pi }^{-1}(U),$ such that if $f_{j}$ is any local
section of $\mathcal{I}$ over $U$, then $\pi ^{\ast
}f_{j}=f_{0}f_{j}^{^{\prime }}$ where $f_{j}^{^{\prime }}$ is a
non-vanishing holomorphic function on $\mathbf{\pi }^{-1}(U)$. Furthermore,
if $\xi _{k}$ are local normal crossings coordinates for $\mathbf{E}$, then
there is a factorisation:%
\begin{equation*}
f_{0}=\tprod_{k}\xi _{k}^{m_{k}}
\end{equation*}%
so that we may write:%
\begin{equation*}
\pi ^{\ast }f_{j}=\tprod_{k}\xi _{k}^{m_{k}}\cdot f_{j}^{^{\prime }}.
\end{equation*}
\end{theorem}

For the proof, see for example Koll\'{a}r \cite{KO}.

\subsection{Metrics on Blowups and Uniform Bounds on the Degree}

Now we consider the case that the original manifold is K\"ahler. The
following proposition says that this property is preserved under blowing up
and is standard in K\"ahler geometry.

\begin{proposition}
\label{Prop14}Let $(X,\omega )$ be a K\"{a}hler manifold, and $Y$ a compact,
complex submanifold. Then the blowup $\tilde{X}=Bl_{Y}X$ along $Y$ is also K%
\"{a}hler. Moreover $\tilde{X}$ possesses a one-parameter family of K\"{a}%
hler metrics given by $\omega _{\varepsilon }=\pi ^{\ast }\omega
+\varepsilon \eta $ where $\varepsilon >0$, $\pi :\tilde{X}\rightarrow X$ is
the blowup map and $\eta $ is itself a K\"{a}hler form on $\tilde{X}$.
\end{proposition}

For the proof see for example \cite{VO}.

We will need a bound on the $\omega _{\varepsilon }$ degree of an arbitrary
subsheaf of a holomorphic vector bundle $E$ that depends on $\varepsilon $
in such a way that as $\varepsilon \to 0$ the degree converges to the degree
of a subsheaf on the base (namely the pushforward). This will be a
consequence of the following lemma.

\begin{lemma}
\label{Lemma10}Let $X$ be a compact complex manifold of dimension $n$ and
let $\tau $ and $\eta $ be closed $(1,1)$-forms with $\tau $ semi-positve
and $\eta $ a K\"{a}hler form. Let $E\rightarrow X$ a holomorphic vector
bundle. Then there is a constant $M$ such that for any subsheaf $S\subset E$
\ with torsion free quotient and any $0<k\leq n-1$:%
\begin{equation*}
\deg _{k}(S,\tau ,\eta )\equiv \int_{X}c_{1}(S)\wedge \tau ^{n-k-1}\wedge
\eta ^{k}\leq M.
\end{equation*}
\end{lemma}

\begin{proof}
Note that when $k=n-1$, $\deg _{k}(S,\tau ,\eta )$ is the ordinary $\eta $
degree of $S$. We follow Kobayashi's proof that the degree of an arbitrary
subsheaf is bounded above. Fix an hermitian metric $h$ on $E$. The general
case will follow from the case when $S$ is a line subbundle $L$. In this
case we can use the formula: $F_{L}=\pi F_{E}\pi +\beta \wedge \beta ^{\ast
} $, where $\pi $ is the orthogonal projection to $L$ and $\beta $ is the
second fundamental form. Since $c_{1}(L)=\frac{i}{2\pi }F_{L}$ we have that:%
\begin{equation*}
\deg _{k}(L,\tau ,\eta )=\frac{i}{2\pi }\int_{X}\pi F_{E}\pi \wedge \tau
^{n-k-1}\wedge \eta ^{k}+\frac{i}{2\pi }\int_{X}\beta \wedge \beta ^{\ast
}\wedge \tau ^{n-k-1}\wedge \eta ^{k}.
\end{equation*}%
Since $\left\Vert \pi \right\Vert _{L^{\infty }(X)}\leq 1$, the first term
is clearly bounded from above. Therefore we only need to check that the
second term is non-positive. This is the case since $\beta $ is a $(0,1)$
form, and therefore $i\beta \wedge \beta ^{\ast }\leq 0$. Therefore $\deg
_{k}(L,\tau ,\eta )\leq M$, for a constant independent of $L$. To extend the
result to all subbundles $F\subset E$, simply find such an $M$ as above for
each exterior power $\Lambda ^{p}E$ for $p=1,\cdots ,\limfunc{rk}E$, and
take the maximum. Then apply the above argument to the line bundle $L=\det
F\hookrightarrow \Lambda ^{p}E$.

In general $S\overset{\imath }{\hookrightarrow }E$ is not a subbundle but
there is an inclusion of sheaves $\det S\hookrightarrow \Lambda ^{p}E$ where 
$p$ is the rank of $S$. If $V$ is the singular set of $S$, then $S$ is a
subbundle away from $V$, and so the inclusion $\det S\overset{\imath }{%
\hookrightarrow }\Lambda ^{p}E$ is a line subbundle away from $V$. Let $%
\sigma $ be any local holomorphic frame for $\det S$. Now consider the set:$%
\ W=\{x\in X\mid \imath (\sigma )(x)=0\}$. Since $\det S$ is a line bundle
this is clearly independent of $\sigma $. Furthermore because $\imath $ is
an injective bundle map away from $V$, any $x\in W$ must be in $V$; that is, 
$W\subset V$. Now write $H=$ $\imath ^{\ast }\left( \Lambda ^{p}h\right) $.
This is an hermitian metric on $\det S$ over $X-W$. On the other hand there
is some hermitian metric $G$ on $\det S$ over all of $X$. We would like to
show that:%
\begin{equation*}
\deg _{k}(S,\tau ,\eta )=\int_{X}c_{1}(\det S,G)\wedge \tau ^{n-k-1}\wedge
\eta ^{k}=\int_{X-W}c_{1}(\det S,H)\wedge \tau ^{n-k-1}\wedge \eta ^{k}
\end{equation*}%
Then applying the above reasoning, the last integral is bounded since just
as before 
\begin{equation*}
\int_{X-W}c_{1}(\det S,H)\wedge \tau ^{n-k-1}\wedge \eta
^{k}=\int_{X-V}c_{1}(S,h_{S})\wedge \tau ^{n-k-1}\wedge \eta ^{k}\leq \frac{i%
}{2\pi }\int_{X-V}\pi F_{E}\pi \wedge \tau ^{n-k-1}\wedge \eta ^{k}
\end{equation*}%
%
%
%
%
%
%
%
%
%
%
%
%
%
%
%
%
%
%
%
%
%
%
%
%
%
%
%
%
%
%
%
%
%
%
%
%
%
%
%
%
%
where $h_{S}$ is the metric on $S_{\mid X-V}$ induced by $h$. Again this is
bounded independently of $\pi $.

We will construct a $C^{\infty }$ function $f$ on $X$ such that $H=fG$ on $%
X-W$. Then the usual formula for the curvature of the associated Chern
connections implies:%
\begin{eqnarray*}
c_{1}(\det S,H) &=&\frac{i}{2\pi }\bar{\partial}\partial \log H=\frac{i}{%
2\pi }\bar{\partial}\partial \log f+c_{1}(\det S,G) \\
&\Longrightarrow &c_{1}(\det S,G)=c_{1}(\det S,H)-\frac{i}{2\pi }\bar{%
\partial}\partial \log f\text{ }on\text{ }X-W.
\end{eqnarray*}%
Finally we will show:%
\begin{equation*}
\int_{X-W}\frac{i}{2\pi }\bar{\partial}\partial \log f\wedge \tau
^{n-k-1}\wedge \eta ^{k}=0.
\end{equation*}

To construct $f$, let $\sigma $ be any local holomorphic frame for $\det S$.
If $(e_{1,}\cdots ,e_{r})$ is a local holomorphic frame for $E$, then
define: $\imath (\sigma )=\sum_{I}\sigma ^{I}e_{I}$, where $%
e_{I}=e_{i_{1}}\wedge \cdots \wedge e_{i_{p}}$, with $i_{1}<\cdots <i_{p}$.
Then let 
\begin{equation*}
f=H(\sigma ,\sigma )/G(\sigma ,\sigma )=\sum_{I,J}H_{IJ}\sigma ^{I}\bar{%
\sigma}^{J}
\end{equation*}%
where $H_{IJ}=\Lambda ^{p}h(e_{I},e_{J})/G(\sigma ,\sigma )$. Then one may
check that $f$ is well-defined independently of $\sigma $. It is a smooth
non-negative function vanishing exactly on $W$. Since the matrix $(H_{IJ})$
is positive definite, $f$ vanishes exactly where all the $\sigma _{I}$
vanish. It is also clear that we have the equality $H=fG$.

To complete the argument we will show that $\frac{i}{2\pi }\bar{\partial}%
\partial \log f$ integrates to zero. Let $\mathcal{I}$ be the sheaf of
ideals in $\mathcal{O}_{X}$ generated by $\{\sigma _{I}\}$. By Theorem \ref%
{Thm7} there is a sequence of smooth blowups $\pi :\tilde{X}\rightarrow X$
such that $\pi ^{\ast }\mathcal{I}$, the inverse image ideal sheaf of $%
\mathcal{I}$, is the ideal sheaf of a divisor $\mathbf{E}=\sum_{i}m_{i}%
\mathbf{E}_{i}$ where the $\mathbf{E}_{i}$ are the irreducible components of
the support of the exceptional divisor $\limfunc{supp}\mathbf{E=\cup }_{i}$ $%
\mathbf{E}_{i}$. In other words $\pi ^{\ast }\mathcal{I}=\mathcal{O}_{\tilde{%
X}}(-\sum_{i}m_{i}\mathbf{E}_{i})$ for some natural numbers $m_{i}$.
Furthermore, we have: $\pi ^{\ast }\sigma ^{I}=\rho ^{I}\cdot \xi
_{i_{1}}^{m_{i_{1}}}\cdots \xi _{i_{s}}^{m_{i_{s}}}$, where $\{\xi
_{i_{j}}\} $ are normal crossings coordinates for $\mathbf{E}$ on an open
set where $\pi ^{\ast }\sigma ^{I}$ is defined, and $\rho ^{I}$ is a
non-vanishing holomorphic function. Therefore we may locally write: $\pi
^{\ast }f=\chi \cdot \left\vert \xi _{i_{1}}\right\vert ^{2m_{i_{1}}}\cdots
\left\vert \xi _{i_{s}}\right\vert ^{2m_{i_{s}}}$, where $\chi $ is a
strictly positive $C^{\infty }$ function defined on $\tilde{X}$. If we write 
$\Phi =\frac{i}{2\pi }\partial \log \chi $, and $T_{d\Phi }$ for the current
defined by $d\Phi =\frac{i}{2\pi }\bar{\partial}\partial \log \chi $, then
since by definition:%
\begin{eqnarray*}
T_{d\Phi }\left( \pi ^{\ast }(\tau ^{n-k-1}\wedge \eta ^{k})\right)
&=&-dT_{\Phi }(\pi ^{\ast }(\tau ^{n-k-1}\wedge \eta ^{k})) \\
T_{\Phi }(d(\pi ^{\ast }(\tau ^{n-k-1}\wedge \eta ^{k})) &=&0
\end{eqnarray*}%
since $\pi ^{\ast }(\tau ^{n-k-1}\wedge \eta ^{k})$ is closed. Away from the
exceptional set we may write locally: 
\begin{eqnarray*}
\frac{i}{2\pi }\partial \log \pi ^{\ast }f &=&\frac{i}{2\pi }\left( \partial
\log \chi +2m_{i_{1}}\partial \log \left\vert \xi _{i_{1}}\right\vert
+\cdots +2m_{i_{s}}\partial \log \left\vert \xi _{i_{s}}\right\vert \right)
\\
&=&\Phi +\frac{i}{2\pi }\left( \frac{m_{i_{1}}d\xi _{i_{1}}}{\xi _{i_{1}}}%
+\cdots +\frac{m_{i_{s}}d\xi _{i_{s}}}{\xi _{i_{s}}}\right) \text{.}
\end{eqnarray*}%
The second term is integrable on its domain of definition and so $\frac{i}{%
2\pi }\bar{\partial}\partial \log \pi ^{\ast }f$ is a $(1,1)$ form with $%
L_{loc}^{1}(\tilde{X})$ coefficients, and so defines a current. On the other
hand by the Poincar\'{e}-Lelong formula, $\bar{\partial}$ applied to the
second term is equal to $\displaystyle\sum_{i_{j}}m_{i_{j}}T_{\mathbf{E}%
_{i_{j}}}$, in the sense of currents, where $T_{\mathbf{E}_{i_{j}}}$ is the
current defined by the smooth hypersurface $\mathbf{E}_{i_{j}}$. Finally
then:%
\begin{align*}
& \qquad \qquad \int_{X-W}\frac{i}{2\pi }\bar{\partial}\partial \log f\wedge
\pi ^{\ast }\tau ^{n-k-1}\wedge \pi ^{\ast }\eta ^{k}=\int_{\tilde{X}-%
\mathbf{E}}\frac{i}{2\pi }\bar{\partial}\partial \log \pi ^{\ast }f\wedge
\pi ^{\ast }\tau ^{n-k-1}\wedge \pi ^{\ast }\eta ^{k} \\
& =T_{\frac{i}{2\pi }\bar{\partial}\partial \log \pi ^{\ast }f}(\pi ^{\ast
}\tau ^{n-k-1}\wedge \pi ^{\ast }\eta ^{k})=\left( \sum_{i}m_{i}T_{\mathbf{E}%
_{i}}\right) (\pi ^{\ast }\tau ^{n-k-1}\wedge \pi ^{\ast }\eta
^{k})=\sum_{i}m_{i}\int_{\mathbf{E}_{i}}\pi ^{\ast }\tau ^{n-k-1}\wedge \pi
^{\ast }\eta ^{k}=0
\end{align*}%
since the image of $\mathbf{E}_{i}$ under $\pi $ has codimension at least
two. This completes the proof.
\end{proof}

\begin{remark}
\label{Rmk4}If $0\to S\to E\to Q\to 0 $ is an exact sequence, where $E$ is a
vector bundle and $Q$ is torsion free, then the dualised sequence $0\to
Q^{\ast }\to E^{\ast }\to S^{\ast } $ is exact, and so as in the above lemma
there is a constant $M$ associated to $E$ independent of $Q$ so that 
\begin{equation*}
-\int_{X}c_{1}(Q)\wedge \tau ^{n-k-1}\wedge \eta ^{k}=\int_{X}c_{1}(Q^{\ast
})\wedge \tau ^{n-k-1}\wedge \eta ^{k}\leq M.
\end{equation*}%
In other words there is a uniform constant $M$ so that: $-M\leq
\int_{X}c_{1}(Q)\wedge \tau ^{n-k-1}\wedge \eta ^{k} $, where $Q$ is any
torsion-free quotient of $E$.
\end{remark}

\begin{remark}
In the case that $k=n-1$, $\deg _{k}(S,\tau ,\eta )=\deg (S,\eta )$ and the
above constitutes a proof of Simpson's degree formula.
\end{remark}

We note that if $\tilde{X}\rightarrow X$ is a composition of finitely many
blowups then we also have a family of K\"{a}hler metrics on $\tilde{X}$ by
iteratively applying $\limfunc{Proposition}$ \ref{Prop14}. We would now like
to compute the degree of an arbitrary torsion-free sheaf $\tilde{S}$ on $%
\tilde{X}$ with respect to each metric $\omega _{\varepsilon }$ on $\tilde{X}
$.

\begin{theorem}
\label{THM}Let $\tilde{S}$ be a subsheaf (with torsion free quotient $\tilde{%
Q}$) of a holomorphic vector bundle $\tilde{E}$ on $\tilde{X}$, where $\pi :%
\tilde{X}\rightarrow X$ is given by a sequence of blowups along complex
submanifolds of $\limfunc{codim}\geq 2$. Then there is a uniform constant $M$
independent of $\tilde{S}$ such that the degrees of $\tilde{S}$ and $\tilde{Q%
}$ with respect to $\omega _{\varepsilon }$ satisfy: $\deg (\tilde{S},\omega
_{\varepsilon })\leq \deg (\pi _{\ast }\tilde{S})+\varepsilon M$, and $\deg (%
\tilde{Q},\omega _{\varepsilon })\geq \deg (\pi _{\ast }\tilde{Q}%
)-\varepsilon M$.
\end{theorem}

\begin{proof}
The general case will follow from the case when $\tilde{S}$ is a line bundle 
$\tilde{L}$ (perhaps not a line subbundle). Recall that the Picard group of
the blowup $Pic(\tilde{X})=Pic(X)\oplus \mathbb{Z}\mathcal{O}(\mathbf{E}%
_{1})\oplus \cdots \oplus \mathbb{Z}\mathcal{O}(\mathbf{E}_{m})$ where the $%
\mathbf{E}_{i}$ are the irreducible components of the exceptional divisor.
That is, we may write an arbitrary line bundle as $\tilde{L}=\pi ^{\ast
}L\otimes \mathcal{O}_{\tilde{X}}(\sum_{i}m_{i}\mathbf{E}_{i})$ where $L$ is
a line bundle on $X$. Then by definition:%
\begin{equation*}
\deg (\tilde{L},\omega _{\varepsilon })=\int_{\tilde{X}}c_{1}(\tilde{L}%
)\wedge \omega _{\varepsilon }^{n-1}=\int_{\tilde{X}}c_{1}(\tilde{L})\wedge
\left( \pi ^{\ast }\omega +\varepsilon \eta \right) ^{n-1}.
\end{equation*}%
Then we have an expansion:%
\begin{equation*}
\left( \pi ^{\ast }\omega +\varepsilon \eta \right) ^{n-1}=\pi ^{\ast
}\omega ^{n-1}+\varepsilon (n-1)\pi ^{\ast }\omega ^{n-2}\wedge \eta +\cdots
+\varepsilon ^{n-2}(n-1)\pi ^{\ast }\omega \wedge \eta ^{n-2}+\varepsilon
^{n-1}\eta ^{n-1}.
\end{equation*}%
Note that $\displaystyle\int_{\tilde{X}}c_{1}(\mathcal{O}_{\tilde{X}}%
\mathcal{(}\mathbf{E}_{i}))\wedge \pi ^{\ast }\omega ^{n-1}=\int_{\mathbf{E}%
_{i}}\left( \pi ^{\ast }\omega \right) ^{n-1}=0$, since the image in $X$ of
each $\mathbf{E}_{i}$ lives in codimension at least $2$. Therefore we are
left with%
\begin{eqnarray*}
\deg (\tilde{L},\omega _{\varepsilon }) &=&\int_{\tilde{X}}c_{1}(\tilde{L}%
)\wedge \pi ^{\ast }\omega ^{n-1}+\sum_{k}\varepsilon ^{k}\binom{n-1}{k}%
\left( \int_{\tilde{X}}c_{1}(\tilde{L})\wedge \pi ^{\ast }\omega
^{n-k-1}\wedge \eta ^{k}\right) \\
&=&\int_{\tilde{X}}\pi ^{\ast }c_{1}(L)\wedge \pi ^{\ast }\omega
^{n-1}+\sum_{i}m_{i}\int_{\tilde{X}}c_{1}(\mathcal{O}_{\tilde{X}}\mathcal{(}%
\mathbf{E}_{i}))\wedge \pi ^{\ast }\omega ^{n-1} \\
&&+\sum_{k}\varepsilon ^{k}\binom{n-1}{k}\int_{\tilde{X}}c_{1}(\tilde{L}%
)\wedge \pi ^{\ast }\omega ^{n-k-1}\wedge \eta ^{k} \\
&=&\deg (L,\omega )+\sum_{k}\varepsilon ^{k}\binom{n-1}{k}\int_{\tilde{X}%
}c_{1}(\tilde{L})\wedge \pi ^{\ast }\omega ^{n-k-1}\wedge \eta ^{k}
\end{eqnarray*}%
By the previous lemma the terms $\displaystyle\int_{\tilde{X}}c_{1}(\tilde{L}%
)\wedge \pi ^{\ast }\omega ^{n-k-1}\wedge \eta ^{k}$, are all bounded
uniformly independently of $\varepsilon $ since $\pi ^{\ast }\omega $ is
semi-positive and $\eta $ is a K\"{a}hler form. Therefore we have: $\deg (%
\tilde{L},\omega _{\varepsilon })\leq \deg (L,\omega )+\varepsilon M$.

Now note that if $\tilde{X}=Bl_{Y}X$ then $\pi _{\ast }\mathcal{O}(m\mathbf{E%
})=\mathcal{O}_{X}$ if $m\geq 0$ and $\pi _{\ast }\mathcal{O}(m\mathbf{E}%
)=I_{Y}^{\otimes m}$ if $m<0$, where $I_{Y}$ is the ideal sheaf of
holomorphic functions on $X$ vanishing on $Y$. The determinant of an ideal
sheaf is trivial if $Y$ has codimension at least $2$, so we have $\det (\pi
_{\ast }\tilde{L})=\det (L)$ so finally: $\deg (\tilde{L},\omega
_{\varepsilon })\leq \deg (\pi _{\ast }\tilde{L})+\varepsilon M$.

Now for an arbitrary subsheaf $\tilde{S}\subset \tilde{E}$, by definition $%
\deg (\tilde{S},\omega _{\varepsilon })=\deg (\det (\tilde{S}),\omega
_{\varepsilon })$. When $\pi _{\ast }\tilde{S}$ is a vector bundle, that is,
away from its algebraic singular set, we have an isomorphism $\det (\pi
_{\ast }\tilde{S})=\pi _{\ast }\det \tilde{S}$. Their determinants are
therefore isomorphic away from this set, and so by Hartogs' theorem there is
an isomorphism of line bundles: $\det (\pi _{\ast }\tilde{S})=\det (\pi
_{\ast }\det \tilde{S})$ on $X$. Therefore by the previous argument: 
\begin{equation*}
\deg (\tilde{S},\omega _{\varepsilon })=\deg (\det (\tilde{S}),\omega
_{\varepsilon })\leq \deg (\pi _{\ast }\det \tilde{S})+\varepsilon M=\deg
(\pi _{\ast }\tilde{S})+\varepsilon M\ .
\end{equation*}%
The exact same argument together with the previous remark proves the second
inequality as well.
\end{proof}

\subsection{Stability on Blowups and Convergence of the $HN$ Type}

\begin{proposition}
\label{Prop16}Let $\tilde{E}\rightarrow \tilde{X}$ a holomorphic vector
bundle where $\tilde{X}\rightarrow X$ is a sequence of blowups. If $\pi
_{\ast }\tilde{E}$ is $\omega $-stable, then there is an $\varepsilon _{2}$
such that $\tilde{E}$ is $\omega _{\varepsilon }$-stable for all $%
0<\varepsilon \leq \varepsilon _{2}$.
\end{proposition}

\begin{proof}
Suppose there is a destabilising subsheaf $\tilde{S}_{\varepsilon }\subset 
\tilde{E}$, i.e. $\mu _{\omega _{\varepsilon }}(\tilde{S}_{\varepsilon
})\geq \mu _{\omega _{\varepsilon }}(\tilde{E})$ for each $\varepsilon $.
Now among all proper subsheaves of $\pi _{\ast }\tilde{E}$, the maximal
slope is realised by some subsheaf $\mathcal{F}$, in other words:%
\begin{equation*}
\mu _{\omega }(\mathcal{F})=\sup \{\mu _{\omega }(S)\mid S\subset \pi _{\ast
}\tilde{E}\}.
\end{equation*}
Then by the previous theorem we have:%
\begin{equation*}
\mu _{\omega }(\pi _{\ast }\tilde{E})-\varepsilon M\leq \mu _{\omega
_{\varepsilon }}(\tilde{E})\leq \mu _{\omega }(\pi _{\ast }\tilde{S}%
_{\varepsilon })+\varepsilon M\leq \mu _{\omega }(\mathcal{F})+\varepsilon M.
\end{equation*}%
In other words:%
\begin{equation*}
\mu _{\omega }(\pi _{\ast }\tilde{E})\leq \mu _{\omega }(\mathcal{F}%
)+2\varepsilon M
\end{equation*}
Since $\pi _{\ast }\tilde{E}$ is $\omega $-stable, $\mu _{\omega }(\mathcal{F%
})<\mu _{\omega }(\pi _{\ast }\tilde{E})$. Since the constant $M$ is
independent of $\varepsilon $, when $\varepsilon $ is sufficiently small
(more specifically, when $\varepsilon <(\mu _{\omega }(\pi _{\ast }\tilde{E}%
)-\mu _{\omega }(\mathcal{F}))/2M$), we have 
\begin{equation*}
\mu _{\omega }(\pi _{\ast }\tilde{E})\leq \mu _{\omega }(\mathcal{F}%
)+2\varepsilon M<\mu _{\omega }(\pi _{\ast }\tilde{E}),
\end{equation*}%
which is a contradiction. \ 
\end{proof}

\begin{remark}
\label{Rmk5}This shows in particular that for any resolution of a $HNS$
filtration, the quotients $\tilde{Q}_{i}=\tilde{E}_{i}/\tilde{E}_{i-1}$ are
stable with respect to $\omega _{\varepsilon }$ for $\varepsilon $
sufficiently small, since the double dual of the pushforward is the double
dual of $Q_{i}$ which is stable by construction. This fact will be important
in Section$\ 5$.
\end{remark}

For each of the metrics $\omega _{\varepsilon }$ there is also an $HNS$
filtration of the pullback $\pi ^{\ast }E$. We will need information about
what happens to the corresponding $HN$ types as $\varepsilon \rightarrow 0$.
Namely we have:

\begin{proposition}
\label{Prop17}Let $E\rightarrow X$ be a holomorphic vector bundle and $\pi :%
\tilde{X}\rightarrow X$ be a finite sequence of blowups resolving the $HNS$
filtration. Then the $HN$ type $\left( \mu _{1}^{\varepsilon },\cdots ,\mu
_{K}^{\varepsilon }\right) $ of $\pi ^{\ast }E$ with respect to $\omega
_{\varepsilon }$ converges to the $HN$ type $\left( \mu _{1},\cdots ,\mu
_{K}\right) $ of $E$ with respect to $\omega $ as $\varepsilon
\longrightarrow 0$.
\end{proposition}

\begin{proof}
Let 
\begin{equation*}
0=\tilde{E}_{0}\subset \tilde{E}_{1}\subset \tilde{E}_{2}\subset \cdots
\subset \tilde{E}_{l-1}\subset \tilde{E}_{l}=\pi ^{\ast }E
\end{equation*}%
be a resolution of the $HNS$ filtration. Since all the information about the 
$HN$ type is contained in the $HN$ filtration%
\begin{equation*}
0=\mathbb{F}_{0}^{HN}\subset \mathbb{F}_{1}^{HN}(E)\subset \mathbb{F}%
_{2}^{HN}(E)\subset \cdots \subset \mathbb{F}_{l}^{HN}(E)=E,
\end{equation*}%
we will just regard this as a resolution of singularities of the $HN$
filtration and forget about Seshadri filtrations for the rest of this proof.

We would like to relate the resolution of the $HN$ filtration of $(E,\omega
) $, to the $HN$ filtration of $(\pi ^{\ast }E,\omega _{\varepsilon })$ for
small $\varepsilon $. We claim that for all $\varepsilon $ in a sufficient
range we may arrange that $\mu _{\omega _{\varepsilon }}^{\min }(\tilde{E}%
_{i})>\mu _{\omega _{\varepsilon }}^{\max }(\pi ^{\ast }E/\tilde{E}_{i})$.
Let $\mathcal{F}_{1}\subset \tilde{E}_{i}\subset \mathcal{F}_{2}\subset \pi
^{\ast }E$ be any subsheaves such that $\tilde{E}_{i}/\mathcal{F}_{1}$ is
torsion free. Note that for $\tilde{x}\in \tilde{X}$ with $\pi (\tilde{x})=x$%
, we always have maps on the stalks $\left( \pi _{\ast }\mathcal{F}%
_{i}\right) _{x}\rightarrow \left( \mathcal{F}_{i}\right) _{\tilde{x}}$.
Since $\pi $ is in particular a biholomorphism away from $\mathbf{E}$, when $%
\tilde{x}\in \tilde{X}-\mathbf{E}$ these maps are isomorphisms. In other
words the sequences:%
\begin{equation*}
0\longrightarrow \pi _{\ast }\mathcal{F}_{1}\longrightarrow
E_{i}\longrightarrow \pi _{\ast }\left( \tilde{E}_{i}/\mathcal{F}_{1}\right)
\longrightarrow 0
\end{equation*}%
and%
\begin{equation*}
0\longrightarrow E_{i}\longrightarrow \pi _{\ast }\mathcal{F}%
_{2}\longrightarrow \pi _{\ast }\left( \mathcal{F}_{2}/\tilde{E}_{i}\right)
\longrightarrow 0
\end{equation*}%
are exact away from the singular set $Z_{\limfunc{alg}}$. In particular this
means $E_{i}/\pi _{\ast }\mathcal{F}_{1}\hookrightarrow \pi _{\ast }(\tilde{E%
}_{i}/\mathcal{F}_{1})$ and $\pi _{\ast }\mathcal{F}_{2}/E_{i}%
\hookrightarrow \pi _{\ast }(\mathcal{F}_{2}/\tilde{E}_{i})$ with torsion
quotients, which implies $(E_{i}/\pi _{\ast }\mathcal{F}_{1})^{\ast \ast
}=(\pi _{\ast }(\tilde{E}_{i}/\mathcal{F}_{1}))^{\ast \ast }$ and $(\pi
_{\ast }\mathcal{F}_{2}/E_{i})^{\ast \ast }=(\pi _{\ast }(\mathcal{F}_{2}/%
\tilde{E}_{i}))^{\ast \ast }$. Then finally we have $\mu _{\omega
}(E_{i}/\pi _{\ast }\mathcal{F}_{1})=$ $\mu _{\omega }(\pi _{\ast }(\tilde{E}%
_{i}/\mathcal{F}_{1}))$ and $\mu _{\omega }(\pi _{\ast }\mathcal{F}%
_{2}/E_{i})=$ $\mu _{\omega }(\pi _{\ast }(\mathcal{F}_{2}/\tilde{E}_{i}))$.

The above argument together with Theorem \ref{THM} now implies that $\mu
_{\omega _{\varepsilon }}(\tilde{E}_{i}/\mathcal{F}_{1})\geq \mu _{\omega
}(E_{i}/\pi _{\ast }\mathcal{F}_{1})-\varepsilon M$ and $\mu _{\omega
_{\varepsilon }}(\mathcal{F}_{2}/\tilde{E}_{i})\leq \mu _{\omega }(\pi
_{\ast }\mathcal{F}_{2}/E_{i})+\varepsilon M$. On the other hand: $\mu
_{\omega }(E_{i}/\pi _{\ast }\mathcal{F}_{1})\geq \mu _{\omega }(Q_{i})>\mu
_{\omega }(Q_{i+1})\geq \mu _{\omega }(\pi _{\ast }\mathcal{F}_{2}/E_{i})$,
where we have used the facts that $\mu _{\omega }(Q_{i})=\mu _{\omega
}^{\min }(E_{i})$ and $\mu _{\omega }(Q_{i+1})=\mu _{\omega }^{\max
}(E/E_{i})$. Therefore we have:%
\begin{equation*}
\mu _{\omega _{\varepsilon }}(\tilde{E}_{i}/\mathcal{F}_{1})-\mu _{\omega
_{\varepsilon }}(\mathcal{F}_{2}/\tilde{E}_{i})\geq \left( \mu _{\omega
}(E_{i}/\pi _{\ast }\mathcal{F}_{1})-\mu _{\omega }(\pi _{\ast }\mathcal{F}%
_{2}/E_{i})\right) -2\varepsilon M.
\end{equation*}%
As we have shown, the first term on the right hand side is strictly
positive, so when $\varepsilon $ is sufficiently small the entire right hand
side is strictly positive. Since $\mathcal{F}_{1}$ and $\mathcal{F}_{2}$
were arbitrary, for $\varepsilon $ small $\mu _{\omega _{\varepsilon
}}^{\min }(\tilde{E}_{i})$ must be strictly bigger than $\mu _{\omega
_{\varepsilon }}^{\max }(\pi ^{\ast }E/\tilde{E}_{i})$.

Now it follows from $\limfunc{Proposition}$ \ref{Prop4} that the $HN$
filtration of $(\pi ^{\ast }E,\omega _{\varepsilon })$ is:%
\begin{eqnarray*}
0 &\subset &\mathbb{F}_{1}^{HN,\varepsilon }(\tilde{E}_{1})\subset \cdots
\subset \mathbb{F}_{k_{1}}^{HN,\varepsilon }(\tilde{E}_{1})=\tilde{E}%
_{1}\subset \cdots \subset \mathbb{F}_{k_{1}+\cdots
+k_{l-1}}^{HN,\varepsilon }(\tilde{E}_{l-1})=\tilde{E}_{l-1} \\
&\subset &\mathbb{F}_{k_{1}+\cdots +k_{l-1}+1}^{HN,\varepsilon }(\tilde{E}%
_{l})\subset \cdots \subset \mathbb{F}_{k_{1}+\cdots +k_{l}}^{HN,\varepsilon
}(\tilde{E}_{l})=\pi ^{\ast }E.
\end{eqnarray*}

That is, the resolution appears within the $HN$ filtration with respect to $%
\omega _{\varepsilon }$, and two successive subbundles in the resolution are
separated by the $HN$ filtration of the larger bundle. Then for any $i$ we
consider the following part of the above filtration:%
\begin{eqnarray*}
\tilde{E}_{i-1} &=&\mathbb{F}_{k_{1}+\cdots +k_{i-1}}^{HN,\varepsilon }(%
\tilde{E}_{i-1})\subset \mathbb{F}_{k_{1}+\cdots +k_{i-1}+1}^{HN,\varepsilon
}(\tilde{E}_{i})\subset  \\
\cdots  &\subset &\mathbb{F}_{k_{1}+\cdots +k_{i}-1}^{HN,\varepsilon }(%
\tilde{E}_{i})\subset \mathbb{F}_{k_{1}+\cdots +k_{i}}^{HN,\varepsilon }(%
\tilde{E}_{i})=\tilde{E}_{i}.
\end{eqnarray*}%
We claim that:%
\begin{equation*}
\mu _{\omega _{\varepsilon }}\left( \mathbb{F}_{k_{1}+\cdots
+k_{i-1}+j}^{HN,\varepsilon }(\tilde{E}_{i})/\mathbb{F}_{k_{1}+\cdots
+k_{i-1}+j-1}^{HN,\varepsilon }(\tilde{E}_{i})\right) \longrightarrow \mu
_{\omega }(E_{i}/E_{i-1})=\mu _{\omega }(Q_{i})
\end{equation*}%
for each $1\leq j\leq k_{i}$. Then the proposition will follow immediately.
The slopes of the quotients in the $HN$ filtration are strictly decreasing
so we have:%
\begin{eqnarray*}
\mu _{\omega _{\varepsilon }}\left( \tilde{E}_{i}/\mathbb{F}_{k_{1}+\cdots
+k_{i}-1}^{HN,\varepsilon }(\tilde{E}_{i})\right)  &<&\mu _{\omega
_{\varepsilon }}\left( \mathbb{F}_{k_{1}+\cdots +k_{i-1}+j}^{HN,\varepsilon
}(\tilde{E}_{i})/\mathbb{F}_{k_{1}+\cdots +k_{i-1}+j-1}^{HN,\varepsilon }(%
\tilde{E}_{i})\right)  \\
&<&\mu _{\omega _{\varepsilon }}\left( \mathbb{F}_{k_{1}+\cdots
+k_{i-1}+1}^{HN,\varepsilon }(\tilde{E}_{i-1})/\tilde{E}_{i-1}\right) .
\end{eqnarray*}%
Therefore it suffices to prove convergence of 
\begin{equation*}
\mu _{\omega _{\varepsilon }}\left( \tilde{E}_{i}/\mathbb{F}_{k_{1}+\cdots
+k_{i}-1}^{HN,\varepsilon }(\tilde{E}_{i})\right) \text{ }and\text{ }\mu
_{\omega _{\varepsilon }}\left( \mathbb{F}_{k_{1}+\cdots
+k_{i-1}+1}^{HN,\varepsilon }(\tilde{E}_{i-1})/\tilde{E}_{i-1}\right) 
\end{equation*}%
to $\mu _{\omega }(Q_{i})$ as $\varepsilon \rightarrow 0$. Note that just as
before we may argue that%
\begin{equation*}
\mu _{\omega }\left( \pi _{\ast }\left( \tilde{E}_{i}/\mathbb{F}%
_{k_{1}+\cdots +k_{i}-1}^{HN,\varepsilon }(\tilde{E}_{i})\right) \right)
=\mu _{\omega }\left( E_{i}/\pi _{\ast }\mathbb{F}_{k_{1}+\cdots
+k_{i}-1}^{HN,\varepsilon }(\tilde{E}_{i})\right) 
\end{equation*}%
and%
\begin{equation*}
\mu _{\omega }\left( \pi _{\ast }\left( \mathbb{F}_{k_{1}+\cdots
+k_{i-1}+1}^{HN,\varepsilon }(\tilde{E}_{i-1})/\tilde{E}_{i-1}\right)
\right) =\mu _{\omega }\left( \pi _{\ast }\mathbb{F}_{k_{1}+\cdots
+k_{i-1}+1}^{HN,\varepsilon }(\tilde{E}_{i-1})/E_{i-1}\right) .
\end{equation*}

By Theorem \ref{THM} we have:%
\begin{eqnarray*}
\mu _{\omega }(Q_{i})-\varepsilon M &=&\mu _{\omega }(\pi _{\ast }\tilde{Q}%
_{i})-\varepsilon M\leq \mu _{\omega _{\varepsilon }}(\tilde{Q}_{i})\leq \mu
_{\omega _{\varepsilon }}\left( \mathbb{F}_{k_{1}+\cdots
+k_{i-1}+1}^{HN,\varepsilon }(\tilde{E}_{i-1})/\tilde{E}_{i-1}\right) \\
&\leq &\mu _{\omega }\left( \pi _{\ast }\mathbb{F}_{k_{1}+\cdots
+k_{i-1}+1}^{HN,\varepsilon }(\tilde{E}_{i-1})/E_{i-1}\right) +\varepsilon
M\leq \mu _{\omega }\left( E_{i}/E_{i-1}\right) +\varepsilon M \\
&=&\mu _{\omega }(Q_{i})+\varepsilon M
\end{eqnarray*}%
where we have used that $F_{k_{1}+\cdots +k_{i-1}+1}^{HN,\varepsilon }(%
\tilde{E}_{i-1})$ is maximally destabilising in $\pi ^{\ast }E/\tilde{E}%
_{i-1}$ and $E_{i}/E_{i-1}$ is maximally destabilising in $E/E_{i-1}$. So 
\begin{equation*}
\mu _{\omega _{\varepsilon }}\left( \mathbb{F}_{k_{1}+\cdots
+k_{i-1}+1}^{HN,\varepsilon }(\tilde{E}_{i-1})/\tilde{E}_{i-1}\right)
\longrightarrow \mu _{\omega }(Q_{i})\text{.}
\end{equation*}

Similiarly we have:%
\begin{eqnarray*}
\mu _{\omega }\left( Q_{i}\right) -\varepsilon M &=&\mu _{\omega }\left(
E_{i}/E_{i-1}\right) -\varepsilon M\leq \mu _{\omega }\left( E_{i}/\pi
_{\ast }\mathbb{F}_{k_{1}+\cdots +k_{i}-1}^{HN,\varepsilon }(\tilde{E}%
_{i})\right) -\varepsilon M \\
&\leq &\mu _{\omega _{\varepsilon }}\left( \tilde{E}_{i}/\mathbb{F}%
_{k_{1}+\cdots +k_{i}-1}^{HN,\varepsilon }(\tilde{E}_{i})\right) \leq \mu
_{\omega _{\varepsilon }}(\tilde{Q}_{i})\leq \mu _{\omega }\left( \pi _{\ast
}\tilde{Q}_{i}\right) +\varepsilon M \\
&=&\mu _{\omega }\left( Q_{i}\right) +\varepsilon M
\end{eqnarray*}%
where we have used that $\mu _{\omega }\left( E_{i}/E_{i-1}\right) =\mu
_{\omega }^{\min }\left( E_{i}\right) $ and $\mu _{\omega _{\varepsilon
}}\left( \tilde{E}_{i}/\mathbb{F}_{k_{1}+\cdots +k_{i}-1}^{HN,\varepsilon }(%
\tilde{E}_{i})\right) =\mu _{\omega _{\varepsilon }}^{\min }(\tilde{E}_{i})$%
. Then taking limits implies $\mu _{\omega _{\varepsilon }}\left( \tilde{E}%
_{i}/\mathbb{F}_{k_{1}+\cdots +k_{i}-1}^{HN,\varepsilon }(\tilde{E}%
_{i})\right) \rightarrow \mu _{\omega }(Q_{i})$. This completes the proof.
\end{proof}

\begin{remark}
\label{R6}Note that the argument of the above proof also shows that we have
convergence:%
\begin{equation*}
\left( \mu _{\omega _{\varepsilon }}(\tilde{Q}_{1}),\cdots ,\mu _{\omega
_{\varepsilon }}(\tilde{Q}_{l})\right) \longrightarrow \left( \mu _{\omega
}(Q_{1}),\cdots ,\mu _{\omega }(Q_{l})\right) ,
\end{equation*}%
where as usual $\mu _{\omega _{\varepsilon }}(\tilde{Q}_{i})$ is repeated $%
\limfunc{rk}(\tilde{Q}_{i})$ times. We will use this fact in the following
section.
\end{remark}

\section{Approximate Critical Hermitian Structures/$HN$ Type of the Limit}

In this section we accomplish two important aims. One is the construction of
a certain canonical type of metric on a holomorphic vector bundle over a K%
\"{a}hler manifold called an $L^{p}$-approximate critical hermitian
structure. The other is identifying the Harder-Narasimhan type of the
limiting vector bundle $E_{\infty }$ along the flow, namely we prove that
this is the same as the type of the original bundle $E$. This latter fact
will be a crucial element in the proof of the main theorem, whereas the
former will play no role in the remainder of the proof. However we remark
that these two theorems are, due to certain technical considerations to be
discussed below, very much intertwined.

If we fix a holomorphic structure on $E$, then a critical point of the $HYM$
functional, thought of as a function $h\mapsto HYM(\bar{\partial}_{E},h)$ on
the space of metrics, is called (see Kobayashi \cite{KOB}) a \textit{%
critical hermitian structure}. The K\"{a}hler identities imply that this
happens exactly when the corresponding connection $(\bar{\partial}_{E},h)$
is Yang-Mills, and hence in this case the Hermitian-Einstein tensor splits: $%
\sqrt{-1}\Lambda _{\omega }F_{(\bar{\partial}_{E},h)}=\mu
_{1}Id_{Q_{1}}\oplus \cdots \oplus \mu _{l}Id_{Q_{l}}$. Here the holomorphic
structure $\bar{\partial}_{E}$ splits into the direct sum $\oplus _{i}Q_{i}$
and the metric induced on each summand is Hermitian-Einstein with constant
factor $\mu _{i}$.

In general, the holomorphic structure on $E$ is not split, and of course the 
$Q_{i}$ may not be subbundles as at all, so it is not the case that we
always have a critical hermitian structure. We therefore need to define a
correct approximate notion of a critical point. In the subsequent discussion
we follow Daskalopoulos-Wentworth \cite{DW1}.

Let $h$ be a smooth metric on $E$ and $\mathcal{F}=\{F_{i}\}_{i=0}^{l}$ a
filtration of $E$ by saturated subsheaves. For every $F_{i}$ we have the
corresponding weakly holomorphic projection $\pi _{i}^{h}$. These are
bounded, $L_{1}^{2}$ hermitian endomorphisms of $E$. Here $F_{0}=0$, and so $%
\pi _{0}^{h}=0$. Given real numbers $\mu _{1},\cdots ,\mu _{l}$, define the
following $L_{1}^{2}$ hermitian endomorphism of $E:$%
\begin{equation*}
\Psi (\mathcal{F},(\mu _{1},\cdots ,\mu _{l}),h)=\sum_{i=1}^{l}\mu
_{i}\left( \pi _{i}^{h}-\pi _{i-1}^{h}\right) .
\end{equation*}%
Notice that away from the singular set of the filtration (points where it is
given by sub-bundles), the bundle $E$ splits smoothly as $\oplus
Q_{i}=\oplus _{i}E_{i}/E_{i-1}$, and with respect to the splitting, the
endomorphism $\Psi (\mathcal{F},(\mu _{1},\cdots ,\mu _{l}),h)$ is just the
diagonal map $\mu _{1}Id_{Q_{1}}\oplus \cdots \oplus \mu _{l}Id_{Q_{l}}$.

In the special case where $E$ is a holomorphic vector bundle over a K\"{a}%
hler manifold $(X,\omega )$, we will write $\Psi _{\omega }^{HNS}(\bar{%
\partial}_{E},h)$ when the filtration of $E$ is the $HNS$ filtration $F_{i}=%
\mathbb{F}_{i}^{HNS}(E)$ and$\ \mu _{1},\cdots ,\mu _{l}$ are the distinct
slopes appearing the $HN$ type.

\begin{definition}
\label{Def6}Fix $\delta >0$ and $1\leq p\leq \infty $. An $L^{p}$ $\delta $%
-approximate critical hermitian structure on a holomorphic bundle $E$ is a
smooth metric $h$ such that:%
\begin{equation*}
\left\Vert \sqrt{-1}\Lambda _{\omega }F_{(\bar{\partial}_{E},h)}-\Psi
_{\omega }^{HNS}(\bar{\partial}_{E},h)\right\Vert _{L^{p}(\omega )}\leq
\delta .
\end{equation*}
\end{definition}

The following theorem first appeared in \cite{DW1}.

\begin{theorem}
\label{Thm8}If the $HNS$ filtration of $E$ is given by subbundles, then for
any $\delta >0$, $E$ has an $L^{\infty }$ $\delta $-approximate critical
hermitian structure.
\end{theorem}

We begin by giving a (very simple) proof of this theorem in the case that
the $HNS$ filtration has length two (the general case follows from an
inductive argument). Namely we assume that there is an exact sequence of the
form:%
\begin{equation*}
0\longrightarrow S\longrightarrow E\longrightarrow Q\longrightarrow 0
\end{equation*}%
where $S$ and $Q$ are stable vector bundles. Then fix Hermitian-Einstein
metrics $h_{S}$ and $h_{Q}$ on $S$ and $Q$. There is a smooth splitting $%
E\simeq S\oplus Q$ and so we may fix the metric $h_{E}=$ $h_{S}\oplus h_{Q}$
on $E$. Of course in general we there is no holomorphic splitting. The
failure of the sequence to split holomorphically is determined by the second
fundamental form $\beta \in \Omega ^{0,1}(\limfunc{Hom}(Q,S))$, and the
holomorphic structure of $E$ may be written as:%
\begin{equation*}
\bar{\partial}_{E}=%
\begin{pmatrix}
\bar{\partial}_{S} & \beta  \\ 
0 & \bar{\partial}_{Q}%
\end{pmatrix}%
\end{equation*}%
and similarly%
\begin{equation*}
\partial _{E}=%
\begin{pmatrix}
\partial _{S} & 0 \\ 
-\beta ^{\ast } & \partial _{Q}%
\end{pmatrix}%
.
\end{equation*}%
Now the curvature of the connection $(\bar{\partial}_{E},h_{E})$ is $F_{(%
\bar{\partial}_{E},h_{E})}=(\bar{\partial}_{E},h_{E})\circ (\bar{\partial}%
_{E},h_{E})=\bar{\partial}_{E}\circ \partial _{E}+\partial _{E}\circ \bar{%
\partial}_{E}$. Therefore we have:%
\begin{equation*}
F_{(\bar{\partial}_{E},h_{E})}=%
\begin{pmatrix}
F_{(\bar{\partial}_{S},h_{S})}-\beta \wedge \beta ^{\ast } & \partial
_{E}\beta  \\ 
-\bar{\partial}_{E}\beta ^{\ast } & F_{(\bar{\partial}_{Q},h_{Q})}-\beta
^{\ast }\wedge \beta 
\end{pmatrix}%
.
\end{equation*}%
Now applying $\sqrt{-1}\Lambda _{\omega }$ and using the K\"{a}hler
identities we have:%
\begin{equation*}
\sqrt{-1}\Lambda _{\omega }F_{(\bar{\partial}_{E},h_{E})}=%
\begin{pmatrix}
\sqrt{-1}\Lambda _{\omega }F_{(\bar{\partial}_{S},h_{S})}-\sqrt{-1}\Lambda
_{\omega }\left( \beta \wedge \beta ^{\ast }\right)  & -\left( \bar{\partial}%
_{E}\right) ^{\ast }\beta  \\ 
-\left( \left( \bar{\partial}_{E}\right) ^{\ast }\beta \right) ^{\ast } & 
\sqrt{-1}\Lambda _{\omega }F_{(\bar{\partial}_{Q},h_{Q})}-\sqrt{-1}\Lambda
_{\omega }\left( \beta ^{\ast }\wedge \beta \right) 
\end{pmatrix}%
.
\end{equation*}%
Therefore we have:%
\begin{eqnarray*}
&&\left\Vert \sqrt{-1}\Lambda _{\omega }F_{(\bar{\partial}_{E},h_{E})}-\mu
_{\omega }(S)Id_{S}\oplus \mu _{\omega }(Q)Id_{Q}\right\Vert _{L^{\infty
}(X,\omega )} \\
&\leq &\left\Vert \sqrt{-1}\Lambda _{\omega }F_{(\bar{\partial}%
_{S},h_{S})}-\mu _{\omega }(S)Id_{S}\right\Vert _{L^{\infty }(X,\omega
)}+\left\Vert \sqrt{-1}\Lambda _{\omega }F_{(\bar{\partial}_{Q},h_{Q})}-\mu
_{\omega }(Q)Id_{Q}\right\Vert _{L^{\infty }(X,\omega )} \\
&&+2C\sup \left( \left\vert \beta \right\vert ^{2}+\left\vert \left( \bar{%
\partial}_{E}\right) ^{\ast }\beta \right\vert ^{2}\right)  \\
&=&2C\sup \left( \left\vert \beta \right\vert ^{2}+\left\vert \left( \bar{%
\partial}_{E}\right) ^{\ast }\beta \right\vert ^{2}\right) ,
\end{eqnarray*}%
where we have used that $h_{S}$ and $h_{Q}$ are Hermitian-Einstein as well
as the fact that $\limfunc{Tr}-\sqrt{-1}\Lambda _{\omega }\left( \beta
\wedge \beta ^{\ast }\right) =\left\vert \beta \right\vert ^{2}$. Now change
the holomorphic structure on $E$ by applying the complex gauge
transformation $g_{t}=t^{-1}Id_{S}\oplus tId_{Q}$, so that:%
\begin{equation*}
g_{t}(\bar{\partial}_{E})=%
\begin{pmatrix}
\bar{\partial}_{S} & t^{2}\beta  \\ 
0 & \bar{\partial}_{Q}%
\end{pmatrix}%
.
\end{equation*}%
Then we have:%
\begin{equation*}
\left\Vert \sqrt{-1}\Lambda _{\omega }F_{(g_{t}(\bar{\partial}%
_{E}),h_{E})}-\mu _{\omega }(S)Id_{S}\oplus \mu _{\omega
}(Q)Id_{Q}\right\Vert _{L^{\infty }(X,\omega )}\leq 2Ct^{4}\sup \left(
\left\vert \beta \right\vert ^{2}+\left\vert \left( \bar{\partial}%
_{E}\right) ^{\ast }\beta \right\vert ^{2}\right) 
\end{equation*}%
which goes to $0$ as $t$ goes to $0$.

In general, we will not obtain an $L^{\infty }$ approximate structure. In
the remainder of this section we show that for an arbitrary holomorphic
bundle we have such a metric for $1\leq p<\infty $. We must modify the above
approach in the general case, since the filtration is not given by
subbundles. A simple example of where this can happen is as follows.

\begin{example}
It can be shown (see \cite{OSS} page $103$) that for $k<3$ there is a
locally free representative of rank $2$ in $Ext_{\mathbb{CP}^{2}}^{1}(%
\mathcal{I}_{p},\mathcal{O}_{\mathbb{CP}^{2}}(-k))$, where $\mathcal{I}_{p}$
is the ideal sheaf of a point. In other words there is a short exact
sequence:%
\begin{equation*}
0\longrightarrow \mathcal{O}\longrightarrow E\longrightarrow \mathcal{I}%
_{p}\otimes \mathcal{O}_{\mathbb{CP}^{2}}(k)\longrightarrow 0
\end{equation*}%
where $\mathcal{O}$ is the trivial line bundle. Moreover, one can compute
that $c_{1}(E)=k$. Therefore, if we take $k<0$, then $\mu (E)<0$. Since $\mu
(\mathcal{O})=0$, the section given by $\mathcal{O}\longrightarrow E$
vanishing at $p$, is a destabilising subsheaf of $E$, so $E$ is unstable in
this case. Since $\mathcal{O}$ and $\mathcal{I}_{p}\otimes \mathcal{O}_{%
\mathbb{CP}^{2}}(k)$ are rank one and hence are stable, and the slopes are
strictly decreasing ($0=$ $\mu (\mathcal{O})>\mu (\mathcal{I}_{p}\otimes 
\mathcal{O}_{\mathbb{CP}^{2}}(k))=k$) this sequence is precisely the
Harder-Narasimhan filtration for $E$. On the other hand the quotient $%
\mathcal{I}_{p}\otimes \mathcal{O}_{\mathbb{CP}^{2}}(k)$ fails to be locally
free at the point $p$, since the ideal sheaf of a point on a complex surface
is not locally free. Generalisations of this example are given by replacing
the point $p$ in $\mathbb{CP}^{2}$ by a locally complete intersection in $%
\mathbb{CP}^{n}$ with $n>2$, or replacing $\mathbb{CP}^{2}$ by a a K\"{a}%
hler surface $X$ with $\dim H^{2}(X,\mathcal{O}_{X})=0$ for instance.
\end{example}

\begin{example}
In the above example, the only singular point of the filtration is the point 
$p$. If we blowup the point $p$, and consider $Bl_{p}$ $\mathbb{CP}^{2}=%
\widetilde{\mathbb{CP}}^{2}\overset{\pi }{\longrightarrow }\mathbb{CP}^{2}$,
then the exceptional divisor $\mathbf{E}$ in this case is just a copy of $%
\mathbb{CP}^{1}$. By construction $\pi ^{\ast }E$ is trivial over this $%
\mathbb{CP}^{1}$ and is equal to $E$ away from it. Therefore, since $E$
contains the trivial line bundle $\mathcal{O}$ as a subsheaf, $\pi ^{\ast }E$
contains as a subbundle a copy of the line bundle $\mathcal{O}(\mathbf{E})$.
Since $\mathcal{O}(\mathbf{E})=\mathcal{O}$ away from $\mathbf{E}$ and $\pi
^{\ast }\mathcal{O}=$ $\mathcal{O}$, there is an inclusion of sheaves $%
\mathcal{O}\hookrightarrow \mathcal{O}(\mathbf{E})$. Indeed, since the
quotient is supported on $\mathbf{E}$ and therefore torsion, by Lemma \ref%
{Lemma1}, $\limfunc{Sat}_{\pi ^{\ast }E}\mathcal{O=O}(\mathbf{E})$. In other
words, a single blowup of the point $p$, gives a resolution of singularities
in this case, and the filtration by subbundles of $\pi ^{\ast }E$ is given
by $\mathcal{O}(\mathbf{E})\subset \pi ^{\ast }E$. Therefore on $\widetilde{%
\mathbb{CP}}^{2}$we have an exact sequence: 
\begin{equation*}
0\longrightarrow \mathcal{O}(\mathbf{E})\longrightarrow \pi ^{\ast
}E\longrightarrow \mathcal{O}(-\mathbf{E})\otimes \pi ^{\ast }\mathcal{O}_{%
\mathbb{CP}^{2}}(k)\longrightarrow 0.
\end{equation*}
\end{example}

Therefore, in the general case we will need a more sophisticated argument to
deal with the fact that the subsheaf $S$ (and the quotient $Q$) can have
singularities. We outline our argument as follows. First we pass to a
resolution of singularities $\pi :\tilde{X}\longrightarrow X$ for the $HNS$
filtration. The blowup $\tilde{X}$ is equipped with a family of K\"{a}hler
metrics $\omega _{\varepsilon }$ as described in the previous section.
Therefore, if we fix some value $\varepsilon _{1}$, then with respect to the
metric $\omega _{\varepsilon _{1}}$ on the blowup $\tilde{X}$, we will be in
the same situation as above, when the filtration is given by subbundles.
Just as in that case, by scaling the extension classes we can produce a
metric $\tilde{h}$ with the desired property on the pullback bundle $\pi
^{\ast }E\longrightarrow \bar{X}$.

Of course this is not what we want, but we may use this metric to produce a
metric on $E$ via a cut-off argument. Namely, we first assume that the
singular set is a complex submanifold, and that the resolution of
singularities is achieved by performing one blowup operation. Then we choose
a cut-off function $\psi $ in a tubular neighbourhood of the singular set,
and fix any smooth background metric $H$ on this tubular neighbourhood. We
define the metric on $E\longrightarrow X$ by $h=\psi H+(1-\psi )\tilde{h}$.

Now we can break the estimate up into three estimates on three different
regions. We define $\psi $ so that on a smaller neighbourhood of the
singular set $h$ is equal to $H$. The desired estimate will follow on this
region by taking the radius of the neighbourhood to be arbitrarily small.
Outside of the tubular neighbourhood, $h$ is equal to $\tilde{h}$ and we can
estimate as in the case of subbundles. Finally we must also estimate in the
annulus defined by these two open sets. This can be achieved by defining $%
\psi $ to have bounds on its first and second derivatives that depend on the
reciprocal of the radius of the tubular neighbourhood and its square
respectively. The Hermitian-Einstein tensor will depend on two derivatives
of $\psi $ on the annulus, so a pointwise estimate on this quantity will
depend on this radius, but a simple argument using the fact that the
Hausdorff codimension of the singular set is a least $4$, shows we can also
obtain the appropriate estimate in this region.

Strictly speaking, we need to estimate the difference of the
Hermitian-Einstein tensor $\Lambda _{\omega }F_{h}$ of this metric with the
endomorphism $\Psi ^{HNS}(\mu _{1},\cdots ,\mu _{l})$ constructed from the
slopes obtained from the $HNS$ filtration on $E\longrightarrow X$. On the
other hand, $h$ has been constructed from $\tilde{h}$, which has been
defined so that the difference between its Hermitian-Einstein tensor $%
\Lambda _{\omega _{\varepsilon _{1}}}F_{\tilde{h}}$ and the corresponding
endomorphism coming from the filtration (by subbundles) of $\pi ^{\ast
}E\longrightarrow \bar{X}$ can be estimated on $\tilde{X}$. Because $\omega
_{\varepsilon _{1}}$ is a perfectly defined K\"{a}hler metric on $X$ away
from the singular set (which is where this estimate must be performed), one
could try to do the estimate on this region directly, as described in the
preceding paragraph, by first estimating $\Lambda _{\omega }F_{\tilde{h}}$
in terms of $\Lambda _{\omega _{\varepsilon _{1}}}F_{\tilde{h}}$ uniformly
in $\varepsilon _{1}$ and the size of the neighbourhood, but attempts to do
this were unsuccessful.

Therefore, in order to perform the estimate properly, we will need to work
on the blowup. Namely, we estimate the Hermitian-Einstein tensor for $\pi
^{\ast }h$ with respect to the family of K\"{a}hler metrics $\omega
_{\varepsilon }$. Since this metric is a pullback, it suffices to show that
we obtain estimates on the blowup that are uniform in $\varepsilon $. Then
taking the limit as $\varepsilon \rightarrow 0$ will yield an estimate with
respect to the metric $\omega $ on $X$. However, note again that the metric $%
\tilde{h}$ must be chosen at some point, and this requires fixing a value $%
\varepsilon _{1}$. Therefore, to imitate our argument above, we need to
estimate the $L^{p}$ norm of $\Lambda _{\omega _{\varepsilon }}F_{\tilde{h}}$
uniformly in $\varepsilon $ in terms of $\Lambda _{\omega _{\varepsilon
_{1}}}F_{\tilde{h}}$. Here we crucially use the fact that we are working on
the blowup. Namely, all that is required is an estimate close to the
exceptional divisor (since it is trivial on the complement of such a
neighbourhood). The fact that the exceptional divisor has only normal
crossings singularities is the key to proving that such an estimate holds.

Something very similar was done in \cite{DW1}. The author has noticed an
error in \cite{DW1} on this point. In particular, Lemma $3.14$ is slightly
incorrect. Instead, the right hand side should have an additional term
involving the $L^{2}$ norm of the full curvature. This does not essentially
disrupt the proof, because the Yang-Mills and Hermitian-Yang-Mills
functionals differ only by a topological term, but it has the effect of
changing the logic of the argument somewhat, as well as increasing the
technical complexity.

This is the reason behind most of the work done in this section. The precise
proof, given below, is a delicate balancing act between the scaling
parameter $t$, the parameter $\varepsilon _{1}$ used to define $\tilde{h}$,
the radius $R$ of the tubular neighbourhood, and the parameter $%
0<\varepsilon \leq \varepsilon _{1}$ defining the family of K\"{a}hler
metrics on $\tilde{X}$. Furthermore, the scheme explained above will only
give the correct estimate in $L^{p}$ for $p$ sufficiently close to $1$. On
the other hand, such a metric is all that is required to prove that the
Harder-Narasimhan type of the limiting sheaf $E_{\infty }$ is the same as
that of $E$. With this knowledge, it is in fact very easy to prove in turn
that $E$ has an $L^{p}$ $\delta $-approximate structure for all $1\leq
p<\infty $. This new metric depends on the value of $p$, and is in fact
given by running the Yang-Mills flow for some finite time.

We begin with a preliminary technical lemma, which will be used repeatedly
throughout this section. It will be used in conjunction with H\"{o}lder's
inequality to show that certain quantities depending a priori on $%
\varepsilon $ can in fact be estimated independently of $\varepsilon $ in
certain $L^{p}$ spaces with $p$ very close to $1$. It is the use of this
lemma that limits this particular method of constructing a $\delta $%
-approximate structure to these particular values of $p$. We use this to
prove the $L^{p}$ bound on $\Lambda _{\omega _{\varepsilon }}F$ in terms of $%
\Lambda _{\omega _{\varepsilon _{1}}}F$ for any $(1,1)$-form $F$. The
construction of the metric together with the estimate in $L^{p}$ for $p$
close to $1$ is the substance of Proposition \ref{Prop18}. We use this and
the material in Section $3.2$ to prove the statement concerning the $HN$
type of the limit. Then we quote a result about convergence of the $HN$
filtration along the flow from \cite{DW1}, and use this to prove the
existence of an $L^{p}$ structure for each $1\leq p<\infty $. Finally, at
the end of this section we do an inductive argument on the number of blowups
required to resolve singularities in order to remove the restriction we put
on the singular set. This argument actually uses the existence of an $L^{p}$
structure for $p=2$ (in the special case in which it has been proven).

\begin{lemma}
\label{Lemma11}Let $X$ be a compact K\"{a}hler manifold of dimension $n$,
and let $\pi :\tilde{X}\rightarrow X$ be a blowup along a complex
submanifold $Y$ of complex codimension $k$ where $k\geq 2$. Consider the
natural family $\omega _{\varepsilon }=\pi ^{\ast }\omega +\varepsilon \eta $
where $0<\varepsilon \leq \varepsilon _{1}$ and $\eta $ is a K\"{a}hler form
on $\tilde{X}$. Then given any $\alpha $ and $\tilde{\alpha}$ such that $%
1<\alpha <1+\frac{1}{2(k-1)},$and $\frac{\alpha }{1-2(k-1)(\alpha -1)}<%
\tilde{\alpha}<\infty $, and if we let $s=\frac{\tilde{\alpha}}{\tilde{\alpha%
}-\alpha }$ then if we write $g_{\varepsilon }$ for the K\"{a}hler metric
associated to $\omega _{\varepsilon }$, and $g_{\varpi }$ for the hermitian
metric associated to a fixed K\"{a}hler form $\varpi $ on $\tilde{X},$ we
have: $\det \left( g_{\varepsilon }^{-1}g_{\varpi }\right) \in L^{2(\alpha
-1)s}(\tilde{X},\varpi )$, and the value of the $L^{2(\alpha -1)s}$ norm is
uniformly bounded in $\varepsilon $.
\end{lemma}

\begin{proof}
Since $g_{\varepsilon }$ converges to the K\"{a}hler metric $\pi ^{\ast
}\omega $ away from the exceptional divisor $\mathbf{E}$, on the complement
of a neighbourhood of $\mathbf{E}$ there is always such a uniform bound (and
on this set $\left( \det g_{\varepsilon }/\det g_{\varpi }\right)
^{2(1-\alpha )s}$ is clearly integrable). It therefore suffices to prove the
result in a neighbourhood of the exceptional divisor. Let $y\in Y$ and $U$
be a local coordinate chart containing $y$ consisting of coordinates $%
(z_{1},\cdots ,z_{n}).$ Now $Y$ has codimension $k$ so that locally $Y$ is
given by the slice coordinates $\{z_{1}=z_{2}=\cdots =z_{k}=0\}$. Recall
that on the blow-up $\tilde{X}$ we have explicit coordinate charts $\tilde{U}%
_{m}\subset \tilde{U}=\pi ^{-1}(U)$ where $\tilde{U}_{m}=\{z\in U-Y\mid
z_{m}\neq 0\}\cup \{(z,\left[ \nu \right] )\in \mathbb{P}(\mathbf{\zeta }%
)_{\mid Y\cap U}\mid \nu _{m}\neq 0\}$, where $\mathbb{P}(\mathbf{\zeta })$
is the projectivisation of the normal bundle of $Y$. Let $(\xi _{1},\cdots
,\xi _{n})$ denote local coordinates on $\tilde{U}_{m}$. In these
coordinates the map $\pi :\tilde{X}\rightarrow X$ is given by:%
\begin{equation*}
(\xi _{1},\cdots ,\xi _{n})\longrightarrow (\xi _{1}\xi _{m},\cdots ,\xi
_{s-1}\xi _{m},\xi _{m},\xi _{m+1}\xi _{m},\cdots ,\xi _{k}\xi _{m},\xi
_{k+1},\cdots ,\xi _{n}).
\end{equation*}%
Now locally, we may write the K\"{a}hler form on $X$ in terms of the
associated metric $g$, as $\omega =\frac{i}{2}g_{ij}dz^{i}\wedge d\bar{z}%
^{j} $. Then the top power has the form: $\omega ^{n}=n!(i/2)^{n}\det
g_{ij}\ dz_{1}\wedge d\bar{z}_{1}\wedge \cdots \wedge dz_{n}\wedge d\bar{z}%
_{n}$, and using this coordinate description we may compute: $\pi ^{\ast
}\omega ^{n}=n!(i/2)^{n}\left( \pi ^{\ast }\det g_{ij}\right) \left\vert \xi
_{m}\right\vert ^{2k-2}\ d\xi _{1}\wedge d\bar{\xi}_{1}\wedge \cdots \wedge
d\xi _{n}\wedge d\bar{\xi}_{n}$.

Note that $\pi ^{\ast }\det g_{ij}$ is non-vanishing since $\det g_{ij}$ is
non-vanishing, and so degeneracy of the pullback occurs only along the
hypersurface defined by $\xi _{m}=0$. In other words, $(\xi _{1,}\cdots ,\xi
_{n})$ are normal crossings coordinates on the blow-up for the exceptional
divisor $\mathbf{E}$, and locally $\mathbf{E}$ takes the form $\{\xi
_{m}=0\} $.

The top power of the K\"{a}hler form $\omega _{\varepsilon }$ is:%
\begin{equation*}
\omega _{\varepsilon }^{n}=\pi ^{\ast }\omega ^{n}+\varepsilon n\pi ^{\ast
}\omega ^{n-1}\wedge \eta +..+\varepsilon ^{l}\binom{n}{l}\pi ^{\ast }\omega
^{n-l}\wedge \eta ^{l}+\cdots +\varepsilon ^{n-1}n\pi ^{\ast }\omega \wedge
\eta ^{n-1}+\varepsilon ^{n}\eta ^{n}.
\end{equation*}%
In the local coordinates $(\xi _{1},\cdots ,\xi _{n})$ we have: $\omega
_{\varepsilon }^{n}=n!(i/2)^{n}\det g_{ij}^{\varepsilon }d\xi _{1}\wedge d%
\bar{\xi}_{1}\wedge \cdots \wedge d\xi _{n}\wedge d\bar{\xi}_{n}$. We may
therefore obtain a lower bound (not depending on $\varepsilon $) on $\det
g_{ij}^{\varepsilon }$ as follows. Note that $\eta >0$. On the other hand,
the only degeneracy of $\pi ^{\ast }\omega $ is only on vectors tangent to
the exceptional divisor (in other words, the restriction of $\pi ^{\ast
}\omega $ vanishes on $\mathbf{E}$), so $\pi ^{\ast }\omega \geq 0$.
Therefore $\pi ^{\ast }\omega ^{l}\wedge \eta ^{n-l}$ is non-negative for
every $l$.

Then comparing the two expressions for $\omega _{\varepsilon }^{n}$, this
implies that we have the lower bound: $\det g_{ij}^{\varepsilon }\geq
C\left\vert \xi _{m}\right\vert ^{2k-2}$, where $C=\inf \pi ^{\ast }\det
g_{ij}$ on $\tilde{U}_{m}$ for each $0<\varepsilon \leq \varepsilon _{1}$.
Taking the $2(1-\alpha )s$ power of both sides we see that%
\begin{equation*}
\int_{\tilde{U}_{m}}(\det g_{\varepsilon }/\det g_{\varpi })^{2(1-\alpha
)s}\varpi ^{n}\leq C\int_{\tilde{U}_{m}}(\det g_{ij}^{\varepsilon
})^{2(1-\alpha )s}\leq C\int_{\tilde{U}_{m}}\left\vert \xi _{m}\right\vert
^{4(1-\alpha )(k-1)s},
\end{equation*}%
where the last two integrals are with respect to the standard Euclidean
measure. Using the condition on $\tilde{\alpha}$ one computes that $%
4(1-\alpha )(k-1)s>-2$ and so the functions $\left\vert \xi _{m}\right\vert
^{4(1-\alpha )s(k-1)}$, are integrable (as can be seen by computing the
integral in polar coordinates), and the result follows.
\end{proof}

\begin{lemma}
\label{Lemma12}Let $\pi :\tilde{X}\rightarrow X$, the codimension $k$, and
the family of metrics $\omega _{\varepsilon }$ be the same as in the
previous lemma. Let $\tilde{B}$ be a holomorphic vector bundle on $\tilde{X}$
and $F$ a $(1,1)$- form with values in the auxiliary vector bundle $End(%
\tilde{B})$. Let $1<\alpha <1+\frac{1}{4k(k-1)}$and $\frac{\alpha }{%
1-2(k-1)(\alpha -1)}<\tilde{\alpha}<1+\frac{1}{2(k-1)}$. Then there is a
number $\kappa _{0}$ such that for any $0<\kappa \leq \kappa _{0}$, there
exists a constant $C$ independent of $\varepsilon $, $\varepsilon _{1}$, and 
$\kappa $, and a constant $C(\kappa )$ such that:%
\begin{equation*}
\left\Vert \Lambda _{\omega _{\varepsilon }}F\right\Vert _{L^{\alpha }(%
\tilde{X},\omega _{\varepsilon })}\leq C\left( \left\Vert \Lambda _{\omega
_{\varepsilon _{1}}}F\right\Vert _{L^{\tilde{\alpha}}(\tilde{X},\omega
_{\varepsilon _{1}})}+\kappa \left\Vert F\right\Vert _{L^{2}(\tilde{X}%
,\omega _{\varepsilon _{1}})}\right) +\varepsilon _{1}C(\kappa )\left\Vert
F\right\Vert _{L^{2}(\tilde{X},\omega _{\varepsilon _{1}})}
\end{equation*}
\end{lemma}

\begin{proof}
In the following argument, out of convenience, we will engage in the slight
(and quite orthodox) abuse of notation of dividing a top degree form by the
volume form. Since the determinant line bundle of $T^{\ast }X$ is trivial,
any such form may be written as the product of some smooth function (or in
this case endomorphism) with the volume form, and so dividing by the volume
form simply returns this function (endomorphism).

Recall that $\left( \Lambda _{\omega _{\varepsilon }}F\right) \omega
_{\varepsilon }^{n}=F\wedge \omega _{\varepsilon }^{n-1}$ and $\left(
\Lambda _{\omega _{\varepsilon _{1}}}F\right) \omega _{\varepsilon
_{1}}^{n}=F\wedge \omega _{\varepsilon _{1}}^{n-1}$ so that:%
\begin{equation*}
\Lambda _{\omega _{\varepsilon }}F=\frac{F\wedge \omega _{\varepsilon }^{n-1}%
}{\omega _{\varepsilon }^{n}},\Lambda _{\omega _{\varepsilon _{1}}}F=\frac{%
F\wedge \omega _{\varepsilon _{1}}^{n-1}}{\omega _{\varepsilon _{1}}^{n}}.
\end{equation*}%
Note also that 
\begin{equation*}
\omega _{\varepsilon }^{n}=\frac{\det g^{\varepsilon }}{\det g^{\varepsilon
_{1}}}\omega _{\varepsilon _{1}}^{n}
\end{equation*}%
Now we write:%
\begin{eqnarray*}
\Lambda _{\omega _{\varepsilon }}F &=&\frac{F\wedge \omega _{\varepsilon
}^{n-1}}{\omega _{\varepsilon }^{n}}=\frac{F\wedge (\omega _{\varepsilon
_{1}}^{n-1}+\omega _{\varepsilon }^{n-1}-\omega _{\varepsilon _{1}}^{n-1})}{%
\omega _{\varepsilon }^{n}} \\
&=&\left( \frac{F\wedge \omega _{\varepsilon
_{1}}^{n-1}+\sum_{l=1}^{n-1}(\varepsilon ^{l}-\varepsilon _{1}^{l})\binom{n-1%
}{l}F\wedge \pi ^{\ast }\omega ^{(n-1)-l}\wedge \eta ^{l}}{\omega
_{\varepsilon }^{n}}\right) . \\
&=&\frac{\det g_{\varepsilon _{1}}}{\det g_{\varepsilon }}\left( \Lambda
_{\omega _{\varepsilon _{1}}}F+\frac{\sum_{l=1}^{n-1}(\varepsilon
^{l}-\varepsilon _{1}^{l})\binom{n-1}{l}F\wedge \pi ^{\ast }\omega
^{(n-1)-l}\wedge \eta ^{l}}{\omega _{\varepsilon _{1}}^{n}}\right) .
\end{eqnarray*}%
Therefore:%
\begin{equation*}
\left\vert \Lambda _{\omega _{\varepsilon }}F\right\vert ^{\alpha }\leq
C\left\vert \frac{\det g_{\varepsilon _{1}}}{\det g_{\varepsilon }}%
\right\vert ^{\alpha }\left( \left\vert \Lambda _{\omega _{\varepsilon
_{1}}}F\right\vert ^{\alpha }+\sum_{l=1}^{n-1}\left\vert \varepsilon
^{l}-\varepsilon _{1}^{l}\right\vert ^{\alpha }\left\vert \frac{F\wedge \pi
^{\ast }\omega ^{(n-1)-l}\wedge \eta ^{l}}{\omega _{\varepsilon _{1}}^{n}}%
\right\vert ^{\alpha }\right)
\end{equation*}%
(by convexity of the function $\left\vert \cdot \right\vert ^{\alpha }$ when 
$\alpha >1$). Again, we set $s=\frac{\tilde{\alpha}}{\tilde{\alpha}-\alpha }$
(note again that $s$ is a conjugate variable to $\frac{\tilde{\alpha}}{%
\alpha })$. By the above expression and H\"{o}lder's inequality with respect
to the metric $\omega _{\varepsilon _{1}}$:%
\begin{eqnarray*}
\left\Vert \Lambda _{\omega _{\varepsilon }}F\right\Vert _{L^{\alpha }(%
\tilde{X},\omega _{\varepsilon })} &=&\left( \int_{\tilde{X}}\left\vert
\Lambda _{\omega _{\varepsilon }}F\right\vert ^{\alpha }\omega _{\varepsilon
}^{n}\right) ^{\frac{1}{\alpha }}\leq \\
&&\hskip-1.4inC\left( \int_{\tilde{X}}\left( \frac{\det g_{\varepsilon }}{%
\det g_{\varepsilon _{1}}}\right) ^{(1-\alpha )s}\omega _{\varepsilon
_{1}}^{n}\right) ^{\frac{1}{\alpha s}}\left( \left( \int_{\tilde{X}%
}\left\vert \Lambda _{\omega _{\varepsilon _{1}}}F\right\vert ^{\tilde{\alpha%
}}\omega _{\varepsilon _{1}}^{n}\right) ^{\frac{1}{^{\tilde{\alpha}}}%
}+\left( \int_{\tilde{X}}\sum_{l=1}^{n-1}\left( \varepsilon _{1}^{l}\right)
^{\tilde{\alpha}}\left\vert \frac{F\wedge \pi ^{\ast }\omega
^{(n-1)-l}\wedge \eta ^{l}}{\omega _{\varepsilon _{1}}^{n}}\right\vert ^{%
\tilde{\alpha}}\omega _{\varepsilon _{1}}^{n}\right) ^{\frac{1}{\tilde{\alpha%
}}}\right) .
\end{eqnarray*}%
By the previous lemma the factor 
\begin{equation*}
\left( \int_{\tilde{X}}\left( \frac{\det g_{\varepsilon }}{\det
g_{\varepsilon _{1}}}\right) ^{(1-\alpha )s}\omega _{\varepsilon
_{1}}^{n}\right) ^{\frac{1}{\alpha s}}
\end{equation*}%
is uniformly bounded in $\varepsilon $.

Now we need to control the second term of the second factor above. We divide 
$\tilde{X}$ into two pieces: an arbitrarily small neighbourhood $V_{\kappa }$
with $\limfunc{Vol}(V_{\kappa },\omega _{\varepsilon _{1}})=\kappa ^{\frac{2%
}{2-\tilde{\alpha}}}$ of the exceptional divisor $\mathbf{E}$ and its
complement. We will perform two separate estimates, one for each piece.
Write the components of $F$ in a local basis as $F_{\rho i\bar{j}}^{\gamma
}. $ At any point we may choose an orthonormal basis for the tangent space
so that $\eta $ is standard and $\pi ^{\ast }\omega $ is diagonal. Then if
we call this basis $\{e_{i}\}$, we have%
\begin{eqnarray*}
\biggl\vert\frac{F\wedge \pi ^{\ast }\omega ^{(n-1)-l}\wedge \eta ^{l}}{%
\omega _{\varepsilon _{1}}^{n}}\biggr\vert^{2\tilde{\alpha}} &=&\biggl\vert%
\frac{\left( \sum_{i,j}F_{i\bar{j}}e_{i}\wedge \bar{e}_{j}\right) \wedge
\left( \sum_{i}\pi ^{\ast }g_{ii}e^{i}\wedge \bar{e}^{i}\right)
^{(n-1)-l}\wedge \left( \sum_{i}e^{i}\wedge \bar{e}^{i}\right) ^{l}}{\omega
_{\varepsilon _{1}}^{n}}\biggr\vert^{2\tilde{\alpha}} \\
&\leq &\frac{C}{\left\vert \omega _{\varepsilon _{1}}^{n}\right\vert ^{2%
\tilde{\alpha}}}\left( \sum_{i.j,\gamma ,\rho }\left\vert F_{\rho i\bar{j}%
}^{\gamma }\right\vert ^{2}\right) ^{\tilde{\alpha}}=C\frac{\left\vert
F\right\vert _{\eta }^{2\tilde{\alpha}}}{\left\vert \omega _{\varepsilon
_{1}}^{n}\right\vert ^{2\tilde{\alpha}}}.
\end{eqnarray*}%
Now note that on $\tilde{X}-V_{\kappa }$ the pullback $\pi ^{\ast }\omega $
determines a metric, in other words $(\pi ^{\ast }\omega )^{n}$ is
non-vanishing, so since $\omega _{\varepsilon _{1}}^{n}\longrightarrow (\pi
^{\ast }\omega )^{n}$, the quantity $\left\vert \omega _{\varepsilon
_{1}}^{n}\right\vert ^{2\tilde{\alpha}}$ is uniformly bounded away from $0$.
Therefore 
\begin{equation*}
\left\vert \frac{F\wedge \pi ^{\ast }\omega ^{(n-1)-l}\wedge \eta ^{l}}{%
\omega _{\varepsilon _{1}}^{n}}\right\vert ^{\tilde{\alpha}}\leq C\left\vert
F\right\vert _{\eta }^{\tilde{\alpha}}.
\end{equation*}%
On the other hand, if we again choose a basis for which $\eta $ is standard
and such that $\omega _{\varepsilon _{1}}$ is diagonal, we have: 
\begin{equation*}
\left\vert F\right\vert _{\eta }^{2\tilde{\alpha}}=\left\vert \left(
\sum_{i.j,\gamma ,\rho }\left\vert F_{\rho ij}^{\gamma }\right\vert
^{2}\right) \right\vert ^{\tilde{\alpha}}\leq C\left\vert \left(
\sum_{i.j,\gamma ,\rho }\frac{1}{g_{ii}^{\varepsilon
_{1}}g_{jj}^{\varepsilon _{1}}}\left\vert F_{\rho ij}^{\gamma }\right\vert
^{2}\right) \right\vert ^{\tilde{\alpha}}=C\left\vert F\right\vert _{\omega
_{\varepsilon _{1}}}^{2\tilde{\alpha}}
\end{equation*}%
since the product of the eigenvalues $g_{ii}^{\varepsilon
_{1}}g_{jj}^{\varepsilon _{1}}$ is again uniformly bounded ($%
g_{ii}^{\varepsilon _{1}}g_{jj}^{\varepsilon _{1}}\rightarrow \pi ^{\ast
}g_{ii}\pi ^{\ast }g_{jj}$ as $\varepsilon _{1}\rightarrow 0$). Thus, on $%
\tilde{X}-V_{\kappa }$ we have the further pointwise bound: $\left\vert
F\right\vert _{\eta }^{\tilde{\alpha}}\leq C\left\vert F\right\vert _{\omega
_{\varepsilon _{1}}}^{\tilde{\alpha}}$. Therefore the integral on $\tilde{X}%
-V_{\kappa }$ is:%
\begin{eqnarray*}
\left( \int_{\tilde{X}-V_{\kappa }}\left( \varepsilon _{1}^{l}\right) ^{%
\tilde{\alpha}}\left\vert \frac{F\wedge \pi ^{\ast }\omega ^{(n-1)-l}\wedge
\eta ^{l}}{\omega _{\varepsilon _{1}}^{n}}\right\vert ^{\tilde{\alpha}%
}\omega _{\varepsilon _{1}}^{n}\right) ^{\frac{1}{\tilde{\alpha}}} &\leq
&C\varepsilon _{1}\left( \int_{\tilde{X}-V_{\kappa }}\left\vert F\right\vert
_{\omega _{\varepsilon _{1}}}^{\tilde{\alpha}}\omega _{\varepsilon
_{1}}^{n}\right) ^{\frac{1}{\tilde{\alpha}}} \\
&\leq &C\varepsilon _{1}\left\Vert F\right\Vert _{L^{\tilde{\alpha}}(\omega
_{\varepsilon _{1}})}\leq C(\kappa )\varepsilon _{1}\left\Vert F\right\Vert
_{L^{2}(\omega _{\varepsilon _{1}})}
\end{eqnarray*}%
since by assumption $\tilde{\alpha}<2.$

Now we estimate this term on $V_{\kappa }$. Choose an orthonormal basis for
the tangent space at a point in $V_{\kappa }$ such that $\omega
_{\varepsilon _{1}}$ is standard and $\eta $ is diagonal. Then we have $%
g_{ij}^{\varepsilon _{1}}=\pi ^{\ast }g_{ij}+\varepsilon _{1}\eta _{ij}$, so
if $i\neq j$, $\pi ^{\ast }g_{ij}=0$, and if $i=j$, $\eta _{ii}=$ $\frac{1-%
\tilde{g}_{ii}}{\varepsilon _{1}}$. Note also that $0\leq \tilde{g}_{ii}<1$
since $0<\eta _{ii}$. If we write $\Omega $ for the standard Euclidean
volume form then:%
\begin{eqnarray*}
&&\sum_{l=1}^{n-1}\left( \varepsilon _{1}^{l}\right) ^{\tilde{\alpha}%
}\left\vert \frac{F\wedge \pi ^{\ast }\omega ^{(n-1)-l}\wedge \eta ^{l}}{%
\omega _{\varepsilon _{1}}^{n}}\right\vert ^{\tilde{\alpha}} \\
&=&\sum_{l=1}^{n-1}\left\vert \frac{\left( \sum_{i,j}F_{i\bar{j}}e_{i}\wedge 
\bar{e}_{j}\right) \wedge \left( \sum_{i}\pi ^{\ast }g_{i\overline{i}%
}e^{i}\wedge \bar{e}^{i}\right) ^{(n-1)-l}\wedge \left( \sum_{i}\left( 1-\pi
^{\ast }g_{ii}\right) e^{i}\wedge \bar{e}^{i}\right) ^{l}}{\Omega }%
\right\vert ^{\tilde{\alpha}} \\
&\leq &C\left( \sum_{i.j,\gamma ,\rho }\left\vert F_{\rho i\bar{j}}^{\gamma
}\right\vert \right) ^{\tilde{\alpha}}\leq C\left\vert F\right\vert _{\omega
_{\varepsilon _{1}}}^{\tilde{\alpha}}.
\end{eqnarray*}%
Therefore:%
\begin{eqnarray*}
&&\left( \int_{V_{\kappa }}\sum_{l=1}^{n-1}\left( \varepsilon
_{1}^{l}\right) ^{\tilde{\alpha}}\left\vert \frac{F\wedge \pi ^{\ast }\omega
^{(n-1)-l}\wedge \eta ^{l}}{\omega _{\varepsilon _{1}}^{n}}\right\vert ^{%
\tilde{\alpha}}\omega _{\varepsilon _{1}}^{n}\right) ^{\frac{1}{\tilde{\alpha%
}}} \\
&\leq &C\left( \int_{V_{\kappa }}\left\vert F\right\vert _{\omega
_{\varepsilon _{1}}}^{\tilde{\alpha}}\omega _{\varepsilon _{1}}^{n}\right) ^{%
\frac{1}{\tilde{\alpha}}}\leq C\limfunc{Vol}(V_{\kappa },\omega
_{\varepsilon _{1}})^{1-\frac{\tilde{\alpha}}{2}}\left\Vert F\right\Vert
_{L^{2}(V_{\kappa },\omega _{\varepsilon _{1}})}\leq C\kappa \left\Vert
F\right\Vert _{L^{2}(\tilde{X},\omega _{\varepsilon _{1}})}\text{ }(H\ddot{o}%
lder).
\end{eqnarray*}%
Now we obtain the desired estimate:%
\begin{equation*}
\left\Vert \Lambda _{\omega _{\varepsilon }}F\right\Vert _{L^{\alpha }(%
\tilde{X},\omega _{\varepsilon })}\leq C\left( \left\Vert \Lambda _{\omega
_{\varepsilon _{1}}}F\right\Vert _{L^{\tilde{\alpha}}(\tilde{X},\omega
_{\varepsilon _{1}})}+\kappa \left\Vert F\right\Vert _{L^{2}(\tilde{X}%
,\omega _{\varepsilon _{1}})}\right) +\varepsilon _{1}C(\kappa )\left\Vert
F\right\Vert _{L^{2}(\tilde{X},\omega _{\varepsilon _{1}})}.
\end{equation*}%
\ 
\end{proof}

\begin{proposition}
\label{Prop18}Let $E\rightarrow X$ be a holomorphic vector bundle of rank $K$
over a compact K\"{a}hler manifold with K\"{a}hler form $\omega $. Assume
that $E$ has Harder-Narasimhan type $\mu =(\mu _{1},\cdots ,\mu _{K})$ that
the singular set $Z_{\limfunc{alg}}$ of the $HNS$ filtration is smooth, and
furthermore that blowing up along the singular set resolves the
singularities of the HNS filtration. There is an $\alpha _{0}>1$ such that
the following holds: given any $\delta >0$ and any $N$, there is an
hermitian metric $h$ on $E$ such that $HYM_{\alpha ,N}^{\omega }(\bar{%
\partial}_{E},h)\leq HYM_{\alpha ,N}(\mu )+\delta $, for all $1\leq \alpha
<\alpha _{0}$.
\end{proposition}

\begin{proof}
As before, let $\pi :\tilde{X}\rightarrow X$ be a blow-up along a smooth,
complex submanifold $Y$, and we assume that this resolves the singularities
of the $HNS$ filtration. In other words there is a filtration of $\tilde{E}%
=\pi ^{\ast }E$ on $\tilde{X}$ that is given by sub-bundles and is equal to
the $HNS$ filtration of $E$ away from the divisor $\mathbf{E}$. Let $\omega
_{\varepsilon }$ denote the aforementioned family of K\"{a}hler metrics on $%
\tilde{X}$ given by $\omega _{\varepsilon }=\pi ^{\ast }\omega +\varepsilon
\eta $ where $0<\varepsilon \leq 1$ and $\eta $ is a fixed K\"{a}hler metric
on $\tilde{X}$. We will construct the metric $h$ on $E$ from an hermitian
metric $\tilde{h}$ on $\pi ^{\ast }E$ to be specified later.

Since $Z_{\limfunc{alg}}$ is a complex submanifold, we consider its normal
bundle $\boldsymbol{\zeta }$, or more particularly the open subset: $\mathbf{%
\zeta }_{R}=\{(x,\nu )\in \mathbf{\zeta }\mid \left\vert \nu \right\vert
<R\} $. By the tubular neighbourhood theorem, for $R$ sufficiently small
this set is diffeomorphic to an open neighbourhood $U_{R}$ of $Z_{\limfunc{%
alg}}$. We choose a background metric $H$ on this open set.

Let $\psi $ be a smooth cut-off function supported in $U_{1}$ and and
identically $1$ on $U_{1/2}$ and such that $0\leq \psi \leq 1$ everywhere.
Then if we define $\psi _{R}(x,\nu )=\psi (x,\frac{\nu }{R})$, $\psi _{R}$
is identically $1$ on $U_{R/2}$ and supported in $U_{R}$ with $0\leq \psi
_{R}\leq 1$ and furthermore there are bounds on the derivatives: 
\begin{equation*}
\left\vert \frac{\partial \psi _{R}}{\partial z_{i}}\right\vert \leq \frac{C%
}{R}\qquad ,\qquad \left\vert \frac{\partial }{\partial \bar{z}_{i}}\frac{%
\partial \psi _{R}}{\partial z_{i}}\right\vert \leq \frac{C}{R^{2}}
\end{equation*}%
where the constant $C$ does not depend on $R$. Suppose for the moment that
we have constructed an hermitian metric $\tilde{h}$ on $\pi ^{\ast }E$. If
we continue to denote by $H$ and $\psi _{R}$ their pullbacks to $\tilde{X}$,
then we may define the following metric on $\pi ^{\ast }E:$%
\begin{equation*}
h_{\psi _{R}}:=\psi _{R}H+(1-\psi _{R})\tilde{h}
\end{equation*}
Observe that on $X-U_{R}$ we have $h_{\psi _{R}}=\tilde{h}$ and on $U_{R/2}$%
, $h_{\psi _{R}}=H$.

Now we need to estimate the difference:%
\begin{eqnarray*}
&&\left\vert HYM_{\alpha ,N}^{\omega _{\varepsilon }}(\bar{\partial}_{\tilde{%
E}},h_{\psi _{R}})-HYM_{\alpha ,N}(\mu )\right\vert  \\
&=&\left\vert \int_{\tilde{X}}\Phi _{\alpha }(\Lambda _{\omega _{\varepsilon
}}F_{h_{\psi _{R}}}+\sqrt{-1}N\mathbf{I}_{\tilde{E}})-\Phi _{\alpha }\left( 
\sqrt{-1}(\mu _{1}+N),\cdots ,\sqrt{-1}(\mu _{K}+N)\right) \right\vert 
\end{eqnarray*}%
where $\Phi _{\alpha }$ is the convex functional on $\mathfrak{u}(\tilde{E})$
given as in Section $3.2$ by $\Phi _{\alpha }(a)=\sum_{j=1}^{k}\left\vert
\lambda _{j}\right\vert ^{\alpha }$, where the $\sqrt{-1}\lambda _{j}$ are
the eigenvalues of $a$. From here on out we will write $\sqrt{-1}(\mu +N)$
in place of $(\sqrt{-1}(\mu _{1}+N),\cdots ,\sqrt{-1}(\mu _{K}+N))$.
Therefore we have:%
\begin{align*}
& \qquad \qquad \bigl\vert HYM_{\alpha ,N}^{\omega _{\varepsilon }}(\bar{%
\partial}_{\tilde{E}},h_{\psi _{R}})-HYM_{\alpha ,N}(\mu )\bigr\vert \\
& \leq \biggl\vert\int_{\tilde{X}-\pi ^{-1}(U_{R/2})}\hskip-.25in\Phi
_{\alpha }(\Lambda _{\omega _{\varepsilon }}F_{h_{\psi _{R}}}+\sqrt{-1}N%
\mathbf{I}_{\tilde{E}})-\Phi _{\alpha }(\sqrt{-1}(\mu +N))\biggr\vert \\
& +\biggl\vert\int_{\pi ^{-1}(U_{R/2})}\hskip-.25in\Phi _{\alpha }(\Lambda
_{\omega _{\varepsilon }}F_{h_{\psi _{R}}}+\sqrt{-1}N\mathbf{I}_{\tilde{E}%
})-\Phi _{\alpha }(\sqrt{-1}(\mu +N))\biggr\vert \\
& =\left\vert \int_{\tilde{X}-\pi ^{-1}(U_{R/2})}\hskip-.25in\Phi _{\alpha
}(\Lambda _{\omega _{\varepsilon }}F_{h_{\psi _{R}}}+\sqrt{-1}N\mathbf{I}_{%
\tilde{E}})-\Phi _{\alpha }(\sqrt{-1}(\mu +N))\right\vert  \\
& +\left\vert \int_{\pi ^{-1}(U_{R/2})}\hskip-.25in\Phi _{\alpha }(\Lambda
_{\omega _{\varepsilon }}F_{H}+\sqrt{-1}N\mathbf{I}_{\tilde{E}})-\Phi
_{\alpha }(\sqrt{-1}(\mu +N))\right\vert 
\end{align*}%
where the last equality comes from the fact that $h_{\psi _{R}}$ is equal to 
$H$ on $U_{R/2}$. Dividing the first integral further we have:%
\begin{eqnarray*}
&&\left\vert HYM_{\alpha ,N}^{\omega _{\varepsilon }}(\bar{\partial}%
_{E},h_{\psi _{R}})-HYM_{\alpha ,N}(\mu )\right\vert  \\
&\leq &\left\vert \int_{\pi ^{-1}(U_{R}-U_{R/2})}\Phi _{\alpha }(\Lambda
_{\omega _{\varepsilon }}F_{h_{\psi _{R}}}+\sqrt{-1}N\mathbf{I}_{\tilde{E}%
})-\Phi _{\alpha }(\Lambda _{\omega _{\varepsilon }}F_{\tilde{h}}+\sqrt{-1}N%
\mathbf{I}_{\tilde{E}})\right\vert  \\
&&+\left\vert \int_{\tilde{X}-\pi ^{-1}(U_{R/2})}\Phi _{\alpha }(\Lambda
_{\omega _{\varepsilon }}F_{\tilde{h}}+\sqrt{-1}N\mathbf{I}_{\tilde{E}%
})-\Phi _{\alpha }(\sqrt{-1}(\mu _{\omega _{\varepsilon
_{1}}}+N))\right\vert  \\
&&+\left\vert \int_{\tilde{X}-\pi ^{-1}(U_{R/2})}\Phi _{\alpha }(\sqrt{-1}%
(\mu _{\omega _{\varepsilon _{1}}}+N))-\Phi _{\alpha }(\sqrt{-1}(\mu
+N))\right\vert  \\
&&+\left\vert \int_{\pi ^{-1}(U_{R/2})}\Phi _{\alpha }(\Lambda _{\omega
_{\varepsilon }}F_{H}+\sqrt{-1}N\mathbf{I}_{\tilde{E}})-\Phi _{\alpha }(%
\sqrt{-1}(\mu +N))\right\vert 
\end{eqnarray*}%
where in the first integral on the right hand side we have used the fact
that outside of $U_{R}$ the metrics $h_{\psi _{R}}$ and $\tilde{h}$ agree.
Here, $\mu _{\omega _{\varepsilon _{1}}}$ denotes the usual $K$-tuple of
rational numbers made from the $\omega _{\varepsilon _{1}}$ slopes of the
quotients of the resolution.

Recall that the norm on $L^{\alpha }(\mathfrak{u}(\tilde{E}))$, $a\mapsto %
\bigl(\int_{M}\Phi _{\alpha }(a)\bigr)^{1/\alpha }$ is equivalent to the $%
L^{\alpha }$ norm and so there is a universal constant $C$ independent of $R$
and $\varepsilon $ such that:%
\begin{eqnarray*}
&&\left\vert \int_{\pi ^{-1}(U_{R}-U_{R/2})}\Phi _{\alpha }(\Lambda _{\omega
_{\varepsilon }}F_{h_{\psi _{R}}}+\sqrt{-1}N\mathbf{I}_{\tilde{E}})-\Phi
_{\alpha }(\Lambda _{\omega _{\varepsilon }}F_{\tilde{h}}+\sqrt{-1}N\mathbf{I%
}_{E})\right\vert \\
&&+\left\vert \int_{\tilde{X}-\pi ^{-1}(U_{R/2})}\Phi _{\alpha }(\Lambda
_{\omega _{\varepsilon }}F_{\tilde{h}}+\sqrt{-1}N\mathbf{I}_{E})-\Phi
_{\alpha }(\sqrt{-1}(\mu _{\omega _{\varepsilon _{1}}}+N))\right\vert \\
&\leq &C\left( \left\Vert \Lambda _{\omega _{\varepsilon }}F_{h_{\psi
_{R}}}-\Lambda _{\omega _{\varepsilon }}F_{\tilde{h}}\right\Vert _{L^{\alpha
}(\pi ^{-1}(U_{R}-U_{R/2}),\omega _{\varepsilon })}^{\alpha }+\left\Vert
\Lambda _{\omega _{\varepsilon }}F_{\tilde{h}}-\sqrt{-1}\mu _{\omega
_{\varepsilon _{1}}}\right\Vert _{L^{\alpha }(\tilde{X}-\pi
^{-1}(U_{R/2}),\omega _{\varepsilon })}^{\alpha }\right) .
\end{eqnarray*}

First we dispose of%
\begin{equation*}
\left\vert \int_{\tilde{X}-\pi ^{-1}(U_{R/2})}\Phi _{\alpha }(\sqrt{-1}(\mu
_{\omega _{\varepsilon _{1}}}+N))-\Phi _{\alpha }(\sqrt{-1}(\mu
+N))\right\vert
\end{equation*}%
by choosing $\varepsilon _{1}$ close to zero and using $\limfunc{Remark}\ref%
{R6}$. That is, we may choose $\varepsilon _{1}$ small enough so that%
\begin{equation*}
\left\vert \int_{\tilde{X}-\pi ^{-1}(U_{R/2})}\Phi _{\alpha }(\sqrt{-1}(\mu
_{\omega _{\varepsilon _{1}}}+N))-\Phi _{\alpha }(\sqrt{-1}(\mu
+N))\right\vert <\frac{\delta }{2}
\end{equation*}%
Next will will bound: 
\begin{equation*}
\left\Vert \Lambda _{\omega _{\varepsilon }}F_{\tilde{h}}-\sqrt{-1}\mu
_{\omega _{\varepsilon _{1}}}Id_{\tilde{E}}\right\Vert _{L^{\alpha }(\tilde{X%
}-\pi ^{-1}(U_{R/2}),\omega _{\varepsilon })}^{\alpha }.
\end{equation*}

Note that at this point we have not specified the metric $\tilde{h}$ on $\pi
^{\ast }E$. We will do so now. Each of the $\omega $-stable quotients $Q_{i}$
of the Harder-Narasimhan-Seshadri filtration remains stable on the blowup
with respect to the metrics $\omega _{\varepsilon }$ with $\varepsilon $
sufficiently small (see $\limfunc{Remark}\ref{Rmk5}$), so that the quotients 
$\tilde{Q}_{i}$ are also $\omega _{\varepsilon _{1}}$-stable and admit a
unique Hermitian-Einstein metric $\tilde{G}_{i}^{\varepsilon _{1}}$. The
prototype for our metric $\tilde{h}$ will be the metric $\tilde{G}%
_{\varepsilon _{1}}={\LARGE \oplus }_{i}\tilde{G}_{i}^{\varepsilon _{1}}$.
Just as in the beginning of this section, we need to modify $\tilde{G}%
_{\varepsilon _{1}}$ by a gauge transformation in order to obtain the
appropriate bound. More precisely, since holomorphic structures on the
bundle $\tilde{E}$ are equivalent to integrable unitary connections, this is
the same as showing that if we fix the metric $\tilde{G}_{\varepsilon _{1}}$%
, there is a complex gauge transformation $\tilde{g}$ of $\tilde{E}$ such
that $\bigl\Vert\Lambda _{\omega _{\varepsilon }}F_{(\tilde{g}(\bar{\partial}%
_{\tilde{E}}),\tilde{G}_{\varepsilon _{1}})}-\sqrt{-1}\mu _{\omega
_{\varepsilon _{1}}}Id_{\tilde{E}}\bigr\Vert_{L^{\alpha }(\tilde{X}-\pi
^{-1}(U_{R/2}),\omega _{\varepsilon })}$ is small. When we take the direct
sum, the second fundamental form enters into the curvature and so we ask
that there is a gauge transformation making this contribution small.

We will do this iteratively. If we write $\tilde{S}=\tilde{E}_{1}=\tilde{Q}%
_{1}\subset \pi ^{\ast }E$ for the first sub-bundle in the filtration of $%
\pi ^{\ast }E$ $\ $on $\tilde{X}$, then the discussion at the beginning of
this section applies in exactly the same way to the exact sequence:%
\begin{equation*}
0\longrightarrow \tilde{S}\longrightarrow \pi ^{\ast }E\longrightarrow 
\tilde{Q}\longrightarrow 0
\end{equation*}%
where $\tilde{Q}=\oplus _{i=2}^{l}\tilde{Q}_{i}$. Therefore if we apply the
gauge transformation $g_{t}=$ $t^{-1}Id_{\tilde{S}}\oplus tId_{\tilde{Q}}$
to the operator $\bar{\partial}_{\tilde{E}}$ as before, the curvature may be
written as:%
\begin{equation*}
F_{\left( g_{t}(\bar{\partial}_{\tilde{E}}),\tilde{G}_{\varepsilon
_{1}}\right) }=%
\begin{pmatrix}
F_{(\bar{\partial}_{\tilde{S}},\tilde{G}_{1}^{\varepsilon _{1}})}-t^{4}\beta
\wedge \beta _{\tilde{S}}^{\ast } & t^{2}\partial _{E}\beta _{\tilde{S}} \\ 
-t^{2}\bar{\partial}_{E}\beta _{\tilde{S}}^{\ast } & F_{(\bar{\partial}_{Q},%
{\LARGE \oplus }_{i=2}^{l}\tilde{G}_{i}^{\varepsilon _{1}})}-t^{4}\beta _{%
\tilde{S}}^{\ast }\wedge \beta _{\tilde{S}}%
\end{pmatrix}%
.
\end{equation*}%
Taking $\Lambda _{\omega _{\varepsilon }}$, we obtain terms involving:%
\begin{eqnarray*}
t^{4}\Lambda _{\omega _{\varepsilon }}\beta \wedge \beta _{\tilde{S}}^{\ast
} &=&t^{4}\frac{\beta \wedge \beta _{\tilde{S}}^{\ast }\wedge \omega
_{\varepsilon }^{n-1}}{\omega _{\varepsilon }^{n}},t^{2}\Lambda _{\omega
_{\varepsilon }}\partial _{E}\beta _{\tilde{S}}=t^{2}\frac{\partial
_{E}\beta _{\tilde{S}}\wedge \omega _{\varepsilon }^{n-1}}{\omega
_{\varepsilon }^{n}}, \\
t^{2}\Lambda _{\omega _{\varepsilon }}\bar{\partial}_{E}\beta _{\tilde{S}%
}^{\ast } &=&t^{2}\frac{\bar{\partial}_{E}\beta _{\tilde{S}}^{\ast }\wedge
\omega _{\varepsilon }^{n-1}}{\omega _{\varepsilon }^{n}},t^{4}\Lambda
_{\omega _{\varepsilon }}\beta _{\tilde{S}}^{\ast }\wedge \beta _{\tilde{S}%
}=t^{4}\frac{\beta _{\tilde{S}}^{\ast }\wedge \beta _{\tilde{S}}\wedge
\omega _{\varepsilon }^{n-1}}{\omega _{\varepsilon }^{n}}.
\end{eqnarray*}%
Recalling also that 
\begin{equation*}
\omega _{\varepsilon }^{n}=\left\vert \frac{\det g_{\varepsilon }}{\det
g_{\eta }}\right\vert \eta ^{n},
\end{equation*}%
and applying H\"{o}lder's inequality we see that%
\begin{eqnarray*}
&&\bigl\Vert\Lambda _{\omega _{\varepsilon }}F_{(\tilde{g}_{t}(\bar{\partial}%
_{\tilde{E}}),\tilde{G}_{\varepsilon _{1}})}-\sqrt{-1}\mu _{\omega
_{\varepsilon _{1}}}Id_{\tilde{E}}\bigr\Vert_{L^{\alpha }(\tilde{X}-\pi
^{-1}(U_{R/2}),\omega _{\varepsilon })} \\
&\leq &\bigl\Vert\Lambda _{\omega _{\varepsilon }}F_{_{\tilde{G}%
_{1}^{\varepsilon _{1}}}}-\sqrt{-1}\mu _{\omega _{\varepsilon _{1}}}(\tilde{S%
})Id_{\tilde{Q}_{1}}\bigr\Vert_{L^{\alpha }(\tilde{X}-\pi
^{-1}(U_{R/2}),\omega _{\varepsilon })} \\
&&+\bigl\Vert\Lambda _{\omega _{\varepsilon }}F_{_{{\LARGE \oplus }_{i=2}^{l}%
\tilde{G}_{i}^{\varepsilon _{1}}}}-\sqrt{-1}\oplus _{i=2}^{l}\mu _{\omega
_{\varepsilon _{1}}}(\tilde{Q}_{i})Id_{\tilde{Q}i}\bigr\Vert_{L^{\alpha }(%
\tilde{X}-\pi ^{-1}(U_{R/2}),\omega _{\varepsilon })} \\
&&\hskip-.25in+\left( \int_{\tilde{X}-\pi ^{-1}(U_{R/2})}\left\vert \frac{%
\det g_{\varepsilon }}{\det g_{\eta }}\right\vert ^{(1-\alpha )s}\eta
^{n}\right) ^{\frac{1}{\alpha s}}\left( \int_{\tilde{X}-\pi
^{-1}(U_{R/2})}\left( \left\vert t^{4}\frac{\beta \wedge \beta _{\tilde{S}%
}^{\ast }\wedge \omega _{\varepsilon }^{n-1}}{\eta ^{n}}\right\vert ^{\tilde{%
\alpha}}+\left\vert t^{2}\frac{\partial _{E}\beta _{\tilde{S}}\wedge \omega
_{\varepsilon }^{n-1}}{\eta ^{n}}\right\vert ^{\tilde{\alpha}}\right) \eta
^{n}\right) ^{\frac{1}{\tilde{\alpha}}} \\
&&\hskip-.25in+\left( \int_{\tilde{X}-\pi ^{-1}(U_{R/2})}\left\vert \frac{%
\det g_{\varepsilon }}{\det g_{\eta }}\right\vert ^{(1-\alpha )s}\eta
^{n}\right) ^{\frac{1}{\alpha s}}\left( \int_{\tilde{X}-\pi
^{-1}(U_{R/2})}\left( \left\vert t^{2}\frac{\bar{\partial}_{E}\beta _{\tilde{%
S}}^{\ast }\wedge \omega _{\varepsilon }^{n-1}}{\eta ^{n}}\right\vert ^{%
\tilde{\alpha}}+\left\vert t^{4}\frac{\beta _{\tilde{S}}^{\ast }\wedge \beta
_{\tilde{S}}\wedge \omega _{\varepsilon }^{n-1}}{\eta ^{n}}\right\vert ^{%
\tilde{\alpha}}\right) \eta ^{n}\right) ^{\frac{1}{\tilde{\alpha}}}
\end{eqnarray*}%
where $\tilde{\alpha}$ and $s$ are as in Lemma \ref{Lemma11} (recall that $s$
and $\frac{\tilde{\alpha}}{\alpha }$ are a conjugate pair). By the lemma,
the last two terms above are bounded uniformly in $\varepsilon $. Therefore,
the contribution of these terms can be made small by making $t$ sufficiently
small.

Similarly, we can apply the same argument to the extensions:%
\begin{equation*}
0\longrightarrow \tilde{E}_{i}/\tilde{E}_{i-1}=\tilde{Q}_{i}\longrightarrow
\pi ^{\ast }E/\tilde{E}_{i-1}\longrightarrow \pi ^{\ast }E/\tilde{E}%
_{i}=\oplus _{j=i+1}^{l}\tilde{Q}_{j}\longrightarrow 0
\end{equation*}%
using the gauge transformation $g_{t}=Id_{\tilde{Q}_{1}}\oplus \cdots \oplus
Id_{\tilde{Q}_{l-1}}\oplus t^{-1}Id_{\tilde{Q}_{i}}\oplus _{j=i+1}^{l}tId_{%
\tilde{Q}_{j}}$. Such an extension will give a further second fundamental
form $\beta _{\tilde{Q}_{i}}$, and its contribution can be estimated in
exactly the same way as above.

Continuing in this way, we see that there is a complex gauge transformation $%
g$ of $\pi ^{\ast }E$ such that: 
\begin{align*}
\bigl\Vert\Lambda _{\omega _{\varepsilon }}F_{(\tilde{g}(\bar{\partial}_{%
\tilde{E}}),\tilde{G}_{\varepsilon _{1}})}-\sqrt{-1}\mu _{\omega
_{\varepsilon _{1}}}& Id_{\tilde{E}}\bigr\Vert_{L^{\alpha }(\tilde{X}-\pi
^{-1}(U_{R/2}),\omega _{\varepsilon })}\leq \bigl\Vert\Lambda _{\omega
_{\varepsilon }}F_{_{\tilde{G}_{1}^{\varepsilon _{1}}}}-\sqrt{-1}\mu
_{\omega _{\varepsilon _{1}}}(\tilde{Q}_{1})Id_{\tilde{Q}_{1}}\bigr\Vert%
_{L^{\alpha }(\tilde{X}-\pi ^{-1}(U_{R/2}),\omega _{\varepsilon })} \\
& +\cdots +\bigl\Vert\Lambda _{\omega _{\varepsilon }}F_{_{\tilde{G}%
_{l}^{\varepsilon _{1}}}}-\sqrt{-1}\mu _{\omega _{\varepsilon _{1}}}(\tilde{Q%
}_{l})Id_{\tilde{Q}_{l}}\bigr\Vert_{_{L^{\alpha }(\tilde{X}-\pi
^{-1}(U_{R/2}),\omega _{\varepsilon })}}+\Theta (t,\beta _{\tilde{Q}%
_{1}},\cdots ,\beta _{\tilde{Q}_{l}})
\end{align*}%
where $\Theta (t,\beta _{\tilde{Q}_{1}},\cdots ,\beta _{\tilde{Q}%
_{l}})\rightarrow 0$ as $t\rightarrow 0$. Therefore we have reduced this
estimate to an estimate on each of the terms:%
\begin{equation*}
\left\Vert \Lambda _{\omega _{\varepsilon }}F_{_{\tilde{G}_{i}^{\varepsilon
_{1}}}}-\sqrt{-1}\mu _{\omega _{\varepsilon _{1}}}(\tilde{Q}_{i})Id_{\tilde{Q%
}_{i}}\right\Vert _{_{L^{\alpha }(\tilde{X}-\pi ^{-1}(U_{R/2}),\omega
_{\varepsilon })}}.
\end{equation*}%
On the other hand we have:%
\begin{align*}
& \left\Vert \Lambda _{\omega _{\varepsilon }}F_{\tilde{G}_{i}^{\varepsilon
_{1}}}-\sqrt{-1}\mu _{\omega _{\varepsilon _{1}}}(\tilde{Q}_{i})Id_{\tilde{Q}%
_{i}}\right\Vert _{L^{\alpha }(\tilde{X}-\pi ^{-1}(U_{R/2}),\omega
_{\varepsilon })} \\
& \leq \left\Vert \Lambda _{\omega _{\varepsilon }}\left( F_{\tilde{G}%
_{i}^{\varepsilon _{1}}}-\frac{\sqrt{-1}}{n}\omega _{\varepsilon _{1}}\mu
_{\omega _{\varepsilon _{1}}}(\tilde{Q}_{i})Id_{\tilde{Q}_{i}}\right)
\right\Vert _{L^{\alpha }(\tilde{X}-\pi ^{-1}(U_{R/2}),\omega _{\varepsilon
})} \\
& +\left\Vert \frac{\sqrt{-1}}{n}\Lambda _{\omega _{\varepsilon }}(\omega
_{\varepsilon _{1}}-\omega _{\varepsilon })\mu _{\omega _{\varepsilon _{1}}}(%
\tilde{Q}_{i})Id_{\tilde{Q}_{i}}\right\Vert _{L^{\alpha }(\tilde{X}-\pi
^{-1}(U_{R/2}),\omega _{\varepsilon })}
\end{align*}%
where we have used the fact that $\Lambda _{\omega _{\varepsilon }}\omega
_{\varepsilon }=n$. Now by Lemma \ref{Lemma12} we have: 
\begin{align*}
\biggl\Vert\Lambda _{\omega _{\varepsilon }}\bigl(F_{_{\tilde{G}%
_{i}^{\varepsilon _{1}}}}-\frac{\sqrt{-1}}{n}\omega _{\varepsilon _{1}}\mu
_{\omega _{\varepsilon _{1}}}(\tilde{Q}_{i})Id_{\tilde{Q}_{i}}\bigr)& %
\biggr\Vert_{L^{\alpha }(\tilde{X}-\pi ^{-1}(U_{R/2}),\omega _{\varepsilon
})} \\
& \leq C\left( \left\Vert \Lambda _{\omega _{\varepsilon _{1}}}F_{\tilde{G}%
_{i}^{\varepsilon _{1}}}-\sqrt{-1}\mu _{\omega _{\varepsilon _{1}}}(\tilde{Q}%
)Id_{\tilde{Q}_{i}}\right\Vert _{L^{\tilde{\alpha}}(\tilde{X},\omega
_{\varepsilon _{1}})}\right)  \\
& +\kappa C\left( \left\Vert F_{\tilde{G}_{i}^{\varepsilon _{1}}}\right\Vert
_{L^{2}(\tilde{X},\omega _{\varepsilon _{1}})}+\frac{1}{n}\left\Vert \omega
_{\varepsilon _{1}}\mu _{\omega _{\varepsilon _{1}}}(\tilde{Q}_{i})Id_{%
\tilde{Q}_{i}}\right\Vert _{L^{2}(\tilde{X},\omega _{\varepsilon
_{1}})}\right)  \\
& +\varepsilon _{1}C(\kappa )\left( \left\Vert F_{\tilde{G}_{i}^{\varepsilon
_{1}}}\right\Vert _{L^{2}(\tilde{X},\omega _{\varepsilon _{1}})}+\frac{1}{n}%
\left\Vert \omega _{\varepsilon _{1}}\mu _{\omega _{\varepsilon _{1}}}(%
\tilde{Q}_{i})Id_{\tilde{Q}_{i}}\right\Vert _{L^{2}(\tilde{X},\omega
_{\varepsilon _{1}})}\right) 
\end{align*}%
and%
\begin{eqnarray*}
&&\left\Vert \frac{\sqrt{-1}}{n}\Lambda _{\omega _{\varepsilon }}(\omega
_{\varepsilon _{1}}-\omega _{\varepsilon })\mu _{\omega _{\varepsilon _{1}}}(%
\tilde{Q}_{i})Id_{\tilde{Q}_{i}}\right\Vert _{L^{\alpha }(\tilde{X}-\pi
^{-1}(U_{R/2}),\omega _{\varepsilon })} \\
&\leq &\frac{\varepsilon _{1}}{n}C\left( \left\Vert \left( \Lambda _{\omega
_{\varepsilon _{1}}}\eta \right) \mu _{\omega _{\varepsilon _{1}}}(\tilde{Q}%
_{i})Id_{\tilde{Q}_{i}}\right\Vert _{L^{\tilde{\alpha}}(\tilde{X},\omega
_{\varepsilon _{1}})}+\kappa \left\Vert \eta \mu _{\omega _{\varepsilon
_{1}}}(\tilde{Q}_{i})Id_{\tilde{Q}_{i}}\right\Vert _{L^{2}(\tilde{X},\omega
_{\varepsilon _{1}})}\right)  \\
&&+\frac{\varepsilon _{1}^{2}}{n}C(\kappa )\left\Vert \eta \mu _{\omega
_{\varepsilon _{1}}}(\tilde{Q}_{i})Id_{\tilde{Q}_{i}}\right\Vert _{L^{2}(%
\tilde{X},\omega _{\varepsilon _{1}})}
\end{eqnarray*}%
again using Lemma \ref{Lemma12}. Here we have used the fact that $\omega
_{\varepsilon _{1}}-\omega _{\varepsilon }=(\varepsilon _{1}-\varepsilon
)\eta $ in the second inequality. Of course, $\bigl\Vert\Lambda _{\omega
_{\varepsilon _{1}}}F_{\tilde{G}_{i}^{\varepsilon _{1}}}-\sqrt{-1}\mu
_{\omega _{\varepsilon _{1}}}(\tilde{Q}_{i})id_{\tilde{Q}_{i}}\bigr\Vert_{L^{%
\tilde{\alpha}}(\tilde{X},\omega _{\varepsilon _{1}})}=0$, by the
construction of $G_{i}^{\varepsilon _{1}}$. On the other hand:%
\begin{align*}
\left\Vert F_{\tilde{G}_{i}^{\varepsilon _{1}}}\right\Vert _{L^{2}(\tilde{X}%
,\omega _{\varepsilon _{1}})}& =\left\Vert \Lambda _{\omega _{\varepsilon
_{1}}}F_{\tilde{G}_{i}^{\varepsilon _{1}}}\right\Vert _{L^{2}(\tilde{X}%
,\omega _{\varepsilon _{1}})}+\pi ^{2}n(n-1)\int_{\tilde{X}}\left( 2c_{2}(%
\tilde{Q}_{i})-c_{1}^{2}(\tilde{Q}_{i})\right) \wedge \omega _{\varepsilon
_{1}}^{n-2} \\
& =\left\Vert \mu _{\omega _{\varepsilon _{1}}}(\tilde{Q}_{i})Id_{\tilde{Q}%
_{i}}\right\Vert _{L^{2}(\tilde{X},\omega _{\varepsilon _{1}})}+\pi
^{2}n(n-1)\int_{\tilde{X}}\left( 2c_{2}(\tilde{Q}_{i})-c_{1}^{2}(\tilde{Q}%
_{i})\right) \wedge \omega _{\varepsilon _{1}}^{n-2}
\end{align*}%
which is bounded. Likewise the terms $\left\Vert \omega _{\varepsilon
_{1}}\mu _{\omega _{\varepsilon _{1}}}(\tilde{Q}_{i})Id_{\tilde{Q}%
_{i}}\right\Vert _{L^{2}(\tilde{X},\omega _{\varepsilon _{1}})}$ and $%
\left\Vert \eta \mu _{\omega _{\varepsilon _{1}}}(\tilde{Q}_{i})Id_{\tilde{Q}%
_{i}}\right\Vert _{L^{2}(\tilde{X},\omega _{\varepsilon _{1}})}$ are
bounded. The only remaining issue is: $\left\Vert \left( \Lambda _{\omega
_{\varepsilon _{1}}}\eta \right) \mu _{\omega _{\varepsilon _{1}}}(\tilde{Q}%
_{i})Id_{\tilde{Q}_{i}}\right\Vert _{L^{\tilde{\alpha}}(\tilde{X},\omega
_{\varepsilon _{1}})}$. But writing 
\begin{equation*}
\left\vert \Lambda _{\omega _{\varepsilon _{1}}}\eta \right\vert ^{^{\tilde{%
\alpha}}}=\left\vert \frac{\eta \wedge \omega _{\varepsilon _{1}}^{n-1}}{%
\omega _{\varepsilon _{1}}^{n}}\right\vert ^{^{\tilde{\alpha}}}=\left\vert 
\frac{\eta \wedge \omega _{\varepsilon _{1}}^{n-1}}{\eta ^{n}}\right\vert
^{^{\tilde{\alpha}}}\left\vert \frac{\det g_{\eta }}{\det g_{\varepsilon
_{1}}}\right\vert ^{^{\tilde{\alpha}}}
\end{equation*}%
and%
\begin{equation*}
\omega _{\varepsilon _{1}}^{n}=\left\vert \frac{\det g_{\varepsilon _{1}}}{%
\det g_{\eta }}\right\vert \eta ^{n}
\end{equation*}%
\begin{eqnarray*}
&&\left\Vert \left( \Lambda _{\omega _{\varepsilon _{1}}}\eta \right) \mu
_{\omega _{\varepsilon _{1}}}(\tilde{Q}_{i})Id_{\tilde{Q}_{i}}\right\Vert
_{L^{\tilde{\alpha}}(\tilde{X},\omega _{\varepsilon _{1}})} \\
&\leq &C\left( \int_{\tilde{X}}\left\vert \frac{\det g_{\varepsilon _{1}}}{%
\det g_{\eta }}\right\vert ^{(1-\tilde{\alpha})\tilde{s}}\eta ^{n}\right) ^{%
\frac{1}{\tilde{\alpha}\tilde{s}}}\left( \int_{\tilde{X}}\left\vert \frac{%
\eta \wedge \omega _{\varepsilon _{1}}^{n-1}}{\eta ^{n}}\right\vert
^{^{\beta }}\left\vert \mu _{\omega _{\varepsilon _{1}}}(\tilde{Q}_{i})Id_{%
\tilde{Q}_{i}}\right\vert ^{\beta }\eta ^{n}\right) ^{\frac{1}{^{\beta }}}
\end{eqnarray*}%
by H\"{o}lder's inequality with respect to the metric $\eta $. Here again $%
\tilde{\alpha}$ is as in Lemma \ref{Lemma12} and $\tilde{s}=\frac{\beta }{%
\beta -\tilde{\alpha}}$ where $\frac{\tilde{\alpha}}{1-2(k-1)(\tilde{\alpha}%
-1)}<\beta <\infty $. By Lemma \ref{Lemma11} this is uniformly bounded in $%
\varepsilon _{1}$ since we also have $\omega _{\varepsilon
_{1}}^{n-1}\longrightarrow \pi ^{\ast }\omega ^{n-1}$.

Therefore we may choose $t,\kappa ,$ and $\varepsilon _{1}$ so that%
\begin{equation*}
\left\Vert \Lambda _{\omega _{\varepsilon }}F_{_{\tilde{G}_{\varepsilon
_{1}}}}-\sqrt{-1}\mu _{\omega _{\varepsilon _{1}}}(\tilde{Q}_{i})Id_{\tilde{G%
}_{\varepsilon _{1}}}\right\Vert _{_{L^{\alpha }(\tilde{X}-\pi
^{-1}(U_{R/2}),\omega _{\varepsilon })}}<\frac{\delta }{4}
\end{equation*}%
for all $\varepsilon $ and all $\alpha $ sufficiently close to $1$. We will
now fix these values of $t$,$\kappa $, and $\varepsilon _{1}$.

The term 
\begin{equation*}
\left\vert \int_{\pi ^{-1}(U_{R/2})}\Phi _{\alpha }(\Lambda _{\omega
_{\varepsilon }}F_{H}+\sqrt{-1}N\mathbf{I}_{E})-\Phi _{\alpha }(\sqrt{-1}%
(\mu +N))\right\vert
\end{equation*}%
is bounded by:%
\begin{equation*}
C\left\Vert \Lambda _{\omega _{\varepsilon }}F_{H}-\sqrt{-1}\mu \right\Vert
_{L^{\alpha }(\pi ^{-1}(U_{R/2}),\omega _{\varepsilon })}.
\end{equation*}%
Now write%
\begin{equation*}
\left\vert \Lambda _{\omega _{\varepsilon }}F_{H}\right\vert ^{\alpha
}=\left\vert \frac{F_{H}\wedge \omega _{\varepsilon }^{n-1}}{\omega
_{\varepsilon }^{n-1}}\right\vert ^{^{\alpha }}=\left\vert \frac{F_{H}\wedge
\omega _{\varepsilon }^{n-1}}{\eta ^{n}}\right\vert ^{^{\tilde{\alpha}%
}}\left\vert \frac{\det g_{\eta }}{\det g_{\varepsilon }}\right\vert ^{^{%
\tilde{\alpha}}}
\end{equation*}%
and%
\begin{equation*}
\omega _{\varepsilon }^{n}=\left\vert \frac{\det g_{\varepsilon }}{\det
g_{\eta }}\right\vert \eta ^{n},
\end{equation*}%
we have%
\begin{eqnarray*}
\left\Vert (\Lambda _{\omega _{\varepsilon }}F_{H}-\sqrt{-1}\mu \right\Vert
_{L^{\alpha }(\pi ^{-1}(U_{R/2}),\omega _{\varepsilon })} &\leq
&C_{1}\left\Vert \Lambda _{\omega _{\varepsilon }}F_{H}\right\Vert
_{L^{\alpha }(\pi ^{-1}(U_{R/2}),\omega _{\varepsilon })}+C_{2}\limfunc{Vol}%
(U_{R/2},\omega )\leq \\
&&\hskip-1.5inC_{1}\left( \int_{\pi ^{-1}(U_{R/2})}\left\vert \frac{\det
g_{\varepsilon }}{\det g_{\eta }}\right\vert ^{(1-\alpha )s}\eta ^{n}\right)
^{\frac{1}{s\tilde{\alpha}}}\left( \int_{\pi ^{-1}(U_{R/2})}\left\vert \frac{%
F_{H}\wedge \omega _{\varepsilon }^{n-1}}{\eta ^{n}}\right\vert ^{^{\tilde{%
\alpha}}}\eta ^{n}\right) ^{\frac{1}{^{^{\tilde{\alpha}}}}}+C_{2}\limfunc{Vol%
}(U_{R/2},\omega ),
\end{eqnarray*}%
where $\alpha $ and $s$ are as in Lemma \ref{Lemma11}. By that lemma, the
factor%
\begin{equation*}
\left( \int_{\pi ^{-1}(U_{R/2})}\left\vert \frac{\det g_{\varepsilon }}{\det
g_{\eta }}\right\vert ^{(1-\alpha )s}\eta ^{n}\right)
\end{equation*}%
is uniformly bounded, and so the result is that there is an $R$ such that%
\begin{equation*}
\left\vert \int_{\pi ^{-1}(U_{R/2})}\Phi _{\alpha }(\Lambda _{\omega
_{\varepsilon }}F_{H}+\sqrt{-1}N\mathbf{I}_{E})-\Phi _{\alpha }(\sqrt{-1}%
(\mu +N))\right\vert <\frac{\delta }{8}.
\end{equation*}%
Therefore the only remaining estimates required are on: $\bigl\Vert\Lambda
_{\omega _{\varepsilon }}F_{\tilde{h}_{\psi _{R}}}-\Lambda _{\omega
_{\varepsilon }}F_{\tilde{h}}\bigr\Vert_{L^{\alpha }(\pi
^{-1}(U_{R}-U_{R/2}),\omega _{\varepsilon })}$. If we let $k_{\psi _{R}}$ be
an endomorphism such that $\tilde{h}=k_{\psi _{R}}h_{\psi _{R}}$. Then 
\begin{equation*}
F_{h_{\psi _{R}}}-F_{\tilde{h}}=\bar{\partial}_{\tilde{E}}(k_{\psi
_{R}}^{-1}\partial _{\tilde{h}}k_{\psi _{R}})
\end{equation*}%
where $\partial _{\tilde{h}}$ is the $(1,0)$ part of the Chern connection
for $\tilde{h}$. The expression on the right hand side involves only two
derivatives of $\psi _{R}$, and so, using the bound on the derivatives of $%
\psi _{R}$, there is a bound of the form:%
\begin{equation*}
\left\vert F_{h_{\psi _{R}}}-F_{\tilde{h}}\right\vert \leq C_{1}+\frac{C_{2}%
}{R^{2}}.
\end{equation*}

where $C_{1}$ and $C_{2}$ are independent of both $\varepsilon $ and $R$.
Now as usual we have:%
\begin{eqnarray*}
\left\vert \Lambda _{\omega _{\varepsilon }}\left( F_{\tilde{h}_{\psi
_{R}}}-F_{\tilde{h}}\right) \right\vert ^{\alpha } &=&\left\vert \frac{%
\left( F_{\tilde{h}_{\psi _{R}}}-F_{\tilde{h}}\right) \wedge \omega
_{\varepsilon }^{n-1}}{\omega _{\varepsilon }^{n}}\right\vert ^{\alpha
}=\left\vert \frac{\left( F_{\tilde{h}_{\psi _{R}}}-F_{\tilde{h}}\right)
\wedge \omega _{\varepsilon }^{n-1}}{\eta ^{n}}\right\vert ^{\alpha
}\left\vert \frac{\det g_{\eta }}{\det g_{\varepsilon }}\right\vert ^{\alpha
} \\
\text{and}\ \omega _{\varepsilon }^{n} &=&\frac{\det g_{\varepsilon }}{\det
g_{\eta }}\eta ^{n}.
\end{eqnarray*}%
Then we compute:%
\begin{eqnarray*}
&&\left\Vert \Lambda _{\omega _{\varepsilon }}F_{h_{\psi _{R}}}-\Lambda
_{\omega _{\varepsilon }}F_{\tilde{h}}\right\Vert _{L^{\alpha }(\pi
^{-1}(U_{R}-U_{R/2}),\omega _{\varepsilon })} \\
&=&\left( \int_{\pi ^{-1}(U_{R}-U_{R/2})}\left\vert \frac{\left( F_{\tilde{h}%
_{\psi _{R}}}-F_{\tilde{h}}\right) \wedge \omega _{\varepsilon }^{n-1}}{\eta
^{n}}\right\vert ^{\alpha }\left\vert \frac{\det g_{\eta }}{\det
g_{\varepsilon }}\right\vert ^{\alpha }\frac{\det g_{\varepsilon }}{\det
g_{\eta }}\eta ^{n}\right) ^{\frac{1}{\alpha }} \\
&\leq &\left( \int_{\pi ^{-1}(U_{R}-U_{R/2})}\left( \frac{\det
g_{\varepsilon }}{\det g_{\eta }}\right) ^{\left( 1-\alpha \right) s}\eta
^{n}\right) ^{\frac{1}{\alpha s}}\left( \int_{\pi
^{-1}(U_{R}-U_{R/2})}\left( C_{1}+\frac{C_{2}}{R^{2\tilde{\alpha}}}\right)
\eta ^{n}\right) ^{\frac{1}{\tilde{\alpha}}}.
\end{eqnarray*}%
Here $s$ and $\tilde{\alpha}$ are as in Lemma \ref{Lemma11} and we have
applied H\"{o}lder's inequality to the conjugate pair $s$ and $\frac{\tilde{%
\alpha}}{\alpha }$. By that lemma, the first factor is uniformly bounded in $%
\varepsilon $. We must therefore show that as $R\rightarrow 0$, the first
factor can be made arbitrarily small. To do this we note that the open set $%
U_{R}$ may be covered by a union of balls ${\LARGE \cup }_{j}B_{r}^{j}$.
Therefore:%
\begin{equation*}
\int_{\pi ^{-1}(U_{R}-U_{R/2})}C_{1}+C_{2}R^{-2\tilde{\alpha}}\leq
\sum_{j}(C_{1}+C_{2}R^{-2\tilde{\alpha}})vol(B_{r}^{j})
\end{equation*}%
and up to a constant $vol(B_{r}^{j})=r^{2n}$ where $n$ is the complex
dimension of $X$.

The key observation is now that the singular set $Z_{\limfunc{alg}}$ is a
complex submanifold of $X$ and has complex codimension at least $2$, in
other words it is of real dimension at most $2n-4$. This implies that $Z_{%
\limfunc{alg}}$ has Hausdorff dimension at most $2n-4$, i.e. it has zero $d$%
-dimensional Hausdorff measure for $d<2n-4$. In other words, for each $0\leq
d<4$, and a given $\delta >0$, there is a cover of $Z_{\limfunc{alg}}$ and
an $r>0$ such that $\sum_{j}r^{2n-d}<\delta $. Now assume that we have
chosen $R=r$. Then then the cover described above is also a cover for $U_{R}$
so 
\begin{equation*}
\int_{\pi ^{-1}(U_{R}-U_{R/2})}C_{1}+C_{2}R^{-2\tilde{\alpha}}\leq
\sum_{j}(C_{1}r^{2n}+C_{2}r^{2n-2\tilde{\alpha}}).
\end{equation*}%
Note that by assumption $\tilde{\alpha}<2$. In other words, we may select $R$
so that: 
\begin{equation*}
\left\Vert \Lambda _{\omega _{\varepsilon }}F_{\tilde{h}_{\psi
_{R}}}-\Lambda _{\omega _{\varepsilon }}F_{\tilde{h}}\right\Vert _{L^{\alpha
}(\pi ^{-1}(U_{R}-U_{R/2}),\omega _{\varepsilon })}<\frac{\delta }{16}.
\end{equation*}%
Thus choosing $\varepsilon _{1}$ and $R$ in the manner specified above gives
us for each $\varepsilon $ a bound on the difference of the $HYM$
functionals: $\left\vert HYM_{\alpha ,N}^{\omega _{\varepsilon }}(\bar{%
\partial}_{E},\tilde{h}_{\psi _{R}})-HYM_{\alpha ,N}(\mu )\right\vert \leq
\delta $. Now sending $\varepsilon \rightarrow 0$ we finally see that there
exists a metric $h$ with $\left\vert HYM_{\alpha ,N}^{\omega }(\bar{\partial}%
_{E},h)-HYM_{\alpha ,N}(\mu )\right\vert <\delta $, for all $N$ and all $%
\alpha $ sufficiently close to $1$.
\end{proof}

\begin{lemma}
\qquad \label{Lemma13}Let $E\rightarrow X$ and $\alpha _{0}$ be the same as
in the proposition. Let $h$ be any smooth hermitian metric on $E$ and $A_{t}$
a solution of the Yang-Mills flow whose initial condition is $(\bar{\partial}%
_{E},h)$. Let $\mu _{0}$ denote the Harder-Narasimhan type of $E$. Then $%
\lim_{t\rightarrow \infty }HYM_{\alpha ,N}(A_{t})=HYM_{\alpha ,N}(\mu _{0})$%
, for all $1\leq \alpha \leq \alpha _{0}$ and all $N$.
\end{lemma}

As a consequence, if $A_{\infty }$ is an Uhlenbeck limit along the flow: $%
HYM_{\alpha ,N}(A_{\infty })=HYM_{\alpha ,N}(\mu _{0})$, since $HYM_{\alpha
,N}(A_{\infty })=\lim_{t\rightarrow \infty }HYM_{\alpha ,N}(A_{t})$. The
proof of Lemma \ref{Lemma13} is exactly the same as in \cite{DW1}. It uses $%
\limfunc{Proposition}$ \ref{Prop18}. One easily shows that for any initial
metric such that the conclusion of $\limfunc{Proposition}$ \ref{Prop18}
holds, the property $\lim_{t\rightarrow \infty }HYM_{\alpha
,N}(A_{t})=HYM_{\alpha ,N}(\mu _{0})$ holds. The fact that this is true for
any metric then follows from a distance decreasing argument.

We can now identify the Harder-Narasimhan type of the limit.

\begin{proposition}
\label{Prop19}Let $E\rightarrow X$ have the same properties as before. Let $%
A_{t}$ be a solution to the $YM$ flow with initial condition $A_{0}$ whose
limit along the flow is $A_{\infty }$. Let $E_{\infty }$ be the
corresponding holomorphic vector bundle defined away from $Z_{\limfunc{an}}$%
. Then the $HN$ type of $(E_{\infty },A_{\infty })$ is the same as that of $%
(E_{0},A_{0})$.
\end{proposition}

\begin{proof}
Let $\mu _{0}=(\mu _{1},\cdots ,\mu _{K})$ and $\mu _{\infty }=(\mu
_{1}^{\infty },\cdots ,\mu _{K}^{\infty })$ be the $HN$ types of $%
(E_{0},A_{0})$ and $(E_{\infty },A_{\infty })$. A restatement of the above
lemma is that $\Phi _{\alpha }(\mu _{0}+N)=\Phi _{\alpha }(\mu _{\infty }+N)$
for all $1\leq \alpha \leq \alpha _{0}$ and all $N$. Choose $N$ to be large
enough so that $\mu _{K}+N\geq 0$. Then we also have $\mu _{K}^{\infty
}+N\geq 0$ by $\limfunc{Proposition}$ \ref{Prop9}, and therefore $\mu
_{K}+N=\mu _{K}^{\infty }+N$ by $\limfunc{Proposition}$ \ref{Prop10}, so $%
\mu _{K}=\mu _{K}^{\infty }$.
\end{proof}

Let $(E,\bar{\partial}_{A_{0}})$ be a holomorphic bundle, and $A_{0}$ an
initial connection, and $A_{t_{j}}$ its evolution along the flow for a
sequence of times $t_{j}$. Then we have the following.

\begin{lemma}
\label{Lemma14}$(1)$ Let $\bigl\{\pi _{j}^{(i)}\bigr\}$ be the $HN$
filtration of $(E$,$\bar{\partial}_{A_{t_{j}}})$ and $\bigl\{\pi _{\infty
}^{(i)}\bigr\}$ the $HN$ filtration of $(E_{\infty }$,$\partial _{A_{\infty
}})$. Then after passing to a subsequence, $\pi _{j}^{(i)}\rightarrow \pi
_{\infty }^{(i)}$ strongly $L^{p}\cap L_{1,loc}^{2}$ for all $1\leq p<\infty 
$ and all $i$.

\ \ $(2)$ Assume the original bundle $(E,\bar{\partial}_{A_{0}})$ is
semi-stable and $\bigl\{\pi _{ss,j}^{(i)}\bigr\}$ are Seshadri filtrations
of $(E,\bar{\partial}_{A_{t_{j}j}})$. Without loss of generality assume the
ranks of the subsheaves $\pi _{ss,j}^{(i)}$ are constant in $j$. Then there
is a filtration $\bigl\{\pi _{ss,\infty }^{(i)}\bigr\}$ of $(E$,$\bar{%
\partial}_{A_{\infty }})$ such that after passing to a subsequence $\bigl\{%
\pi _{ss,j}^{(i)}\bigr\}\rightarrow $ $\bigl\{\pi _{ss,\infty }^{(i)}\bigr\}$
strongly in $L^{p}\cap L_{1,loc}^{2}$ for all $1\leq p<\infty $ and all $i$.
The rank and degree of $\pi _{ss,\infty }^{(i)}$ is equal to the rank and
degree of $\pi _{ss,j}^{(i)}$ for all $i$ and $j$.
\end{lemma}

For the proof see \cite{DW1} Lemma $4.5$. It uses $\limfunc{Proposition}$ %
\ref{Prop19}.

\begin{proposition}
\label{Prop20}Assume as before that $E\rightarrow X$ is a holomorphic vector
bundle such that $Z_{\limfunc{an}}$ is smooth and that blowing up once
resolves the singularities of the $HNS$ filtration. Then given $\delta >0$
and any $1\leq p<\infty $, $E$ has an $L^{p}$ $\delta $-approximate critical
hermitian structure.
\end{proposition}

\begin{proof}
Let $A_{t}$ be a solution to the $YM$ flow with initial condition $A_{0}=(%
\bar{\partial}_{E},h)$, and let $A_{\infty }$ be the limit along the flow
for some sequence $A_{t_{j}}$. Then we may apply the previous lemma to
conclude that $\Psi _{\omega }^{HNS}(\bar{\partial}_{A_{t_{j}}},h)\overset{%
L^{p}}{\rightarrow }\Psi _{\omega }^{HNS}(\bar{\partial}_{A_{\infty
}},h_{\infty })$ after passing to another subsequence if necessary. Since $%
A_{\infty }$ is a Yang-Mills connection, $\sqrt{-1}\Lambda _{\omega
}F_{A_{\infty }}=\Psi _{\omega }^{HN}(\bar{\partial}_{A_{\infty }},h_{\infty
})$. Therefore:%
\begin{eqnarray*}
\left\Vert \sqrt{-1}\Lambda _{\omega }F_{A_{t_{j}}}-\Psi _{\omega }^{HNS}(%
\bar{\partial}_{A_{t_{j}}},h)\right\Vert _{L^{p}(\omega )} &\leq & \\
\left\Vert \Lambda _{\omega }F_{A_{t_{j}}}-\Lambda _{\omega }F_{A_{\infty
}}\right\Vert _{L^{p}(\omega )}+\left\Vert \Psi _{\omega }^{HNS}(\bar{%
\partial}_{A_{t_{j}}},h)-\Psi _{\omega }^{HNS}(\bar{\partial}_{A_{\infty
}},h_{\infty })\right\Vert _{L^{p}(\omega )} &\longrightarrow &0
\end{eqnarray*}%
where we have also used Lemma \ref{Lemma6}$.$
\end{proof}

Now we would like to eliminate the assumptions that $Z_{\limfunc{an}}$ is
smooth and that blowing up once resolves the singularities of the $HNS$
filtration.

\begin{theorem}
\label{Thm9}Let $E\rightarrow X$ be a holomorphic vector bundle over a K\"{a}%
hler manifold with K\"{a}hler form $\omega $. Then given $\delta >0$ and any 
$1\leq p<\infty $, $E$ has an $L^{p}\ \delta $-approximate critical
hermitian structure.
\end{theorem}

\begin{proof}
By \ref{Prop13}, we know that we can resolve the singularities of the $HNS$
filtration by blowing up finitely many times. Moreover, the $i^{th}$ blowup
is obtained by blowing up along a complex submanifold contained in the
singular set associated to the pullback bundle over the manifold produced at
the $(i-1)st$ stage of the process. In other words there is a tower of
blow-ups:%
\begin{equation*}
\tilde{X}=X_{m}\overset{\pi _{m}}{\longrightarrow }X_{m-1}\overset{\pi _{m-1}%
}{\longrightarrow }\cdots \overset{\pi _{2}}{\longrightarrow }X_{1}\overset{%
\pi _{1}}{\longrightarrow }X_{0}=X
\end{equation*}%
such that if $E=E_{0}$ is the original bundle, and $E_{i}=\pi _{i}^{\ast
}(E_{i-1})$, then there is a filtration of $\tilde{E}=\pi _{m}^{\ast
}(E_{m-1})$ that is given by sub-bundles and isomorphic to the $HNS$
filtration of $E$ away from $\mathbf{E}$. Note that on each blowup $X_{i}$
we have a family of K\"{a}hler metrics defined iteratively by $\omega
_{\varepsilon _{1},\cdots ,\varepsilon _{i}}=\pi ^{\ast }\omega
_{\varepsilon _{1},\cdots ,\varepsilon _{i-1}}+\varepsilon _{i}\eta _{i}$,
where $\eta _{i}$ is any K\"{a}hler form on $X_{i}$. Then consider $\omega
_{\varepsilon _{1},\cdots ,\varepsilon _{m}}$ on $\tilde{X}$ to be a fixed
metric for specified values of $\varepsilon _{1},\cdots ,\varepsilon _{m}<1$%
, and fix $\delta >0$. Fix $\delta _{0}$ to be a number that is very small
with respect to $\delta $. By the previous proposition, for every $p$ there
is a $\delta _{0}$-approximate critical hermitian structure on $E_{n-1}$. In
particular there is such a metric for $p=2$. In other words there is a
metric $h_{m-1}$ so that: 
\begin{equation*}
\left\Vert \sqrt{-1}\Lambda _{\omega _{\varepsilon _{1},\cdots \varepsilon
_{m-1}}}F_{\left( \bar{\partial}_{E_{m-1}},h_{m-1}\right) }-\Psi _{\omega
_{\varepsilon _{1},\cdots \varepsilon _{m-1}}}^{HNS}(\bar{\partial}%
_{E_{m-1}},h_{m-1})\right\Vert _{L^{2}(\omega _{\varepsilon _{1},\cdots
\varepsilon _{m-1}})}<\delta _{0}.
\end{equation*}%
By construction this metric depends on the values of $\varepsilon
_{1},\cdots ,\varepsilon _{m}$, since it is constructed from a metric on the
blowup which itself is constructed using the notion of stability with
respect to $\omega _{\varepsilon _{1},\cdots ,\varepsilon _{m}}$.

We prove the result by induction on the number of blowups. Assume that we
have an $L^{2}$ $\delta _{0}$-approximate critical hermitian structure for
each of the bundles $E_{i}\rightarrow X_{i}\,$\ for $1\leq i\leq m-2$. Then
in particular, with respect to the metric $\omega _{\varepsilon _{1}}$ on $%
X_{1}$, we have a metric $h_{1}$ on $E_{1}\rightarrow X_{1}$ such that: 
\begin{equation*}
\left\Vert \sqrt{-1}\Lambda _{\omega _{\varepsilon _{1}}}F_{(\bar{\partial}%
_{E_{1}},h_{1})}-\Psi _{\omega _{\varepsilon _{1}}}^{HNS}(\bar{\partial}%
_{E_{1}},h_{1})\right\Vert _{_{L^{2}(\omega _{\varepsilon _{1}})}}<\delta
_{0}.
\end{equation*}%
Since $X_{1}$ is obtained from $X$ by blowing up along a smooth, complex
submanifold, we may use the exact same cut-off argument, choosing a cutoff
function with respect to a neighbourhood $U_{R}$ as in Proposition \ref%
{Prop18} to construct a metric $h_{R}$ on the bundle $E\rightarrow X$ which
depends on the value of $\varepsilon _{1}$. In the following we will
continue to denote by $h_{R}$ its pullback to $X_{1}$. As in the proof of
Proposition \ref{Prop18} we have $h_{R}=h_{1}$ outside of the set $\pi
_{1}^{-1}(U_{R})$. We divide the proof into two steps.

\textbf{(Step 1) There is an }$L^{p}\ \delta $\textbf{-approximate critical
hermitian structure for p close to 1 }

First let us assume that $p$ satisfies the hypotheses of Lemma \ref{Lemma12}%
. In other words, substitute $p$ for $\alpha $ in the statement. Similarly,
substitute $\tilde{p}$ for $\tilde{\alpha}$. We will show that a single
metric, namely $h_{R}$, gives an $L^{p}\ \delta $-approximate critical
hermitian structure for all $p$ within this range. We need to estimate the
difference%
\begin{equation*}
\left\Vert \sqrt{-1}\Lambda _{\omega _{\varepsilon }}F_{(\bar{\partial}%
_{E_{1}},h_{R})}-\Psi _{\omega }^{HNS}(\bar{\partial}_{E},h_{R})\right\Vert
_{_{L^{p}(\omega _{\varepsilon })}}
\end{equation*}%
where $\tilde{h}=\pi _{1}^{\ast }h$. Now:%
\begin{eqnarray*}
\left\Vert \sqrt{-1}\Lambda _{\omega _{\varepsilon }}F_{(\bar{\partial}%
_{E_{1}},h_{R})}-\Psi _{\omega }^{HNS}(\bar{\partial}_{E},h_{R})\right\Vert
_{_{L^{p}(\omega _{\varepsilon })}} &\leq & \\
&&\hskip-2.5in\left\Vert \Lambda _{\omega _{\varepsilon }}F_{(\bar{\partial}%
_{E_{1}},h_{R})}-\Lambda _{\omega _{\varepsilon }}F_{(\bar{\partial}%
_{E_{1}},h_{1})}\right\Vert _{_{L^{p}(\omega _{\varepsilon })}}+\left\Vert
\Psi _{\omega _{\varepsilon _{1}}}^{HNS}(\bar{\partial}_{E},h_{1})-\Psi
_{\omega }^{HNS}(\bar{\partial}_{E},h_{R})\right\Vert _{_{L^{p}(\omega
_{\varepsilon })}} \\
&&+\left\Vert \Lambda _{\omega _{\varepsilon }}F_{(\bar{\partial}%
_{E_{1}},h_{1})}-\Psi _{\omega _{\varepsilon _{1}}}^{HNS}(\bar{\partial}%
_{E},h_{1})\right\Vert _{_{L^{p}(\omega _{\varepsilon })}}.
\end{eqnarray*}%
We can make the second term smaller than $\frac{\delta }{3}$ by choosing $%
\varepsilon _{1}$ small and using the convergence of the $HN$ types. The
third term is bounded by two applications of Lemma$\ $\ref{Lemma12} as
follows: 
\begin{eqnarray*}
&&\left\Vert \Lambda _{\omega _{\varepsilon }}F_{(\bar{\partial}%
_{E_{1}},h_{1})}-\Psi _{\omega _{\varepsilon _{1}}}^{HNS}(\bar{\partial}%
_{E},h_{1})\right\Vert _{_{L^{p}(\omega _{\varepsilon })}}\leq  \\
&&\left\Vert \Lambda _{\omega _{\varepsilon }}\left( F_{(\bar{\partial}%
_{E_{1}},h_{1})}-\frac{1}{n}\omega _{\varepsilon _{1}}\Psi _{\omega
_{\varepsilon _{1}}}^{HNS}(\bar{\partial}_{E},h_{1})\right) \right\Vert
_{_{_{L^{p}(\omega _{\varepsilon })}}}+\left\Vert \frac{1}{n}\Lambda
_{\omega _{\varepsilon }}\left( \omega _{\varepsilon _{1}}-\omega
_{\varepsilon }\right) \Psi _{\omega _{\varepsilon _{1}}}^{HNS}(\bar{\partial%
}_{E},h_{1})\right\Vert _{_{_{L^{p}(\omega _{\varepsilon })}}} \\
&\leq &C\left\Vert \Lambda _{\omega _{\varepsilon _{1}}}F_{(\bar{\partial}%
_{E_{1}},h_{1})}-\Psi _{\omega _{\varepsilon _{1}}}^{HNS}(\bar{\partial}%
_{E},h_{1})\right\Vert _{_{_{L^{\tilde{p}}(\omega _{\varepsilon _{1}})}}} \\
&&+\kappa C\left( \left\Vert F_{(\bar{\partial}_{E_{1}},h_{1})}\right\Vert
_{L^{2}(\tilde{X},\omega _{\varepsilon _{1}})}+\frac{1}{n}\left\Vert \omega
_{\varepsilon _{1}}\Psi _{\omega _{\varepsilon _{1}}}^{HNS}(\bar{\partial}%
_{E},h_{1})\right\Vert _{L^{2}(\tilde{X},\omega _{\varepsilon _{1}})}\right) 
\\
&&+\varepsilon _{1}C(\kappa )\left( \left\Vert F_{(\bar{\partial}%
_{E_{1}},h_{1})}\right\Vert _{L^{2}(\tilde{X},\omega _{\varepsilon _{1}})}+%
\frac{1}{n}\left\Vert \omega _{\varepsilon _{1}}\Psi _{\omega _{\varepsilon
_{1}}}^{HNS}(\bar{\partial}_{E},h_{1})\right\Vert _{L^{2}(\tilde{X},\omega
_{\varepsilon _{1}})}\right)  \\
&&+\frac{\varepsilon _{1}^{2}}{n}C(\kappa )\left\Vert \eta \Psi _{\omega
_{\varepsilon _{1}}}^{HNS}(\bar{\partial}_{E},h_{1})\right\Vert _{L^{2}(%
\tilde{X},\omega _{\varepsilon _{1}})} \\
&&+\frac{\varepsilon _{1}}{n}C\left( \left\Vert \Lambda _{\omega
_{\varepsilon _{1}}}\eta \Psi _{\omega _{\varepsilon _{1}}}^{HNS}(\bar{%
\partial}_{E},h_{1})\right\Vert _{L^{\tilde{p}}(\tilde{X},\omega
_{\varepsilon _{1}})}+\kappa \left\Vert \eta \Psi _{\omega _{\varepsilon
_{1}}}^{HNS}(\bar{\partial}_{E},h_{1})\right\Vert _{L^{2}(\tilde{X},\omega
_{\varepsilon _{1}})}\right) .
\end{eqnarray*}%
Recall from the statement of Lemma $\ $\ref{Lemma12} that none of the above
constants depends on $\varepsilon _{1}$. All terms with a $\kappa $ in front
and no $C(\kappa )$ can be made small by choosing $\kappa $ small, so these
terms can be ignored. Clearly the terms 
\begin{equation*}
\left\Vert \omega _{\varepsilon _{1}}\Psi _{\omega _{\varepsilon
_{1}}}^{HNS}(\bar{\partial}_{E},h_{1})\right\Vert _{L^{2}(\tilde{X},\omega
_{\varepsilon _{1}})},\left\Vert \eta \Psi _{\omega _{\varepsilon
_{1}}}^{HNS}(\bar{\partial}_{E},h_{1})\right\Vert _{L^{2}(\tilde{X},\omega
_{\varepsilon _{1}})}
\end{equation*}%
are bounded independently of $\varepsilon _{1}$ since the $HN$ type
converges. Therefore we need only show that%
\begin{equation*}
\left\Vert \Lambda _{\omega _{\varepsilon _{1}}}F_{(\bar{\partial}%
_{E_{1}},h_{1})}-\Psi _{\omega _{\varepsilon _{1}}}^{HNS}(\bar{\partial}%
_{E},h_{1})\right\Vert _{_{_{L^{\tilde{p}}(\omega _{\varepsilon
_{1}})}}},\left\Vert F_{(\bar{\partial}_{E_{1}},h_{1})}\right\Vert _{L^{2}(%
\tilde{X},\omega _{\varepsilon _{1}})},\left\Vert \Lambda _{\omega
_{\varepsilon _{1}}}\eta \Psi _{\omega _{\varepsilon _{1}}}^{HNS}(\bar{%
\partial}_{E},h_{1})\right\Vert _{L^{\tilde{p}}(\tilde{X},\omega
_{\varepsilon _{1}})}
\end{equation*}%
are uniformly bounded in $\varepsilon _{1}$. Then we can choose $\kappa $
first and then $\varepsilon _{1}$ so that:%
\begin{equation*}
\left\Vert \Lambda _{\omega _{\varepsilon }}F_{(\bar{\partial}%
_{E_{1}},h_{1})}-\Psi _{\omega _{\varepsilon _{1}}}^{HNS}(\bar{\partial}%
_{E},h_{1})\right\Vert _{_{L^{p}(\omega _{\varepsilon })}}<\frac{\delta }{3}.
\end{equation*}%
Firstly we have:%
\begin{equation*}
\left\Vert \Lambda _{\omega _{\varepsilon _{1}}}F_{(\bar{\partial}%
_{E_{1}},h_{1})}-\Psi _{\omega _{\varepsilon _{1}}}^{HNS}(\bar{\partial}%
_{E},h_{1})\right\Vert _{_{_{L^{\tilde{p}}(\omega _{\varepsilon
_{1}})}}}\leq C\bigl\Vert\Lambda _{\omega _{\varepsilon _{1}}}F_{(\bar{%
\partial}_{E_{1}},h_{1})}-\Psi _{\omega _{\varepsilon _{1}}}^{HNS}(\bar{%
\partial}_{E},h_{1})\bigr\Vert_{_{_{L^{2}(\omega _{\varepsilon
_{1}})}}}<\delta _{0}
\end{equation*}%
by H\"{o}lder's inequality (since $\tilde{p}<2$), and the induction
hypothesis. Note that the constant above is independent of $\varepsilon _{1}$
since the $\omega _{\varepsilon _{1}}$ volume is bounded. Also, the
following bound: 
\begin{eqnarray*}
\left\Vert F_{(\bar{\partial}_{E_{1}},h_{1})}\right\Vert _{L^{2}(\omega
_{\varepsilon _{1}})} &=&\left\Vert \Lambda _{\omega _{\varepsilon _{1}}}F_{(%
\bar{\partial}_{E_{1}},h_{1})}\right\Vert _{L^{2}(\omega _{\varepsilon
_{1}})}+\pi ^{2}n(n-1)\int_{\tilde{X}}\left(
2c_{2}(E_{1})-c_{1}^{2}(E_{1}\right) )\wedge \omega _{\varepsilon _{1}}^{n-2}
\\
&\leqslant &\left\Vert \Psi _{\omega _{\varepsilon _{1}}}^{HNS}(\bar{\partial%
}_{E},h_{1})\right\Vert _{L^{2}(\omega _{\varepsilon _{1}})}+\delta _{0}+\pi
^{2}n(n-1)\int_{\tilde{X}}\left( 2c_{2}(E_{1})-c_{1}^{2}(E_{1}\right)
)\wedge \omega _{\varepsilon _{1}}^{n-2}
\end{eqnarray*}%
obtained from the usual relationship between the Hermitian-Einstein tensor
and the full curvature in $L^{2}$, together with the induction hypothesis,
shows that this term is bounded in $\varepsilon _{1}$ as well. Finally,
writing%
\begin{eqnarray*}
\Lambda _{\omega _{\varepsilon _{1}}}\eta  &=&\frac{\eta \wedge \omega
_{\varepsilon _{1}}^{n-1}}{\omega _{\varepsilon _{1}}^{n}}=\frac{\eta \wedge
\omega _{\varepsilon _{1}}^{n-1}}{\eta ^{n}}\frac{\det g_{\eta }}{\det
g_{\varepsilon _{1}}} \\
\omega _{\varepsilon _{1}}^{n} &=&\frac{\det g_{\varepsilon _{1}}}{\det
g_{\eta }}\eta ^{n}
\end{eqnarray*}%
then by H\"{o}lder's inequality we have:%
\begin{equation*}
\left\Vert \Lambda _{\omega _{\varepsilon _{1}}}\eta \Psi _{\omega
_{\varepsilon _{1}}}^{HNS}(\bar{\partial}_{E},h_{1})\right\Vert _{L^{\tilde{p%
}}(\tilde{X},\omega _{\varepsilon _{1}})}\leq \left( \int_{\tilde{X}%
}\left\vert \frac{\det g_{\varepsilon _{1}}}{\det g_{\eta }}\right\vert ^{(1-%
\tilde{p})(\tilde{s})}\eta ^{n}\right) ^{\frac{1}{\tilde{p}\tilde{s}}}\left(
\int_{\tilde{X}}\left\vert \frac{\eta \wedge \omega _{\varepsilon _{1}}^{n-1}%
}{\eta ^{n}}\right\vert ^{w}\left\vert \Psi _{\omega _{\varepsilon
_{1}}}^{HNS}(\bar{\partial}_{E},h_{1})\right\vert ^{w}\eta ^{n}\right) ^{%
\frac{1}{w}}
\end{equation*}%
where $\tilde{s}=\frac{w}{w-\tilde{p}}$ and $\frac{\tilde{p}}{1-2(k-1)(%
\tilde{p}-1)}<w<\infty $. By Lemma \ref{Lemma11} this is bounded in $%
\varepsilon _{1}$.

We have already seen that 
\begin{equation*}
\bigl\Vert\Lambda _{\omega _{\varepsilon }}F_{(\bar{\partial}%
_{E_{1}},h_{R})}-\Lambda _{\omega _{\varepsilon }}F_{(\bar{\partial}%
_{E_{1}},h_{1})}\bigr\Vert_{_{L^{p}(\omega _{\varepsilon })}}
\end{equation*}%
can be estimated, since it is $0$ outside of $U_{R}$ and the same argument
as in the proof of $\limfunc{Proposition}$ \ref{Prop18}, shows that by
making $R$ sufficiently small, we can make the contribution from this term
over $U_{R}$ less than $\frac{\delta }{3}$. Therefore the estimate on $%
\bigl\Vert\sqrt{-1}\Lambda _{\omega }F_{(\bar{\partial}_{E},h)}-\Psi
_{\omega }^{HNS}(\bar{\partial}_{E},h)\bigr\Vert_{_{L^{p}(\omega )}}$ for
these values of $p$ follows by sending $\varepsilon \rightarrow 0$.

\textbf{Step 2 (Extending to all p)}

Repeating the arguments of Lemma \ref{Lemma13}, $\limfunc{Proposition}$ \ref%
{Prop19}, Lemma \ref{Lemma14}, and $\limfunc{Proposition}$ \ref{Prop20}, now
gives the existence of an $L^{p}$ $\delta $-approximate critical hermitian
structure on $E$ for each $p$. This metric will depend on $p$.
\end{proof}

Notice that during the course of the above proof we have also proven the
following:

\begin{theorem}
\label{Thm11}Let $E\rightarrow X$ be a holomorphic vector bundle over a K%
\"{a}hler manifold. Let $A_{t}$ be a solution to the $YM$ flow with initial
condition $A_{0}$ whose limit along the flow is $A_{\infty }$. Let $%
E_{\infty }$ be the corresponding holomorphic vector bundle defined away
from $Z_{\limfunc{an}}$. Then the $HN$ type of $(E_{\infty },A_{\infty })$
is the same as $(E_{0},A_{0})$.
\end{theorem}

\section{The Degenerate Yang-Mills Flow}

In this section we introduce a version of the Yang-Mills flow with respect
to the degenerate metric $\omega _{0}=\pi ^{\ast }\omega $ on a sequence of
blowups $\pi :\tilde{X}\rightarrow X$ along complex submanifolds. This flow
will solve the usual Hermitian-Yang-Mills flow equations on $\tilde{X}-%
\mathbf{E}$ with respect the metric $\omega $. It will be useful in the
proof of the main theorem, because we will again need to desingularise the $%
HNS$ filtration, and consider a sequence of blowups. The discussion in this
section is an extension of ideas in \cite{BS}.

Let $\pi :\tilde{X}\to X$ be a sequence of smooth blowups, and let $\omega
_{\varepsilon }$ be the usual family of K\"ahler metrics on $\tilde{X}$. We
will write $L_{k}^{p}(\tilde{X},\omega _{\varepsilon })$ for the
corresponding Sobolev spaces. The following lemma is clear.

\begin{lemma}
\label{Lemma15}Fix a compact subset $W\subset \subset \tilde{X}-\mathbf{E}$.
Let $\tilde{E}$ be a vector bundle. Then there exists a family of constants $%
C(\varepsilon )\rightarrow 0$ as $\varepsilon \rightarrow 0$, such that for
any $r$-form $F\in \Omega ^{r}(\tilde{X}-\mathbf{E},\tilde{E})$%
\begin{equation*}
\left( 1-C(\varepsilon )\right) \left\Vert F\right\Vert _{L_{k}^{p}(W,\omega
_{0})}\leq \left\Vert F\right\Vert _{L_{k}^{p}(W,\omega _{\varepsilon
})}\leq \left( 1+C(\varepsilon )\right) \left\Vert F\right\Vert
_{L_{k}^{p}(W,\omega _{0})}.
\end{equation*}
\end{lemma}

Throughout this section $\tilde{E}\rightarrow \tilde{X}$ will be a
holomorphic vector bundle of rank $K$, equipped with a smooth hermitian
metric $\tilde{h}_{0}$. Although later we will mainly be interested in the
case where $\tilde{E}=\pi ^{\ast }E\,$, we do not assume this here.

Note that $\left\Vert \Lambda _{\omega _{\varepsilon }}F_{(\bar{\partial}_{%
\tilde{E}},\tilde{h}_{0})}\right\Vert _{L^{1}(\omega _{\varepsilon })}$ is
uniformly bounded in $\varepsilon $, since for any fixed K\"{a}hler form $%
\varpi $ on $\tilde{X}$ we have:%
\begin{eqnarray*}
\left\vert \Lambda _{\omega _{\varepsilon }}F_{(\bar{\partial}_{\tilde{E}},%
\tilde{h}_{0})}\right\vert &=&\left\vert \frac{F_{(\bar{\partial}_{\tilde{E}%
},\tilde{h}_{0})}\wedge \omega _{\varepsilon }^{n-1}}{\omega _{\varepsilon
}^{n}}\right\vert =\left\vert \frac{F_{(\bar{\partial}_{\tilde{E}},\tilde{h}%
_{0})}\wedge \omega _{\varepsilon }^{n-1}}{\varpi ^{n}}\right\vert
\left\vert \frac{\det g_{\varpi }}{\det g_{\varepsilon }}\right\vert , \\
\omega _{\varepsilon }^{n} &=&\frac{\det g_{\varepsilon }}{\det g_{\varpi }}%
\varpi ^{n}
\end{eqnarray*}%
so 
\begin{equation*}
\left\Vert \Lambda _{\omega _{\varepsilon }}F_{(\bar{\partial}_{\tilde{E}},%
\tilde{h}_{0})}\right\Vert _{L^{1}(\omega _{\varepsilon })}=\int_{\tilde{X}%
}\left\vert \frac{F_{(\bar{\partial}_{\tilde{E}},\tilde{h}_{0})}\wedge
\omega _{\varepsilon }^{n-1}}{\varpi ^{n}}\right\vert \varpi ^{n}
\end{equation*}%
which is clearly bounded uniformly in $\varepsilon $. Write $\tilde{h}%
_{\varepsilon ,t}$ for the evolution of $\tilde{h}_{0}$ under the $HYM$ flow
with respect to the metric $\omega _{\varepsilon }$.

\begin{lemma}
\label{Lemma16}$(1)$ Let $t_{0}>0.$ Then $\left\vert \Lambda _{\omega
_{\varepsilon }}F_{(\bar{\partial}_{\tilde{E}},\tilde{h}_{\varepsilon
,t})}\right\vert $ is uniformly bounded for all $t\geq t_{0}>0$ and all $%
\varepsilon >0$. The bound depends only on $t_{0}$ and the uniform bound on $%
\left\Vert \Lambda _{\omega _{\varepsilon }}F_{(\bar{\partial}_{\tilde{E}},%
\tilde{h}_{0})}\right\Vert _{L^{1}(\omega _{\varepsilon })}$.

$(2)$ $\left\vert \Lambda _{\omega _{\varepsilon }}F_{(\bar{\partial}_{%
\tilde{E}},\tilde{h}_{\varepsilon ,t})}\right\vert $ is bounded uniformly on
compact subsets of $\tilde{X}-\mathbf{E}$ for all $t\geq 0$ and all $%
\varepsilon >0$. The bound depends only on the local bound on $\left\vert
\Lambda _{\omega _{\varepsilon }}F_{(\bar{\partial}_{\tilde{E}},\tilde{h}%
_{0})}\right\vert $ and the uniform bound on $\left\Vert \Lambda _{\omega
_{\varepsilon }}F_{(\bar{\partial}_{\tilde{E}},\tilde{h}_{0})}\right\Vert
_{L^{1}(\omega _{\varepsilon })}$.
\end{lemma}

\begin{proof}
By Lemma \ref{Lemm5} $(2)$, the pointwise norm $\left\vert \Lambda _{\omega
_{\varepsilon }}F_{(\bar{\partial}_{\tilde{E}},\tilde{h}_{\varepsilon
,t})}\right\vert $ is a subsolution of the heat equation on $(\tilde{X}%
,\omega _{\varepsilon })$ (see also \cite{BS} equation $3.3$). If $%
K_{t}^{\varepsilon }(x,y)$ is the heat kernel for the $\omega _{\varepsilon
} $ Laplacian on $\tilde{X}$ then 
\begin{equation*}
\int_{\tilde{X}}K_{t}^{\varepsilon }(x,y)\left\vert \Lambda _{\omega
_{\varepsilon }}F_{(\bar{\partial}_{\tilde{E}},\tilde{h}_{0})}\right\vert
(y)dvol_{\omega _{\varepsilon }}(y)
\end{equation*}%
is a solution of the heat equation and therefore:%
\begin{equation*}
\left\vert \Lambda _{\omega _{\varepsilon }}F_{(\bar{\partial}_{\tilde{E}},%
\tilde{h}_{\varepsilon ,t})}\right\vert (x)-\int_{\tilde{X}%
}K_{t}^{\varepsilon }(x,y)\left\vert \Lambda _{\omega _{\varepsilon }}F_{(%
\bar{\partial}_{\tilde{E}},\tilde{h}_{0})}\right\vert (y)dvol_{\omega
_{\varepsilon }}(y)
\end{equation*}%
is also a subsolution. Because 
\begin{equation*}
\int_{\tilde{X}}K_{0}^{\varepsilon }(x,y)\left\vert \Lambda _{\omega
_{\varepsilon }}F_{(\bar{\partial}_{\tilde{E}},\tilde{h}_{0})}\right\vert
(y)dvol_{\omega _{\varepsilon }}(y)=\left\vert \Lambda _{\omega
_{\varepsilon }}F_{(\bar{\partial}_{\tilde{E}},\tilde{h}_{0})}\right\vert
(x),
\end{equation*}%
the maximum principle for the heat equation now implies that 
\begin{equation*}
\left\vert \Lambda _{\omega _{\varepsilon }}F_{(\bar{\partial}_{\tilde{E}}%
\tilde{h}_{\varepsilon ,t})}\right\vert (x)\leq \int_{\tilde{X}%
}K_{t}^{\varepsilon }(x,y)\left\vert \Lambda _{\omega _{\varepsilon }}F_{(%
\bar{\partial}_{\tilde{E}},\tilde{h}_{0})}\right\vert (y)dvol_{\omega
_{\varepsilon }}(y).
\end{equation*}%
By \cite{BS} Lemma $4$, there is a bound: $K_{t}^{\varepsilon }(x,y)\leq
C\left( 1+1/t^{n}\right) $ for some constant $C$ independent of $\varepsilon 
$. Part $(1)$ now follows.

For part $(2)$, let $\Omega _{1}\subset \subset \Omega \subset \subset 
\tilde{X}-\mathbf{E}$, and let $\psi $ be a smooth cut-off function
supported in $\Omega $ and identically $1$ in a neighbourhood of $\bar{\Omega%
}_{1}$. Then just as in part $(1)$ we have:%
\begin{eqnarray*}
\left\vert \Lambda _{\omega _{\varepsilon }}F_{(\bar{\partial}_{\tilde{E}}%
\tilde{h}_{\varepsilon ,t})}\right\vert (x) &\leq &\int_{\tilde{X}%
}K_{t}^{\varepsilon }(x,y)\left\vert \Lambda _{\omega _{\varepsilon }}F_{(%
\bar{\partial}_{\tilde{E}},\tilde{h}_{0})}\right\vert (y)dvol_{\omega
_{\varepsilon }}(y) \\
&=&\int_{\tilde{X}}\psi K_{t}^{\varepsilon }(x,y)\left\vert \Lambda _{\omega
_{\varepsilon }}F_{(\bar{\partial}_{\tilde{E}},\tilde{h}_{0})}\right\vert
(y)dvol_{\omega _{\varepsilon }}(y) \\
&&+\int_{\tilde{X}}\left( 1-\psi \right) K_{t}^{\varepsilon }(x,y)\left\vert
\Lambda _{\omega _{\varepsilon }}F_{(\bar{\partial}_{\tilde{E}},\tilde{h}%
_{0})}\right\vert (y)dvol_{\omega _{\varepsilon }}(y).
\end{eqnarray*}%
By the maximum principle, the first term on the right hand side is bounded
from above by: \break $\sup \bigl\{\bigl\vert\Lambda _{\omega _{\varepsilon
}}F_{(\bar{\partial}_{\tilde{E}},\tilde{h}_{0})}\bigr\vert(y)\mid y\in
\Omega \bigr\}$. Since $\Omega \subset \subset \tilde{X}-\mathbf{E}$, the
function $1/\det g_{ij}^{\varepsilon }$ is uniformly bounded in $\varepsilon 
$, so this $\sup $ and hence the first integral above are uniformly bounded
in $\varepsilon $. By \cite{GR} Theorem $3.1$, there are positive constants $%
\delta $, $C_{1}$, $C_{2}$, independent of $t$ and $\varepsilon $, such that
for $x\neq y$,%
\begin{equation*}
K_{t}^{\varepsilon }(x,y)\leq C_{1}\left( 1+\frac{1}{\delta ^{2}t^{2}}%
\right) \exp \left( -\frac{\left( d_{\omega _{\varepsilon }}(x,y)\right) ^{2}%
}{C_{2}t}\right) .
\end{equation*}%
where $d_{\omega _{\varepsilon }}$ is the distance function on $\tilde{X}$
with respect to the Riemannian metric induced by $\omega _{\varepsilon }$.
Of course $d_{\omega _{\varepsilon }}(x,y)$ is bounded from below for $x\in
\Omega _{1}$ and $y\in \limfunc{supp}(1-\psi )$ uniformly in $\varepsilon $.
Therefore, $K_{t}^{\varepsilon }(x,y)$ is uniformly bounded in $\varepsilon $
and $t$, for these values of $x$ and $y$. Then the second term on the right
is uniformly bounded in terms of $\left\Vert \Lambda _{\omega }F_{(\bar{%
\partial}_{\tilde{E}}\tilde{h}_{0})}\right\Vert _{L^{1}(\omega _{\varepsilon
})}$, so $\left\vert \Lambda _{\omega _{\varepsilon }}F_{(\bar{\partial}_{%
\tilde{E}}\tilde{h}_{\varepsilon ,t})}\right\vert $ is uniformly bounded on $%
\Omega _{1}$.
\end{proof}

If we write $\tilde{h}_{\varepsilon ,t}=\tilde{k}_{\varepsilon ,t}\tilde{h}%
_{0}$, then it follows from the $HYM$ flow equations and the second part of
the previous lemma that both $\tilde{k}_{\varepsilon ,t}$ and $\tilde{k}%
_{\varepsilon ,t}^{-1}$ are uniformly bounded on compact subsets of $\tilde{X%
}-\mathbf{E}$ for $0\leq t\leq t_{0}$ (one sees easily that their
determinant and trace are bounded, which is enough). The statement that $%
\left\vert \Lambda _{\omega _{\varepsilon }}F_{(\bar{\partial}_{\tilde{E}}%
\tilde{h}_{\varepsilon ,t})}\right\vert $ is uniformly bounded on compact
subsets of $\tilde{X}-\mathbf{E}$ translates to the statement that there is
a section $f_{\varepsilon ,t}\in \mathfrak{u}(\tilde{E})$, uniformly bounded
on compact subsets of $\tilde{X}-\mathbf{E}$, such that:%
\begin{equation*}
\sqrt{-1}\Lambda _{\omega _{\varepsilon }}\bar{\partial}_{A_{0}}\left( 
\tilde{k}_{\varepsilon ,t}^{-1}\partial _{A_{0}}\tilde{k}_{\varepsilon
,t}\right) =f_{\varepsilon ,t},
\end{equation*}%
where $A_{0}$ is the connection $(\bar{\partial}_{E},\tilde{h}_{0})$. It
therefore follows from \cite{BS} $\limfunc{Proposition}$ $1$, that $\tilde{k}%
_{\varepsilon ,t}$ has a uniform $C^{1,\alpha }$ bound (for any $0<\alpha <1$%
) on compact subsets of $\left( \tilde{X}-\mathbf{E}\right) \times \left[
0,\infty \right) $. Furthermore, we may write:%
\begin{eqnarray*}
\sqrt{-1}\Lambda _{\omega _{\varepsilon }}\bar{\partial}_{A_{0}}\left( 
\tilde{k}_{\varepsilon ,t}^{-1}\partial _{A_{0}}\tilde{k}_{\varepsilon
,t}\right) &=&\tilde{k}_{\varepsilon ,t}^{-1}\sqrt{-1}\Lambda _{\omega
_{\varepsilon }}\left( \bar{\partial}_{A_{0}}\partial _{A_{0}}\tilde{k}%
_{\varepsilon ,t}\right) +\sqrt{-1}\Lambda _{\omega _{\varepsilon }}\left( 
\bar{\partial}_{A_{0}}\tilde{k}_{\varepsilon ,t}^{-1}\right) \left( \partial
_{A_{0}}\tilde{k}_{\varepsilon ,t}\right) \\
&=&\tilde{k}_{\varepsilon ,t}^{-1}\triangle _{(\bar{\partial}_{A_{0}},\omega
_{\varepsilon })}\tilde{k}_{\varepsilon ,t}-\tilde{k}_{\varepsilon ,t}^{-1}%
\sqrt{-1}\Lambda _{\omega _{\varepsilon }}\left( \bar{\partial}_{A_{0}}%
\tilde{k}_{\varepsilon ,t}\right) \tilde{k}_{\varepsilon ,t}^{-1}\left(
\partial _{A_{0}}\tilde{k}_{\varepsilon ,t}\right) ,
\end{eqnarray*}%
where in the last equality we have used the K\"{a}hler identities and the
expression for $\bar{\partial}_{A_{0}}\tilde{k}_{\varepsilon ,t}^{-1}$.
Therefore we have:%
\begin{equation*}
\triangle _{(\bar{\partial}_{A_{0}},\omega _{\varepsilon })}\tilde{k}%
_{\varepsilon ,t}-\sqrt{-1}\Lambda _{\omega _{\varepsilon }}\left( \bar{%
\partial}_{A_{0}}\tilde{k}_{\varepsilon ,t}\right) \tilde{k}_{\varepsilon
,t}^{-1}\left( \partial _{A_{0}}\tilde{k}_{\varepsilon ,t}\right) =\tilde{k}%
_{\varepsilon ,t}f_{\varepsilon ,t}.
\end{equation*}%
By elliptic regularity, this yields a uniform $L_{2}^{p}$ bound (for $%
1<p<\infty $) on $\tilde{k}_{\varepsilon ,t}$ on compact subsets of $\left( 
\tilde{X}-\mathbf{E}\right) \times \left[ 0,\infty \right) $. It now follows
from the $HYM$ the flow equations, that $\frac{\partial \tilde{h}%
_{\varepsilon ,t}}{\partial t}$ has a uniform $L^{p}$ bound on compact
subsets of $\left( \tilde{X}-\mathbf{E}\right) \times \left[ 0,\infty
\right) $, and so for any $W\subset \subset \left( \tilde{X}-\mathbf{E}%
\right) $ and $T\geq 0$, there is a uniform $L_{2/1}^{p}(W\times \left[
0,T\right) )$ bound on $\tilde{h}_{\varepsilon ,t}$, where the $2/1$ in the
previous notation refers to the fact that there is $1$ derivative in the
time variable and $2$ derivatives in the space variables. By weak
compactness, there is a subsequence $\varepsilon _{j}\rightarrow 0$, so that 
$\tilde{h}_{\varepsilon _{j},t}\rightarrow \tilde{h}_{t}$ weakly in $%
L_{2/1}^{p}$ on compact subsets. By the Sobolev imbedding theorem, $\tilde{h}%
_{\varepsilon _{j},t}\rightarrow \tilde{h}_{t}$ in $C^{1/0}$ on compact
subsets. By a further diagonalisation as $T\rightarrow \infty $, $\tilde{h}%
_{\varepsilon _{j},t}\rightarrow \tilde{h}_{t}$ for all $t\geq 0$.

\begin{definition}
\label{Def7}We will refer to the resulting limit $\tilde{h}_{t}$
corresponding to the initial metric $\tilde{h}_{0}$ and the degenerate
metric $\omega _{0}$ as the \textbf{degenerate Hermitian-Yang-Mills flow}.
\end{definition}

Of course a priori $\tilde{h}_{t}$ may depend on the subsequence $%
\varepsilon _{j}$. It is possible to show that under the assumption that $%
\Lambda _{\omega _{\varepsilon }}F_{\tilde{h}_{0}}$is uniformly bounded in $%
L^{\infty }$, $\tilde{h}_{t}$ is unique. This assumption will not be
satisfied in our case. We will show however that $\tilde{h}_{t}$ solves the $%
HYM$ equations on $\tilde{X}-\mathbf{E}$ with respect to the metric $\omega
_{0}$ with initial condition $\tilde{h}_{0}$.

\begin{lemma}
\label{Lemma17}Let $\tilde{h}_{t}$ be defined as above. Then $\tilde{h}_{t}$
is an hermitan metric on $\tilde{E}$ $\rightarrow $ $\tilde{X}-\mathbf{E}$
for all $t\geq 0$, and solves the $HYM$ equations on $\tilde{X}-\mathbf{E:}$%
\begin{equation*}
\tilde{h}_{t}^{-1}\frac{\partial \tilde{h}_{t}}{\partial t}=-2\left( \Lambda
_{\omega _{0}}F_{\tilde{h}_{t}}-\mu _{\omega _{0}}(E)\mathbf{Id}_{E}\right) .
\end{equation*}
\end{lemma}

\begin{proof}
Clearly $\tilde{h}_{t}$ is positive semi-definite since it is a limit of
metrics. Therefore we only need to check that $\det \tilde{h}_{t}$ is
positive. Taking the trace of both sides of the $HYM$ equations for the
metric $\omega _{\varepsilon }$, we get:%
\begin{equation*}
\frac{\partial }{\partial t}\left( \log \det \tilde{h}_{\varepsilon
,t}\right) =-2\limfunc{Tr}\left( \Lambda _{\omega _{\varepsilon }}F_{\tilde{h%
}_{\varepsilon ,t}}-\mu _{\omega _{\varepsilon }}(E)\mathbf{Id}_{E}\right)
\end{equation*}%
integrating both sides:%
\begin{equation*}
\left\vert \log \left( \frac{\det \tilde{h}_{\varepsilon ,T}}{\det \tilde{h}%
_{0}}\right) \right\vert =2\left\vert \int_{0}^{T}\limfunc{Tr}\left( \Lambda
_{\omega _{\varepsilon }}F_{\tilde{h}_{\varepsilon ,t}}-\mu _{\omega
_{\varepsilon }}(E)\mathbf{Id}_{E}\right) \right\vert .
\end{equation*}%
By the previous lemma, the right hand side is bounded uniformly in $%
\varepsilon $, so $\det \tilde{h}_{T}=\lim_{\varepsilon _{j\rightarrow
0}}\det \tilde{h}_{\varepsilon _{j},T}$ must be positive. Since $\tilde{h}%
_{\varepsilon _{j},t}\rightarrow \tilde{h}_{t}$ weakly in $L_{2/1}^{p}$ and $%
C^{1/0}$ it follows that $\tilde{h}_{t}$ solves the $HYM$ equations on $%
\tilde{X}-\mathbf{E}$.
\end{proof}

For the remainder of this section, we will write $F(-)$ for the curvature of
a metric in order to avoid a preponderance of subscripts.

\begin{lemma}
\label{Lemma18}$\left\Vert F(\tilde{h}_{t})\right\Vert _{L^{2}(\tilde{X}%
,\omega _{0})}$ and $\left\Vert \Lambda _{\omega _{0}}F(\tilde{h}%
_{t})\right\Vert _{L^{\infty }(\tilde{X},\omega _{0})}$ are uniformly
bounded for all $t\geq t_{0}>0$. The bound depends only on $t_{0}$ and the
uniform bound on $\left\Vert \Lambda _{\omega _{\varepsilon }}F(\tilde{h}%
_{0})\right\Vert _{L^{1}(\omega _{\varepsilon })}$.
\end{lemma}

\begin{proof}
Let $W\subset \subset \tilde{X}-\mathbf{E}$ be a compact subset. By
construction $F(\tilde{h}_{\varepsilon ,t})\rightarrow F(\tilde{h}_{t})$
weakly in $L^{2}(W,\omega _{0})$. Applying Lemma $\ref{Lemma15}$ and the
relation between $F(\tilde{h}_{\varepsilon ,t})$ and $\Lambda _{\omega
_{\varepsilon }}F(\tilde{h}_{\varepsilon ,t})$ in $L^{2}$ we have: 
\begin{eqnarray*}
\left\Vert F(\tilde{h}_{t})\right\Vert _{L^{2}(W,\omega _{0})} &\leq &\lim
\inf_{\varepsilon \rightarrow 0}\left\Vert F(\tilde{h}_{\varepsilon
,t})\right\Vert _{L^{2}(W,\omega _{0})}\leq C_{1}\lim \inf_{\varepsilon
\rightarrow 0}\left\Vert F(\tilde{h}_{\varepsilon ,t})\right\Vert
_{L^{2}(W,\omega _{\varepsilon })} \\
&\leq &C_{1}\lim \inf_{\varepsilon \rightarrow 0}\left\Vert F(\tilde{h}%
_{\varepsilon ,t})\right\Vert _{L^{2}(\tilde{X},\omega _{\varepsilon })}\leq
C_{1}\lim \inf_{\varepsilon \rightarrow 0}\left\Vert \Lambda _{\omega
_{\varepsilon }}F(\tilde{h}_{\varepsilon ,t})\right\Vert _{L^{2}(\tilde{X}%
,\omega _{\varepsilon })}+C_{2} \\
&\leq &C_{3}\lim \inf_{\varepsilon \rightarrow 0}\left\Vert \Lambda _{\omega
_{\varepsilon }}F(\tilde{h}_{\varepsilon ,t})\right\Vert _{L^{\infty }(%
\tilde{X})}+C_{2},
\end{eqnarray*}%
where $C_{3}$ is independent of $W$, and $C_{2}$ is the product of $C_{1}$
with a topological constant. The bound in $L^{2}$ now follows from Lemma $%
\ref{Lemma16}$ $(1)$.

For the second part again fix $W\subset \subset \tilde{X}-\mathbf{E}$. We
claim that for fixed $t$ and $W$, as $\varepsilon \rightarrow 0$ there is a
uniform bound 
\begin{equation*}
\left\Vert \Lambda _{\omega _{0}}F(\tilde{h}_{\varepsilon ,t})\right\Vert
_{L^{p}(W,\omega _{0})}\leq \left\Vert \Lambda _{\omega _{\varepsilon }}F(%
\tilde{h}_{\varepsilon ,t})\right\Vert _{L^{p}(W,\omega _{0})}+1.
\end{equation*}%
Otherwise, there is a sequence $\varepsilon _{j}$ such that:%
\begin{equation*}
\left\Vert \Lambda _{\omega _{0}}F(\tilde{h}_{\varepsilon
_{j},t})\right\Vert _{L^{p}(W,\omega _{0})}\geq \left\Vert \Lambda _{\omega
_{\varepsilon _{j}}}F(\tilde{h}_{\varepsilon _{j},t})\right\Vert
_{L^{p}(W,\omega _{0})}+1.
\end{equation*}%
Then%
\begin{equation*}
\left\vert \Lambda _{\omega _{0}}-\Lambda _{\omega _{\varepsilon
_{j}}}\right\vert \left\Vert F(\tilde{h}_{\varepsilon _{j},t})\right\Vert
_{L^{p}(W,\omega _{0})}\geq \left\Vert \left( \Lambda _{\omega _{0}}-\Lambda
_{\omega _{\varepsilon _{j}}}\right) \left( F(\tilde{h}_{\varepsilon
_{j},t})\right) \right\Vert _{L^{p}(W,\omega _{0})}\geq 1,
\end{equation*}%
where $\left\vert \Lambda _{\omega _{0}}-\Lambda _{\omega _{\varepsilon
_{j}}}\right\vert $ denotes the operator norm. Now we have $\tilde{h}%
_{\varepsilon _{j},t}\rightarrow \tilde{h}_{_{t}}$ weakly in $%
L_{2}^{p}(\omega _{0},W)$, so $\left\Vert F(\tilde{h}_{\varepsilon
_{j},t})\right\Vert _{L^{p}(W,\omega _{0})}$ is uniformly bounded. Since $%
\Lambda _{\omega _{\varepsilon _{j}}}\rightarrow \Lambda _{\omega _{0}}$ on $%
W$, this is a contradiction, and so we have proved the claim. Therefore:%
\begin{eqnarray*}
\left\Vert \Lambda _{\omega _{0}}F(\tilde{h}_{t})\right\Vert
_{L^{p}(W,\omega _{0})} &\leq &\lim \inf_{\varepsilon \rightarrow
0}\left\Vert \Lambda _{\omega _{0}}F(\tilde{h}_{\varepsilon ,t})\right\Vert
_{L^{p}(W,\omega _{0})}\leq \lim \inf_{\varepsilon \rightarrow 0}\left\Vert
\Lambda _{\omega _{\varepsilon }}F(\tilde{h}_{\varepsilon ,t})\right\Vert
_{L^{p}(W,\omega _{0})}+1 \\
&\leq &C\lim \inf_{\varepsilon \rightarrow 0}\left\Vert \Lambda _{\omega
_{\varepsilon }}F(\tilde{h}_{\varepsilon ,t})\right\Vert _{L^{\infty }(%
\tilde{X})}+1.
\end{eqnarray*}%
Taking $p\rightarrow \infty $, the lemma now follows from Lemma $\ref%
{Lemma16}$.
\end{proof}

\begin{proposition}
\label{Prop21}For almost all $t\geq t_{0}>0$, we have:%
\begin{equation*}
\left\Vert \nabla _{\left( \bar{\partial}_{\tilde{E}},\tilde{h}_{t}\right)
}\Lambda _{\omega _{0}}F(\tilde{h}_{t})\right\Vert _{L^{2}(\tilde{X},\omega
_{0})}\leq \lim \inf_{\varepsilon \rightarrow 0}\left\Vert \nabla _{\left( 
\bar{\partial}_{\tilde{E}},\tilde{h}_{\varepsilon ,t}\right) }\Lambda
_{\omega _{\varepsilon }}F(\tilde{h}_{\varepsilon ,t})\right\Vert _{L^{2}(%
\tilde{X},\omega _{\varepsilon })}<\infty .
\end{equation*}%
As will be seen in the course of the proof, this implies that: $\displaystyle%
\int_{t_{0}}^{\infty }\left\Vert \nabla _{\left( \bar{\partial}_{\tilde{E}},%
\tilde{h}_{t}\right) }\Lambda _{\omega _{0}}F(\tilde{h}_{t})\right\Vert
_{L^{2}(\omega _{0})}dt<\infty $.
\end{proposition}

\begin{proof}
By Lemma \ref{Lemm5} $(1)$ we have:%
\begin{equation*}
\frac{d}{dt}\left\Vert F(\tilde{h}_{\varepsilon ,t})\right\Vert _{L^{2}(%
\tilde{X},\omega _{\varepsilon })}^{2}=-2\left\Vert d_{\left( \bar{\partial}%
_{\tilde{E}},\tilde{h}_{t}\right) }^{\ast }F(\tilde{h}_{\varepsilon
,t})\right\Vert _{L^{2}(\tilde{X},\omega _{\varepsilon })}^{2}=-2\left\Vert
\nabla _{\left( \bar{\partial}_{\tilde{E}},\tilde{h}_{t}\right) }\Lambda
_{\omega _{\varepsilon }}F(\tilde{h}_{\varepsilon ,t})\right\Vert
_{L^{2}(\omega _{\varepsilon })}^{2}.
\end{equation*}%
Then:%
\begin{equation*}
2\int_{t_{0}}^{\infty }\left\Vert \nabla _{\left( \bar{\partial}_{\tilde{E}},%
\tilde{h}_{t}\right) }\Lambda _{\omega _{\varepsilon }}F(\tilde{h}%
_{\varepsilon ,t})\right\Vert _{L^{2}(\tilde{X},\omega _{\varepsilon
})}^{2}dt\leq \left\Vert F(\tilde{h}_{\varepsilon ,t_{0}})\right\Vert
_{L^{2}(\tilde{X},\omega _{\varepsilon })}^{2}\leq \left\Vert \Lambda
_{\omega _{\varepsilon }}F(\tilde{h}_{\varepsilon ,t_{0}})\right\Vert
_{L^{2}(\tilde{X},\omega _{\varepsilon })}^{2}+C.
\end{equation*}%
By Lemma \ref{Lemma16} $(1)$ the right hand side is uniformly bounded as $%
\varepsilon \rightarrow 0$. Then by Fatou's lemma we have:%
\begin{eqnarray*}
2\int_{t_{0}}^{\infty }\lim \inf_{\varepsilon \rightarrow 0}\left\Vert
\nabla _{\left( \bar{\partial}_{\tilde{E}},\tilde{h}_{t}\right) }\Lambda
_{\omega _{\varepsilon }}F(\tilde{h}_{\varepsilon ,t})\right\Vert _{L^{2}(%
\tilde{X},\omega _{\varepsilon })}^{2}dt &\leq &2\lim \inf_{\varepsilon
\rightarrow 0}\int_{t_{0}}^{\infty }\left\Vert \nabla _{\left( \bar{\partial}%
_{\tilde{E}},\tilde{h}_{t}\right) }\Lambda _{\omega _{\varepsilon }}F(\tilde{%
h}_{\varepsilon ,t})\right\Vert _{L^{2}(\tilde{X},\omega _{\varepsilon
})}^{2}dt \\
&\leq &\lim \inf_{\varepsilon \rightarrow 0}\left\Vert \Lambda _{\omega
_{\varepsilon }}F(\tilde{h}_{\varepsilon ,t_{0}})\right\Vert _{L^{2}(\tilde{X%
},\omega _{\varepsilon })}^{2}+C<\infty .
\end{eqnarray*}%
Therefore, for almost all $t\geq t_{0}$, we have:%
\begin{equation*}
\lim \inf_{\varepsilon \rightarrow 0}\left\Vert \nabla _{\left( \bar{\partial%
}_{\tilde{E}},\tilde{h}_{t}\right) }\Lambda _{\omega _{\varepsilon }}F(%
\tilde{h}_{\varepsilon ,t})\right\Vert _{L^{2}(\tilde{X},\omega
_{\varepsilon })}^{2}<\infty .
\end{equation*}%
Now we prove the first inequality:%
\begin{equation*}
\left\Vert \nabla _{\left( \bar{\partial}_{\tilde{E}},\tilde{h}_{t}\right)
}\Lambda _{\omega _{0}}F(\tilde{h}_{t})\right\Vert _{L^{2}(\tilde{X},\omega
_{0})}\leq \lim \inf_{\varepsilon \rightarrow 0}\left\Vert \nabla _{\left( 
\bar{\partial}_{\tilde{E}},\tilde{h}_{\varepsilon ,t}\right) }\Lambda
_{\omega _{\varepsilon }}F(\tilde{h}_{\varepsilon ,t})\right\Vert _{L^{2}(%
\tilde{X},\omega _{\varepsilon })}.
\end{equation*}%
It is enough to show this for an arbitrary compact subset $W\subset \subset 
\tilde{X}-\mathbf{E}$. For almost all $t\geq t_{0}$, we may choose a
sequence $\varepsilon _{j}\rightarrow 0$ such that 
\begin{equation*}
\lim_{j\rightarrow \infty }\left\Vert \nabla _{\left( \bar{\partial}_{\tilde{%
E}},\tilde{h}_{\varepsilon _{j},t}\right) }\Lambda _{\omega _{\varepsilon
_{j}}}F(\tilde{h}_{\varepsilon _{j},t})\right\Vert _{L^{2}(W,\omega
_{\varepsilon _{j}})}^{2}=\lim \inf_{\varepsilon \rightarrow 0}\left\Vert
\nabla _{\left( \bar{\partial}_{\tilde{E}},\tilde{h}_{\varepsilon ,t}\right)
}\Lambda _{\omega _{\varepsilon }}F(\tilde{h}_{\varepsilon ,t})\right\Vert
_{L^{2}(W,\omega _{\varepsilon })}^{2}=b<\infty .
\end{equation*}%
Since $\tilde{h}_{\varepsilon _{j,}t}\rightarrow \tilde{h}_{t}$ weakly in $%
L_{2}^{p}(\tilde{W})$, we have $\Lambda _{\omega _{0}}F_{\tilde{h}%
_{\varepsilon _{j},t}}\rightarrow \Lambda _{\omega _{0}}F_{\tilde{h}_{t}}$
weakly in $L^{p}(\tilde{W})$, and $\nabla _{\left( \bar{\partial}_{\tilde{E}%
},\tilde{h}_{\varepsilon _{j},t}\right) }\rightarrow \nabla _{\left( \bar{%
\partial}_{\tilde{E}},\tilde{h}_{t}\right) }$ in $C^{0}(W)$. It follows by
the triangle inequality and Lemma \ref{Lemma15}, that 
\begin{equation*}
\left\Vert \nabla _{\left( \bar{\partial}_{\tilde{E}},\tilde{h}_{t}\right)
}\Lambda _{\omega _{0}}F(\tilde{h}_{\varepsilon _{j},t})\right\Vert
_{L^{2}(W,\omega _{0})}\leq \left( 1+C_{j}\right) \left\Vert \nabla _{\left( 
\bar{\partial}_{\tilde{E}},\tilde{h}_{\varepsilon _{j},t}\right) }\Lambda
_{\omega _{\varepsilon _{j}}}F(\tilde{h}_{\varepsilon _{j},t})\right\Vert
_{L^{2}(W,\omega _{\varepsilon _{j}})}+c_{j}
\end{equation*}%
where $C_{j}$ and $c_{j}\rightarrow 0$. Then, $\left\Vert \Lambda _{\omega
_{0}}F(\tilde{h}_{\varepsilon _{j},t})\right\Vert _{L_{1}^{2}(W,h_{t},\omega
_{0})}$ is uniformly bounded as $j\rightarrow \infty $. Choose a subsequence
(still written $j$) such that $\Lambda _{\omega _{0}}F(\tilde{h}%
_{\varepsilon _{j},t})$ converges weakly in $L_{1}^{2}(W,\omega _{0})$. By
Rellich compactness we also have strong convergence $\Lambda _{\omega _{0}}F(%
\tilde{h}_{\varepsilon _{j},t})\rightarrow \Lambda _{\omega _{0}}F(\tilde{h}%
_{t})$ in $L^{2}(W)$. By the choice of $\varepsilon _{j}$ and the previous
inequality, we have $\left\Vert \nabla _{\left( \bar{\partial}_{\tilde{E}},%
\tilde{h}_{t}\right) }\Lambda _{\omega _{0}}F(\tilde{h}_{\varepsilon
_{j},t})\right\Vert _{L^{2}(W,\omega _{0})}^{2}\rightarrow b$. Then finally:%
\begin{eqnarray*}
\left\Vert \Lambda _{\omega _{0}}F(\tilde{h}_{t})\right\Vert
_{L_{1}^{2}(W,h_{t},\omega _{0})}^{2} &\leq &\lim \inf_{j\rightarrow \infty
}\left\Vert \Lambda _{\omega _{0}}F(\tilde{h}_{\varepsilon
_{j},t})\right\Vert _{L_{1}^{2}(W,h_{t},\omega _{0})}^{2} \\
&\leq &\lim \inf_{j\rightarrow \infty }\left( \left\Vert \Lambda _{\omega
_{0}}F(\tilde{h}_{\varepsilon _{j},t})\right\Vert _{L^{2}(W,\omega
_{0})}^{2}+\left\Vert \nabla _{\left( \bar{\partial}_{\tilde{E}},\tilde{h}%
_{t}\right) }\Lambda _{\omega _{0}}F(\tilde{h}_{\varepsilon
_{j},t})\right\Vert _{L^{2}(W,\omega _{0})}^{2}\right) \\
&\leq &\left\Vert \Lambda _{\omega _{0}}F_{\tilde{h}_{t}}\right\Vert
_{L^{2}(W,\omega _{0})}^{2}+b.
\end{eqnarray*}%
Since $\left\Vert \Lambda _{\omega _{0}}F(\tilde{h}_{t})\right\Vert
_{L_{1}^{2}(W,h_{t},\omega _{0})}^{2}=\left\Vert \Lambda _{\omega _{0}}F(%
\tilde{h}_{t})\right\Vert _{L^{2}(W,\omega _{0})}^{2}+\left\Vert \nabla
_{\left( \bar{\partial}_{\tilde{E}},\tilde{h}_{t}\right) }\Lambda _{\omega
_{0}}F(\tilde{h}_{t})\right\Vert _{L^{2}(W,\omega _{0})}^{2}$, we have 
\begin{equation*}
\left\Vert \nabla _{\left( \bar{\partial}_{\tilde{E}},\tilde{h}_{t}\right)
}\Lambda _{\omega _{0}}F(\tilde{h}_{t})\right\Vert _{L^{2}(W,\omega
_{0})}^{2}\leq b=\lim \inf_{\varepsilon \rightarrow 0}\left\Vert \nabla
_{\left( \bar{\partial}_{\tilde{E}},\tilde{h}_{\varepsilon ,t}\right)
}\Lambda _{\omega _{\varepsilon }}F(\tilde{h}_{\varepsilon ,t})\right\Vert
_{L^{2}(W,\omega _{\varepsilon })}^{2}.
\end{equation*}%
The second statement in the proposition now follows since:%
\begin{eqnarray*}
\int_{t_{0}}^{\infty }\left\Vert \nabla _{\left( \bar{\partial}_{\tilde{E}},%
\tilde{h}_{t}\right) }\Lambda _{\omega _{0}}F(\tilde{h}_{t})\right\Vert
_{L^{2}(\tilde{X},\omega _{0})}^{2}dt &\leq &\int_{t_{0}}^{\infty }\lim
\inf_{\varepsilon \rightarrow 0}\left\Vert \nabla _{\left( \bar{\partial}_{%
\tilde{E}},\tilde{h}_{\varepsilon ,t}\right) }\Lambda _{\omega _{\varepsilon
}}F(\tilde{h}_{\varepsilon ,t})\right\Vert _{L^{2}(\tilde{X},\omega
_{\varepsilon })}dt \\
(Fatou^{\prime }s\text{ }lemma) &\leq &\lim \inf_{\varepsilon \rightarrow
0}\int_{t_{0}}^{\infty }\left\Vert \nabla _{\left( \bar{\partial}_{\tilde{E}%
},\tilde{h}_{\varepsilon ,t}\right) }\Lambda _{\omega _{\varepsilon }}F(%
\tilde{h}_{\varepsilon ,t})\right\Vert _{L^{2}(\tilde{X},\omega
_{\varepsilon })}dt \\
&\leq &\lim \inf_{\varepsilon \rightarrow 0}\left\Vert \Lambda _{\omega
_{\varepsilon }}F(\tilde{h}_{\varepsilon ,t_{0}})\right\Vert _{L^{2}(\tilde{X%
},\omega _{\varepsilon })}^{2}+C<\infty .
\end{eqnarray*}
\end{proof}

The following is an immediate consequence.

\begin{corollary}
\label{Cor7} There is a sequence $t_{j}\rightarrow \infty $ such that $%
\displaystyle\bigl\Vert\nabla _{\left( \bar{\partial}_{\tilde{E}},\tilde{h}%
_{t_{j}}\right) }\Lambda _{\omega _{0}}F(\tilde{h}_{t})\bigr\Vert_{L^{2}(%
\tilde{X},\omega _{0})}\rightarrow 0$.
\end{corollary}

One result of all this discussion is the following corollary, which follows
from the previous corollary, Lemma \ref{Lemma18}, and Corollary \ref{Cor2}.
Although we will not use it in the sequel, we feel it is worth stating
explicitly.

\begin{corollary}
Let $t_{j}\longrightarrow \infty $ as in the previous corollary. Consider
the sequence $\tilde{A}_{t_{j}}=(\bar{\partial}_{\tilde{E}},\tilde{h}%
_{t_{j}})$ of connections defined over $\tilde{X}-\mathbf{E}=X-Z_{\limfunc{%
alg}}$. Then there is a further subsequence (still denoted $t_{j}$) such
that $\tilde{A}_{t_{j}}$ has an Uhlenbeck limit $\tilde{A}_{\infty }$ on a
reflexive sheaf $\tilde{E}_{\infty }$, which is a vector bundle away from a
set $\tilde{Z}_{\limfunc{an}}$ of Hausdorff codimension at least $4$. The
connection $\tilde{A}_{\infty }$ is Yang-Mills.
\end{corollary}

In the next section we will also need the following proposition.

\begin{proposition}
\label{Prop22}For almost all $t>0$, there is a sequence $\varepsilon
_{j}(t)\rightarrow 0$ such that $\Lambda _{\omega _{\varepsilon _{j}}}F(%
\tilde{h}_{\varepsilon _{j},t})\rightarrow \Lambda _{\omega _{0}}F(\tilde{h}%
_{t})$ in $L^{p}$ for all $1\leq p\leq \infty $. In particular: $HYM_{\alpha
}^{\omega _{\varepsilon _{j}}}\bigl(\nabla _{(\bar{\partial}_{\tilde{E}},%
\tilde{h}_{\varepsilon _{j}},_{t})}\bigr)\rightarrow HYM_{\alpha }^{\omega
_{0}}\bigl(\nabla _{(\bar{\partial}_{\tilde{E}},\tilde{h}_{t})}\bigr)$ for
all $\alpha $.
\end{proposition}

\begin{proof}
Fix $\delta >0$. Let $\tilde{U}$ be an open set containing $\mathbf{E}$ with 
$\limfunc{vol}(\tilde{U})<\frac{\delta }{3C}$ where $C$ is an upper bound on 
$\left\vert \Lambda _{\omega _{\varepsilon }}F_{\tilde{h}_{\varepsilon
,t}}\right\vert $ which exists by Lemma \ref{Lemma16}. Now let $t$, $%
\varepsilon _{j}$ be such that%
\begin{equation*}
\lim_{j\rightarrow \infty }\left\Vert \nabla _{\left( \bar{\partial}_{\tilde{%
E}},\tilde{h}_{\varepsilon _{j},t}\right) }\Lambda _{\omega _{\varepsilon
_{j}}}F(\tilde{h}_{\varepsilon _{j},t})\right\Vert _{L^{2}(W,\omega
_{\varepsilon _{j}})}^{2}=\lim \inf_{\varepsilon \rightarrow 0}\left\Vert
\nabla _{\left( \bar{\partial}_{\tilde{E}},\tilde{h}_{\varepsilon ,t}\right)
}\Lambda _{\omega _{\varepsilon }}F(\tilde{h}_{\varepsilon ,t})\right\Vert
_{L^{2}(W,\omega _{\varepsilon })}^{2}<\infty
\end{equation*}%
as in the proof of the previous proposition, where $W=\tilde{X}-\tilde{U}$.
Therefore, by the same argument as in the above proof we have strong
convergence $\Lambda _{\omega _{0}}F(\tilde{h}_{\varepsilon
_{j},t})\rightarrow \Lambda _{\omega _{0}}F(\tilde{h}_{t})$ in $%
L^{2}(W,\omega _{0})$. Therefore the same is true for $\Lambda _{\omega
_{\varepsilon _{j}}}F(\tilde{h}_{\varepsilon _{j},t})$. In particular there
exists a $J$ such that for $j,k\geq J$, we have:%
\begin{equation*}
\left\Vert \Lambda _{\omega _{\varepsilon _{j}}}F(\tilde{h}_{\varepsilon
_{j},t})-\Lambda _{\omega _{\varepsilon _{k}}}F(\tilde{h}_{\varepsilon
_{k},t})\right\Vert _{L^{2}(W,\omega _{0})}\leq \frac{\delta }{3}.
\end{equation*}%
By the choice of $\tilde{U}$, it follows that for $j,k\geq J$:%
\begin{equation*}
\left\Vert \Lambda _{\omega _{\varepsilon _{j}}}F(\tilde{h}_{\varepsilon
_{j},t})-\Lambda _{\omega _{\varepsilon _{k}}}F(\tilde{h}_{\varepsilon
_{k},t})\right\Vert _{L^{2}(\tilde{X},\omega _{0})}\leq \delta .
\end{equation*}%
Since $\Lambda _{\omega _{\varepsilon _{j}}}F_{\tilde{h}_{\varepsilon
_{j,t}}}$ is a Cauchy sequence it converges strongly in $L^{2}(\tilde{X}%
,\omega _{0})$. Since $\Lambda _{\omega _{\varepsilon _{j}}}F(\tilde{h}%
_{\varepsilon _{j},t})\rightarrow \Lambda _{\omega _{0}}F(\tilde{h}_{t})$
weakly in $L_{loc}^{2}(\tilde{X}-\mathbf{E},\omega _{0})$, it follows that $%
\Lambda _{\omega _{\varepsilon _{j}}}F(\tilde{h}_{\varepsilon
_{j},t})\rightarrow \Lambda _{\omega _{0}}F(\tilde{h}_{t})$ strongly in $%
L^{2}(\tilde{X}-\mathbf{E},\omega _{0})$. Since both $\Lambda _{\omega
_{\varepsilon _{j}}}F(\tilde{h}_{\varepsilon _{j},t})$ and $\Lambda _{\omega
_{0}}F(\tilde{h}_{t})$ are bounded in $L^{\infty }$ (see Lemma \ref{Lemma16}
and Lemma \ref{Lemma18}) it follows that $\Lambda _{\omega _{\varepsilon
_{j}}}F(\tilde{h}_{\varepsilon _{j},t})\rightarrow \Lambda _{\omega _{0}}F(%
\tilde{h}_{t})$ strongly in $L^{p}(\tilde{X}-\mathbf{E},\omega _{0})$ for
all $p$. By Lemma \ref{Lemma16} and Lemma \ref{Lemma9} we have:%
\begin{equation*}
HYM_{\alpha }^{\omega _{\varepsilon _{j}}}\left( \nabla _{(\bar{\partial}_{%
\tilde{E}},\tilde{h}_{\varepsilon _{j}},_{t})}\right) \longrightarrow
HYM_{\alpha }^{\omega _{0}}\left( \nabla _{(\bar{\partial}_{\tilde{E}},%
\tilde{h}_{t})}\right) .
\end{equation*}
\end{proof}

\section{Proof of the Main Theorem}

In this section we complete the proof of the main theorem. The result is a
direct corollary of the following theorem.

\begin{theorem}
\label{Thm12}Let $A_{0}$ be an integrable, unitary connection on a
holomorphic, hermitian vector bundle $E$ , $\mu _{0}$ the Harder-Narasimhan
type of $(E,\bar{\partial}_{A_{0}})$, and $A\subset \lbrack 1,\infty )$ be
any set containing an accumulation point. Let $A_{j}$ be a sequence of
integrable, unitary connections on $E$ such that:

$\bullet $ $(E,\bar{\partial}_{A_{j}})$ is holomorphically isomorphic to $(E,%
\bar{\partial}_{A_{0}})$ for all $i$;

$\bullet $ $HYM_{\alpha ,N}(A_{j})\longrightarrow HYM_{\alpha ,N}(\mu _{0})$
for all $\alpha \in A\cup \{2\}$ and all $N>0$.

Then there is a Yang-Mills connection $A_{\infty }$ on a bundle $E_{\infty }$
defined outside a

a closed subset of Hausdorff codimension at least $4$ such that:

$(1)$ $(E_{\infty },\bar{\partial}_{A_{\infty }})$ is isomorphic to $%
Gr^{HNS}(E,\bar{\partial}_{A_{0}})$ as a holomorphic bundle on

$\ \ \ X-Z_{\limfunc{an}}$;

$(2)$ After passing to a subsequence, $A_{j}\rightarrow A_{\infty }$ in $%
L_{loc}^{2}(X-Z_{\limfunc{an}})$;

$(3)$ There is an extension of the bundle $E_{\infty }$ to a reflexive sheaf

\ \ \ (still denoted $E_{\infty }$) such that $E_{\infty }\approxeq
Gr^{HNS}(E,\bar{\partial}_{A_{0}})^{\ast \ast }$.
\end{theorem}

The proof will be a modification of Donaldson's argument from \cite{DO1}
that there is a non-zero holomorphic map $(E,\bar{\partial}%
_{A_{0}})\rightarrow $ $(E_{\infty },\bar{\partial}_{A_{\infty }})$ in the
case that $(E,\bar{\partial}_{A_{0}})$ is semi-stable. If the bundles in
question are actually stable, we may then apply the elementary fact that a
non-zero holomorphic map between stable bundles with the same slope is
necessarily an isomorphism. Of course in our case $(E,\bar{\partial}%
_{A_{0}}) $ is not necessarily semi-stable so the argument must be modified.
We first construct such a map on the maximal destabilising subsheaf $%
S\subset E$ (which is semi-stable). If we assume that $S$ is stable (in
other words if we construct the map on the first piece of the $HNS$
filtration) this identifies $S$ with a subsheaf of the limiting sheaf $%
E_{\infty }$. We then use an inductive argument to identify each the
successive quotients with a direct summand of $E_{\infty }$. This is
relatively straightforward in the case that the $HNS$ filtration is given by
subbundles, but in the general case technical complications arise.
Therefore, to clearly illustrate our technique, we will first present an
exposition of the simpler case where there are no singularities, and then
explain the modifications necessary to complete the argument.

\subsection{The Subbundles Case}

We begin with the following proposition.

\begin{proposition}
\label{Prop23}Let $E$ be a holomorphic, hermitian vector bundle and $%
A_{j}=g_{j}(A_{0})$ be a sequence of integrable, unitary connections on $E$.
Let $A\subset \lbrack 1,\infty )$ be any set containing an accumulation
point. Assume that $HYM_{\alpha ,N}(A_{j})\rightarrow HYM_{\alpha ,N}(\mu
_{0})$ for all $N>0$ and all $\alpha \in A\cup \{2\}$. Let $S\subset (E,\bar{%
\partial}_{A_{0}})$ be a holomorphic subbundle. Then there is closed subset $%
Z_{\limfunc{an}}$ of Hausdorff codimension at least $4$, a reflexive sheaf $%
E_{\infty }$ which is an hermitian vector bundle away from $Z_{\limfunc{an}}$
and a Yang-Mills connection $A_{\infty }$ on $E_{\infty }$ such that:
\end{proposition}

$\ \ \ \ \ \ \ (1)$ After passing to a subsequence $A_{j}\rightarrow
A_{\infty }$ in $L_{loc}^{2}(X-Z_{\limfunc{an}});$

$\ \ \ \ \ \ \ (2)$ The Harder-Narasimhan type of $(E_{\infty },\bar{\partial%
}_{A_{\infty }})$ is the same as

\ \ \ \ \ \ \ \ \ \ \ that of $(E,\bar{\partial}_{A_{0}});$

$\ \ \ \ \ \ \ (3)$ There is a non-zero holomorphic map $g_{\infty
}^{S}:S\longrightarrow (E_{\infty },\bar{\partial}_{A_{\infty }})$.

\begin{proof}
We first reduce to the case where the Hermitian-Einstein tensors $\Lambda
_{\omega }F_{A_{j}}$ are uniformly bounded. Write $A_{j,t}$ for the time $t$
solution to the $YM$ flow equations with initial condition $A_{j}$. By Lemma %
\ref{Lemm5}, $\left\vert \Lambda _{\omega }F_{A_{j,t}}\right\vert ^{2}$ is a
sub-solution of the heat equation. Then for each $t>0$ and each $x\in X:$%
\begin{equation*}
\left\vert \Lambda _{\omega }F_{A_{j,t}}\right\vert ^{2}(x)\leq
\int_{X}K_{t}(x,y)\left\vert \Lambda _{\omega }F_{A_{j,t}}\right\vert
^{2}(y)dvol_{\omega }(y).
\end{equation*}%
Here $K_{t}(x,y)$ is the heat kernel on $X$. By a theorem of Cheng and Li
(see \cite{CHLI}) there is a bound:%
\begin{equation*}
0<K_{t}(x,y)\leq C\left( 1+\frac{1}{t^{n}}\right) ,
\end{equation*}%
and so for any fixed $t_{0}>0$, $\left\Vert \Lambda _{\omega
}F_{A_{j,t_{0}}}\right\Vert _{L^{\infty }(X,\omega )}$ is uniformly bounded
in terms of $\left\Vert \Lambda _{\omega }F_{A_{j}}\right\Vert
_{L^{2}(X,\omega )}$. Since we assume in particular that $%
HYM(A_{j})\rightarrow HYM(\mu _{0})$ we know that $\left\Vert \Lambda
_{\omega }F_{A_{j}}\right\Vert _{L^{2}(X,\omega )}$ is uniformly bounded
independently of $j,$ and therefore $\left\Vert \Lambda _{\omega
}F_{A_{j,t_{0}}}\right\Vert _{L^{\infty }(X,\omega )}$ is uniformly bounded.

For the remainder of the argument we would like to replace $A_{j}$ with $%
A_{j,t_{0}}$, so that we may assume in the sequel that we have the above
bound. In order to do this we must know that the Uhlenbeck limit of the new
sequence $A_{j,t_{0}}$ is the same as that of $A_{j}$. We argue as follows:%
\begin{align*}
\left\Vert A_{j,t_{o}}-A_{j}\right\Vert _{L^{2}}& \overset{Minkowski}{\leq }%
\int_{0}^{t_{0}}\left\Vert \frac{\partial A_{j,s}}{\partial s}\right\Vert
_{L^{2}}ds\leq \sqrt{t_{0}}\left( \int_{0}^{t_{0}}\left\Vert
d_{A_{j,s}}^{\ast }F_{A_{j,s}}\right\Vert _{L^{2}}^{2}ds\right) ^{\frac{1}{2}%
} \\
& =\sqrt{t_{0}}\left( \int_{0}^{t_{0}}-\frac{1}{2}\frac{d}{ds}\left\Vert
F_{A_{j,s}}\right\Vert _{L^{2}}^{2}ds\right) ^{\frac{1}{2}}=\sqrt{\frac{t_{0}%
}{2}}\left( YM(A_{j})-YM(A_{j,t_{0}})\right) ^{\frac{1}{2}}\longrightarrow 0
\end{align*}%
because $A_{j}$ is minimising for the $YM\ $functional and $YM$ is
non-increasing along the flow. This shows that the two limits are equal, and
moreover the proof also shows that $\bigl\Vert d_{A_{j,s}}^{\ast }F_{D_{j,s}}%
\bigr\Vert_{L^{2}}\rightarrow 0$ for almost all $s$, so this limit is a
Yang-Mills connection. Since we have assumed additionally that $HYM_{\alpha
,N}(A_{j})$ (and hence $HYM_{\alpha ,N}(A_{j,t_{0}})$) is minimising for $%
\alpha \in A$, it follows from $\limfunc{Propositions}$ \ref{Prop10} $(2)$
and \ref{Prop12} that the $HN$ type of $(E_{\infty },A_{\infty })$ is the
same as that of $(E_{0},A_{0})$.

We may therefore assume from here on out that the Hermitian-Einstein tensors 
$\Lambda _{\omega }F_{A_{j}}$ are uniformly bounded independently of $j$.
Note that we have already proven both $(1)$ and $(2)$ above. It remains to
construct the non-zero holomorphic map.

Observe that for any holomorphic section $\sigma $ of a holomorphic vector
bundle $V\longrightarrow (X,\omega )$ equipped with an hermitian metric $%
\left\langle -,-\right\rangle $, and whose Chern connection is $A$, we have
that 
\begin{eqnarray*}
\sqrt{-1}\bar{\partial}\partial \left\vert \sigma \right\vert ^{2} &=&\sqrt{%
-1}\bar{\partial}\partial \left\langle \sigma ,\sigma \right\rangle =\sqrt{-1%
}\left( \left\langle \partial _{A}\sigma ,\partial _{A}\sigma \right\rangle
+\left\langle \sigma ,\bar{\partial}_{A}\partial _{A}\sigma \right\rangle
\right) \\
&=&\sqrt{-1}\left( \left\langle \partial _{A}\sigma ,\partial _{A}\sigma
\right\rangle +\left\langle \sigma ,F_{A}\sigma \right\rangle \right)
\end{eqnarray*}%
since $\sigma $ is holomorphic. Applying $\Lambda _{\omega }$ and using the K%
\"{a}hler identities, we have:%
\begin{equation*}
\triangle _{\partial }\left\vert \sigma \right\vert ^{2}=\sqrt{-1}\Lambda
_{\omega }\bar{\partial}\partial \left\vert \sigma \right\vert
^{2}=-\left\vert \partial _{A}\sigma \right\vert ^{2}+\left\langle \sigma ,%
\sqrt{-1}\left( \Lambda _{\omega }F_{A}\right) \sigma \right\rangle .
\end{equation*}

Now let $g_{j}^{S}:S\rightarrow (E,\bar{\partial}_{A_{j}})$ be given by the
restriction of $g_{j}$ to $S$. By definition, this is a holomorphic section
of $\limfunc{Hom}(S,E)$, whose Chern connection is $A_{0}^{\ast }\otimes
A_{j}$. Then applying the above formula to $g_{j}^{S}$ and writing $%
k_{j}^{S}=(g_{j}^{S})^{\ast }(g_{j}{}^{S})$, and $h^{S}$ and $h_{j}$ for the
metrics corresponding to $A_{0\mid S}$ and $A_{j}$, we have%
\begin{equation*}
\triangle _{\partial }\limfunc{Tr}k_{j}^{S}+\left\vert \partial
_{A_{0}^{\ast }\otimes A_{j}}g_{j}^{S}\right\vert ^{2}=\left\langle
g_{j}^{S},\sqrt{-1}\left( \Lambda _{\omega
}F_{h_{j}}g_{j}^{S}-g_{j}^{S}\Lambda _{\omega }F_{h^{S}}\right)
\right\rangle ,
\end{equation*}%
and so%
\begin{equation*}
\triangle _{\partial }(\limfunc{Tr}k_{j}^{S})\leq (\limfunc{Tr}%
k_{j}^{S})\left( \left\vert \Lambda _{\omega }F_{h_{j}}\right\vert
+\left\vert \Lambda _{\omega }F_{h^{S}}\right\vert \right) .
\end{equation*}%
Now we use the bound on $\left\vert \Lambda _{\omega }F_{h_{j}}\right\vert $%
. Let $C_{1}=\sup_{j}\left\Vert \Lambda _{\omega }F_{h_{j}}\right\Vert
_{L^{\infty }(X,\omega )}$ and $C_{2}=\left\Vert \Lambda _{\omega
}F_{h^{S}}\right\Vert _{L^{\infty }(X)}$. Multiplying both sides of the
above inequality by $\limfunc{Tr}k_{j}^{S}$ and integrating by parts shows:%
\begin{equation*}
\int_{X}\left\vert \nabla \limfunc{Tr}k_{j}^{S}\right\vert ^{2}dvol_{\omega
}\leq (C_{1}+C_{2})\int_{X}\left\vert \limfunc{Tr}k_{j}^{S}\right\vert
^{2}dvol_{\omega }.
\end{equation*}%
By the Sobolev imbedding $L_{1}^{2}\hookrightarrow L^{\frac{^{2n}}{n-1}}$
the previous inequality gives a bound 
\begin{equation*}
\left\Vert \limfunc{Tr}k_{j}^{S}\right\Vert _{L^{\frac{^{2n}}{n-1}}(X,\omega
)}\leq C\left\Vert \limfunc{Tr}k_{j}^{S}\right\Vert _{L^{2}(X,\omega )}
\end{equation*}%
where $C$ depends only on $C_{1}$,$C_{2}$ and the Sobolev constant of $%
(X,\omega )$. A standard Moser iteration gives a bound: $\left\Vert \limfunc{%
Tr}k_{j}^{S}\right\Vert _{L^{\infty }(X,\omega )}\leq C\left\Vert \limfunc{Tr%
}k_{j}^{S}\right\Vert _{L^{2}(X,\omega )}$.

At this point we may repeat Donaldson's argument (appropriately modified for
higher dimensions). For the reader's convenience we reproduce it here. By
definition $\limfunc{Tr}(k_{j}^{S})=\left\vert g_{j}^{S}\right\vert ^{2}$.
Since non-zero constants act trivially on $\mathcal{A}^{1,1}$ we may
normalise the $g_{j}^{S}$ so that $\left\Vert g_{j}^{S}\right\Vert
_{L^{4}(X)}=\left\Vert \limfunc{Tr}(k_{j}^{S})\right\Vert _{L^{2}(X)}=1$.
The above bound implies that there is a subsequence of the $g_{j}^{S}$ that
converges to a limiting gauge transformation $g_{\infty }^{S}$ weakly in
every $L_{2}^{p}$ for example. Since $Z_{\limfunc{an}}$ has Hausdorff
codimension at least $4$, we may of course find a covering of $Z_{\limfunc{an%
}}$ by balls $\{B_{i}^{r}\}_{i}$ of radius $r$ such that: $C\left(
\sum_{i}Vol(B_{i}^{r})\right) <1/2$. If we write $K_{r}=X-\cup
_{i}B_{i}^{r}\cup \limfunc{Sing}(E_{\infty })$, then our $L^{\infty }$ bound
implies that: $\left\Vert g_{j}^{S}\right\Vert _{L^{4}(K_{r})}\geq 1/2$ for
all $j$. This implies that $g_{\infty }^{S}$ is non-zero. We now show $%
g_{\infty }^{S}$ is holomorphic.

If we denote by $\bar{\partial}_{A_{0}\otimes A_{\infty }}$ the $(0,1)$ part
of the connection on $E^{\ast }\otimes E_{\infty }=Hom(E,E_{\infty })$
induced by the connections $A_{0}$ and $A_{\infty }$. We will identify $E$
and $E_{\infty }$ on $K_{r}$. Then by definition we have:%
\begin{equation*}
\bar{\partial}_{A_{0}^{\ast }\otimes A_{\infty }}g_{j}^{S}=\left(
g_{j}^{S}A_{0}-A_{\infty }g_{j}^{S}\right)
=(g_{j}^{S}A_{0}(g_{j}^{S})^{-1}-A_{\infty })g_{j}^{S}=(A_{j}-A_{\infty
})g_{j}^{S}.
\end{equation*}%
Since $A_{0}\rightarrow A_{\infty }$ in $L^{2}(K_{r})$ this implies $\bar{%
\partial}_{A_{0}\otimes A_{\infty }}g_{\infty }^{S}=0$, in other words $%
g_{\infty }^{S}$ is holomorphic on $K_{r}$. Since this argument works for
any choice of $r$, and the $K_{r}$ give an exhaustion of $X-Z_{\limfunc{an}%
}\cup \limfunc{Sing}(E_{\infty })$, $g_{\infty }^{S}$ is holomorphic on $%
X-Z_{\limfunc{an}}\cup \limfunc{Sing}(E_{\infty })$. By a version of Hartogs
theorem (see \cite{SHI} Lemma $3$) there is an extension of $g_{\infty }^{S}$
to $X-\limfunc{Sing}(E_{\infty })$. Finally, by normality of these sheaves
(both are reflexive) there is an extension to a non-zero map $g_{\infty
}^{S}:S\rightarrow E_{\infty }$.
\end{proof}

We are now ready to perform the induction, and therefore prove the main
theorem in the case when the $HNS$ filtration is given by subbundles. We
first assume the quotients $Q_{i}=E_{i}/E_{i-1}$ in the Harder-Narasimhan
filtration $0=E_{0}\subset E_{1}\subset \cdots \subset E_{l}=(E,\bar{\partial%
}_{A_{0}})$ are stable (so the $HN$ and $HNS$ filtrations are the same).
From $\limfunc{Proposition}$ \ref{Prop7}, $E_{\infty }$ has a holomorphic
splitting $E_{\infty }=\oplus _{i=1}^{l^{^{\prime }}}Q_{\infty ,i}$. By
Theorem \ref{Thm11} the $HN$ types of $E$ and $E_{\infty }$ are the same, so 
$l=l^{^{\prime }}$ and $\mu (E_{1})=\mu (Q_{\infty ,1})>\mu (Q_{\infty ,i})$
for $i=2,\cdots ,l$. By the above proposition there is a non-zero
holomorphic map $g_{\infty }:E_{1}\rightarrow E_{\infty }$. Since we are
assuming $E_{1}$ is stable, and the $Q_{\infty ,i}$ ($i>1$) have slope
strictly smaller than $E_{1}$, the induced map onto these summands is $0$
and hence $g_{\infty }:E_{1}\rightarrow Q_{\infty ,1}$. Again by stability
of $E_{1}$ and $Q_{\infty ,1}$ and the fact that $E_{1}$ and $Q_{\infty ,1}$
have the same rank and degree, this map is an isomorphism. This is the first
step in the induction.

The inductive hypothesis will be that the connections $A_{j}$ restricted to $%
E_{i-1}$ converge to connections on the bundle $Gr(E_{i-1})$, in other words 
$Gr(E_{i-1})\subset E_{\infty }$. Let $E_{\infty ,i}=\oplus _{j\leq
i}Q_{\infty ,j}$ and set: $E_{\infty }=Gr(E_{i-1})\oplus R$, and consider
the short exact sequence of bundles: $0\rightarrow E_{i-1}\rightarrow
E_{i}\rightarrow Q_{i}\rightarrow 0$. Since $Gr(E_{i})=Gr(E_{i-1})\oplus
Q_{i}$, to complete the induction we need only show that $Q_{i}$ is a direct
summand of $R$. The sequence of connections on $E_{i}^{\ast }$ induced by $%
A_{j}$ satisfy the hypotheses of the proposition, so we may apply this
result to the dual exact sequence: $0\rightarrow Q_{i}^{\ast }\rightarrow
E_{i}^{\ast }\rightarrow E_{i-1}^{\ast }\rightarrow 0$, and therefore obtain
a holomorphic map $Q_{i}^{\ast }\rightarrow (E_{\infty ,i})^{\ast }$.
Because $Q_{i}^{\ast }$ is the maximal destabilising subsheaf of $(E_{\infty
,i})^{\ast }$ this implies that $Q_{i}^{\ast }$ is isomorphic to a summand
of $R^{\ast }$. This completes the proof under the assumption that the
quotients are stable.

To extend this to the general case, it suffices to consider the case that
the original bundle $(E_{,}\bar{\partial}_{A_{0}})$ is semi-stable. In other
words the filtration is a Seshadri filtration of $E$. Then as in the above
argument we may conclude that $E_{1}$ is isomorphic to a factor of $%
E_{\infty }$ we also again obtain a non-zero holomorphic map $g_{\infty
}:Q_{i}^{\ast }\rightarrow (E_{\infty ,i})^{\ast }$. However, the Seshadri
quotients all have the same slope, so we do not know via slope
considerations that $Q_{i}^{\ast }$ maps into $R^{\ast }$. On the other hand
we know that the weakly holomorphic projections converge. If $\pi
_{j}^{(i-1)}$ denotes the sequence of projections to $g_{j}(E_{i-1})$ and $%
\pi _{\infty }^{(i-1)}$ the projection onto $E_{\infty ,i-1}$, then $\pi
_{j}^{(i-1)}\rightarrow $ $\pi _{\infty }^{(i-1)}$ by the proof of Lemma $%
4.5 $ of \cite{DW1}. If we denote by $\check{\pi}_{j}^{(i-1)}$ the dual
projection, then for each $j$, the image of $Q_{i}^{\ast }$ is in the kernel
of $\check{\pi}_{j}^{(i-1)}$. In other words the image $g_{\infty
}(Q_{i}^{\ast })$ lies in the kernel of $\check{\pi}_{j}^{(i-1)}$. Therefore
since we have convergence, the image of $g_{\infty }(Q_{i}^{\ast })$ lies in
the kernel of $\check{\pi}_{\infty }^{(i-1)}$ which is in $R^{\ast }$.
Therefore $Q_{i}^{\ast }$ is isomorphic to a factor of $R^{\ast }$ and this
completes the proof.

\subsection{The General Case}

In general the $HNS$ filtration is not given by subbundles. The argument we
have given in $\limfunc{Proposition}$ \ref{Prop23} for the construction of
the holomorphic map $S\rightarrow E_{\infty }$ remains valid if $S$ is an
arbitrary torsion free subsheaf since the connections in question are all
defined a priori on the ambient bundle $E$, and since the second fundamental
form $\beta $ of $S$ drops out of the estimates, there is no problem
obtaining a uniform bound on the Hermitian-Einstein tensors. On the other
hand, when we try to run the inductive argument, the restrictions of the
connections $A_{j}$ to the pieces $E_{i}$ of the $HNS$ filtration only make
sense on the locally free part of these subsheaves. This prevents us from
applying the argument of $\limfunc{Proposition}\ $\ref{Prop23} in the
inductive step because to do so requires global $L^{\infty }$ bounds on the
appropriate Hermitian-Einstein tensors, which we do not have, since the
restrictions of the $A_{j}$ do not extend over the singular set $Z_{\limfunc{%
alg}}$.

The strategy for proving the main theorem in the general case mirrors our
method in Section $4$. Roughly speaking we proceed as follows. Let $%
A_{j}=g_{j}(A_{0})$ be a sequence of connections. First we pass to an
arbitrary resolution $\pi :\tilde{X}\rightarrow X$ of singularities of the $%
HNS$ filtration. Then we construct an isomorphism from the associated graded
object of the filtration for the pullback bundle $\pi ^{\ast }E$ (away from
the exceptional set $\mathbf{E}$) to the Uhlenbeck limit of the sequence $%
\pi ^{\ast }A_{j}$ on the K\"{a}hler manifold $(\tilde{X}-\mathbf{E},\omega
_{0})=(X-Z_{\limfunc{alg}},\omega )$ where $\omega _{0}=\pi ^{\ast }\omega $%
. Then we will use the fact that these bundles extend as reflexive sheaves
over $Z_{\limfunc{alg}}$ to the double dual of the associated graded object
of $E$ and the Uhlenbeck limit of $A_{j}$ respectively, and hence by
normality of these sheaves, the isomorphism extends as well.

The outline of the proof given above has to be modified somewhat for
technical reasons which we will now explain. Just as for the case of
subbundles, by first running the $YM$ flow for finite time we may assume
there is a uniform bound $\left\Vert \Lambda _{\omega }F_{A_{j}}\right\Vert
_{L^{\infty }(X)}$ or equivalently on $\left\Vert \Lambda _{\omega _{0}}F_{%
\tilde{A}_{j}}\right\Vert _{L^{\infty }(\tilde{X}-\mathbf{E})}$ where $%
\tilde{A}_{j}=\pi ^{\ast }A_{j}$. As usual we will denote by $A_{\infty }$
the Uhlenbeck limit of $A_{j}$ on $(X,\omega )$ and we have $%
A_{j}\rightarrow A_{\infty }$ in $L_{1,loc}^{p}(X-Z_{\limfunc{an}})$ for $%
p>n $. The proof of the proposition proves all but $(3)$ of Theorem \ref%
{Thm12}. Let $E_{i}\subset E$ be a factor of the $HNS$ filtration and $%
A_{j}^{(i)}=\pi _{j}^{(i)}A_{j}$ be the connections on $g_{j}(E_{i})$
induced from $A_{j}$, and $A_{\infty }^{(i)}=\pi _{\infty }^{(i)}A_{\infty }$%
. By Lemma \ref{Lemma14} it follows that $A_{j}^{(i)}\rightarrow A_{\infty
}^{(i)}$ weakly in $L_{1,loc}^{p}(X-Z_{\limfunc{an}}\cup Z_{\limfunc{alg}})$.

If $\pi :\tilde{X}\rightarrow X$ is the aforementioned resolution of
singularities then the filtration of $\pi ^{\ast }E=\tilde{E}$ is given by
subbundles $\tilde{E}_{i}\subset \tilde{E}$, isomorphic to $E_{i}$ away from
the exceptional divisor $\mathbf{E}$. Write $\tilde{g}_{j}=g_{j}\circ \pi $
and let $\tilde{A}_{j}^{(i)}$ be the connection induced by $\tilde{A}%
_{j}=\pi ^{\ast }A_{j}$ on $\tilde{g}_{j}(\tilde{E}_{i})$. We will write $%
\tilde{\pi}_{j}$ for the projection to $\tilde{g}_{j}(\tilde{E}_{i})$ and $%
\tilde{\beta}_{j}$ for the second fundamental forms for the connections $%
\tilde{A}_{j}$ with respect to the subbundles $\tilde{E}_{i}$; in other
words these are sections of the bundle $\Omega ^{0,1}\left( \tilde{X},Hom(%
\tilde{Q}_{i},\tilde{E}_{i})\right) $ for an auxiliary bundle $\tilde{Q}_{i}$%
. Then this sequence of connections satisfies the following:%
\begin{equation*}
\end{equation*}

$(1)$ There is a closed subset $\tilde{Z}_{\limfunc{an}}\subset \tilde{X}-%
\mathbf{E}$ of Hausdorff codimension at least $4$

\ \ \ and a Yang-Mills connection $\tilde{A}_{\infty }^{(i)}$ defined on a
reflexive sheaf $\tilde{E}_{\infty ,i}\rightarrow \tilde{X}-\mathbf{E}$
(which is a bundle on $\tilde{X}-(\tilde{Z}_{\limfunc{an}}\cup \mathbf{E)}$%
), such

\ \ \ that $\tilde{A}_{j}^{(i)}\rightarrow \tilde{A}_{\infty }^{(i)}$ weakly
in $L_{1,loc}^{p}\left( \tilde{X}-(\tilde{Z}_{\limfunc{an}}\cup \mathbf{E)}%
\right) $.

$(2)$ We have the standard formula for the curvature:%
\begin{equation*}
\sqrt{-1}\Lambda _{\omega _{0}}F_{\tilde{A}_{j}^{(i)}}=\sqrt{-1}\Lambda
_{\omega _{0}}\left( \tilde{\pi}_{j}F_{\tilde{A}_{j}}\tilde{\pi}_{j}\right) +%
\sqrt{-1}\Lambda _{\omega _{0}}\left( \tilde{\beta}_{j}\wedge \tilde{\beta}%
_{j}^{\ast }\right) .
\end{equation*}

\qquad \qquad Also:

$\qquad \bullet $ The $\tilde{\beta}_{j}$ are locally bounded on $\tilde{X}-(%
\tilde{Z}_{\limfunc{an}}\cup \mathbf{E)}$ uniformly in $j$ (Lemma \ref%
{Lemma4})

$\qquad \bullet $ The $\tilde{\beta}_{j}\rightarrow 0$ in $L^{2}(\omega
_{0}) $. In particular, they are uniformly bounded in $L^{2}(\omega _{0})$
(see the proof of \cite{DW1} Lemma $4.5$). 
\begin{equation*}
\end{equation*}

Note that the term%
\begin{equation*}
\sqrt{-1}\Lambda _{\omega _{0}}\left( \tilde{\pi}_{j}F_{\tilde{A}_{j}}\tilde{%
\pi}_{j}\right)
\end{equation*}%
is bounded in $L^{\infty }(\tilde{X},\omega _{0}\mathbf{)}$ since $\tilde{A}%
_{j}=\pi ^{\ast }A_{j}$. The key point here is that term 
\begin{equation*}
\sqrt{-1}\Lambda _{\omega _{0}}\left( \tilde{\beta}_{j}\wedge \tilde{\beta}%
_{j}^{\ast }\right)
\end{equation*}%
is not bounded in $L^{\infty }(\tilde{X},\omega _{0})$ since it may be
written as%
\begin{equation*}
\sqrt{-1}\frac{\left( \tilde{\beta}_{j}\wedge \tilde{\beta}_{j}^{\ast
}\right) \wedge \omega _{0}^{n-1}}{\omega _{0}^{n}}
\end{equation*}%
which blows up near $\mathbf{E}$. This is a problem because in order to
carry out the induction in the preceding sub-section we had to consider
exact sequences of the form:%
\begin{equation*}
0\longrightarrow \tilde{Q}_{i}^{\ast }\longrightarrow \tilde{E}_{i}^{\ast
}\longrightarrow \tilde{E}_{i-1}^{\ast }\longrightarrow 0
\end{equation*}%
(here $\tilde{Q}_{i}=\tilde{E}_{i}/\tilde{E}_{i-1}$) and apply $\limfunc{%
Proposition}$ \ref{Prop23} to construct a non-zero holomorphic map $\tilde{Q}%
_{i}^{\ast }\rightarrow \tilde{E}_{\infty ,i}^{\ast }$. This involved
knowing that there was a uniform $L^{\infty }$ bound on the
Hermitian-Einstein tensors of the induced connections $(\tilde{A}%
_{j}^{(i)})^{\ast }$ and $(\tilde{A}_{j,Q}^{(i)})^{\ast }$ on $\tilde{E}%
_{i}^{\ast }$ and $\tilde{Q}_{i}^{\ast }$. Since this is not the case we
cannot apply this argument directly. On the other hand we do know that for
all positive times $t>0$, the degenerate Yang-Mills flow of Section $6$
gives connections $\tilde{A}_{j,t}^{(i)}$ such that $\Lambda _{\omega
_{0}}F_{\tilde{A}_{j,t}^{(i)}}$ is uniformly bounded (see Lemma \ref{Lemma18}%
). For each $t$ the deformed sequence of connections has an Uhlenbeck limit $%
\tilde{A}_{\infty ,t}^{(i)}$ on a reflexive sheaf $\tilde{E}_{\infty ,i}^{t}$
which a priori depends on $t$.

There are now two points to address. In parallel to $\limfunc{Proposition}$ %
\ref{Prop23} we will show that after resolving the singularities of the
maximal destabilising subsheaf $S$ to a bundle $\tilde{S}$ there is a
non-zero holomorphic map $\tilde{S}\rightarrow \tilde{E}_{\infty }^{t}$
(where $\tilde{E}_{\infty }^{t}$ is an Uhlenbeck limit of $\tilde{A}_{j,t}$)
away from $\mathbf{E}$. This is not automatic from the proof of Proposition %
\ref{Prop23} because the connections $\tilde{A}_{j,t}$ do not extend
smoothly across $\mathbf{E}$, so the integration by parts involved in the
proof is not valid. We will instead derive this map as a limit of the maps
produced from the corresponding argument for the family of K\"{a}hler
manifolds $(\tilde{X},\omega _{\varepsilon })$. Secondly we need to know
that the Uhlenbeck limits $(\tilde{E}_{\infty }^{t},\tilde{A}_{\infty ,t})$
are independent of $t$ and are all equal to $(\tilde{E}_{\infty },\tilde{A}%
_{\infty })$. Again, this does not follow from our previous argument since,
as we have noted, the second fundamental forms of the restricted connections
are only bounded in $L^{2}$ and therefore the curvatures are only bounded in 
$L^{1}$. In particular we do not have that $\tilde{A}_{j}^{(i)}$ is
minimising for the functional $YM$. Establishing these two facts will
complete the proof of the main theorem, since then we may use induction just
as for the case when the $HNS$ filtration is given by subbundles.

We begin with the first point.

\begin{proposition}
\label{Prop24}Let $\tilde{E}\rightarrow \tilde{X}$ be a vector bundle with
an hermitian metric $\tilde{h}$. Let $\tilde{A}_{j}=\tilde{g}_{j}(\tilde{A}%
_{0})$ be a sequence of unitary connections on $\tilde{E}$, and assume $%
\Lambda _{\omega _{0}}F_{\tilde{A}_{j}}$ is bounded uniformly in $j$ in $%
L^{1}(\tilde{X},\omega _{0})$. Let $\tilde{A}_{j,t}$ be the solution of the
degenerate $YM$ flow at time $t$ with initial condition $\tilde{A}_{j}$, and
suppose that this sequence has an Uhlenbeck limit $(\tilde{E}_{\infty }^{t},%
\tilde{A}_{\infty ,t})$. Finally let $\tilde{S}\subset \tilde{E}$ be a
subbundle of $(\tilde{E},\tilde{A}_{0})$. Then there is a non-zero
holomorphic map $\tilde{g}_{\infty }:\tilde{S}\rightarrow \tilde{E}_{\infty
}^{t}$ on $\tilde{X}-\mathbf{E}$. Furthermore, assume that $(\tilde{E}%
_{\infty }^{t},\tilde{A}_{\infty ,t})$ has an extension $(E_{\infty
}^{t},A_{\infty ,t})$ as a reflexive sheaf over $Z_{\limfunc{alg}}$ to $X$,
assume $\tilde{S}$ also extends to a reflexive sheaf $S$ on $X$. Then $%
\tilde{g}_{\infty }$ induces a non-zero holomorphic map $g_{\infty
}:S\rightarrow E_{\infty }^{t}$.
\end{proposition}

\begin{proof}
Let $\omega _{\varepsilon }$ be the standard family of K\"{a}hler metrics on 
$\tilde{X}$ and fix $t>0$. Let $\varepsilon _{i}\rightarrow 0$ be a sequence
as in Section $6$, i.e. if $\tilde{A}_{j,t}^{\varepsilon _{i}}$ is the time $%
t$ $YM$ flow on $(\tilde{X},\omega _{\varepsilon _{i}})$, then $\tilde{A}%
_{j,t}^{\varepsilon _{i}}\rightarrow \tilde{A}_{j,t}$ in $C^{1/0}$ on
compact subsets of $\tilde{X}-\mathbf{E}$. Choose a family of metrics $%
\tilde{h}_{\varepsilon _{i}}^{\tilde{S}}$ on $\tilde{S}$ converging
uniformly on compact subsets of $\tilde{X}-\mathbf{E}$ to a metric $\tilde{h}%
_{0}^{\tilde{S}}$ defined away from $\mathbf{E}$, and such that $\sup
\left\vert \Lambda _{\omega _{\varepsilon i}}F_{\tilde{h}_{\varepsilon
_{i}}^{\tilde{S}}}\right\vert $ is uniformly bounded as $\varepsilon
_{i}\rightarrow 0$ (take for example the time $1$ $HYM$ flow of $\tilde{h}$
with respect $\omega _{\varepsilon }$). For each $j$ and each $\varepsilon
_{i}>0$, we have a non-zero holomorphic map $\tilde{g}_{\varepsilon _{i},j}^{%
\tilde{S}}:\tilde{S}\rightarrow (\tilde{E},\bar{\partial}_{\tilde{A}%
_{j,t}^{\varepsilon _{i}}})$. Just as in Section $7.1$, we set $%
k_{\varepsilon _{i},j}^{\tilde{S}}=$ $\left( \tilde{g}_{\varepsilon _{i},j}^{%
\tilde{S}}\right) ^{\ast }\tilde{g}_{\varepsilon _{i},j}^{\tilde{S}}$. As in 
$\limfunc{Proposition}$ \ref{Prop23} we have the inequality:%
\begin{equation*}
\Delta _{(\partial ,\omega _{\varepsilon })}(\limfunc{Tr}\tilde{k}%
_{\varepsilon _{i},j}^{\tilde{S}})\leq (\limfunc{Tr}\tilde{k}_{\varepsilon
_{i},j}^{\tilde{S}})\left( \left\vert \Lambda _{\omega _{\varepsilon
_{i}}}F_{\tilde{A}_{j,t}^{\varepsilon _{i}}}\right\vert +\left\vert \Lambda
_{\omega _{\varepsilon i}}F_{\tilde{h}_{\varepsilon _{i}}^{\tilde{S}%
}}\right\vert \right) .
\end{equation*}

Both factors on the right are uniformly bounded as $\varepsilon
_{i}\rightarrow 0$ by assumption. It follows that we have the inequality: $%
\bigl\Vert\limfunc{Tr}\tilde{k}_{\varepsilon _{i},j}^{\tilde{S}}\bigr\Vert%
_{L^{\infty }(\tilde{X})}\leq C\bigl\Vert\limfunc{Tr}\tilde{k}_{\varepsilon
_{i},j}^{\tilde{S}}\bigr\Vert_{L^{2}(\tilde{X},\omega _{\varepsilon })}$,
where the constant $C$ depends only on these uniform bounds and the Sobolev
constant of $(\tilde{X},\omega _{\varepsilon _{i}})$ is also uniformly
bounded away from zero by \cite{BS} Lemma $3$. As in the proof of $\limfunc{%
Proposition}\ $\ref{Prop23} we rescale $\tilde{g}_{\varepsilon _{i},j}^{%
\tilde{S}}$ so that $\bigl\Vert\tilde{g}_{\varepsilon _{i},j}^{\tilde{S}}%
\bigr\Vert_{L^{4}(\tilde{X},\omega _{\varepsilon })}=1$. A diagonalisation
argument for an exhaustion of $\tilde{X}-\mathbf{E}$ together with the $\sup 
$ bound gives a sequence of non-zero holomorphic maps $\tilde{g}_{j}^{\tilde{%
S}}:\tilde{S}\rightarrow (\tilde{E},\tilde{\partial}_{\tilde{A}_{j,t}})$
defined on $\tilde{X}-\mathbf{E}$ with $\tilde{g}_{\varepsilon _{i},j}^{%
\tilde{S}}\rightarrow $ $\tilde{g}_{j}^{\tilde{S}}$ uniformly on compact
subsets as $\varepsilon _{i}\rightarrow 0$ such that: $\bigl\Vert\tilde{g}%
_{j}^{\tilde{S}}\bigr\Vert_{L^{\infty }}\leq C$, and $\bigl\Vert\tilde{g}%
_{j}^{\tilde{S}}\bigr\Vert_{L^{4}(\omega _{0})}=1$. Repeating the proof of $%
\limfunc{Proposition}$ \ref{Prop23} yields a nonzero limit $\tilde{g}%
_{\infty }^{\tilde{S}}:\tilde{S}\rightarrow (\tilde{E}_{\infty }^{t},\tilde{A%
}_{\infty ,t})$. The last statement follows from the normality of the
sheaves in question.
\end{proof}

Secondly we have:

\begin{proposition}
\label{Prop25}Let $\tilde{E}\rightarrow \tilde{X}$ be a Hermitian vector
bundle with a unitary integrable connection $\tilde{A}_{0}$. We assume that
the holomorphic bundle $(\tilde{E},\bar{\partial}_{A_{0}})$ restricted to $%
\tilde{X}-\mathbf{E}=X-Z_{\limfunc{alg}}$ extends to a holomorphic bundle $%
(E,\bar{\partial}_{E})$ on $X$ with Harder-Narasimhan type $\mu =(\mu
_{1},\cdots ,\mu _{R}).$ Let $\tilde{A}_{j}=\tilde{g}_{j}(\tilde{A}_{0})$ be
a sequence of unitary connections on $\tilde{E}$, and assume there is a
subset $\tilde{Z}_{\limfunc{an}}\subset \tilde{X}-\mathbf{E}$ of Hausdorff
codimension at least $4$, and a $YM$ connection $\tilde{A}_{\infty }$ on a
bundle $\tilde{E}_{\infty }\rightarrow \tilde{X}-(\tilde{Z}_{\limfunc{an}%
}\cup \mathbf{E})$ such that $\tilde{A}_{j}\rightarrow \tilde{A}_{\infty }$
weakly in $L_{1,loc}^{p}$ (where $p>n$) on compact subsets of $\tilde{X}-(%
\tilde{Z}_{\limfunc{an}}\cup \mathbf{E)}$. We assume that the constant
eigenvalues of $\sqrt{-1}\Lambda _{\omega _{0}}F_{\tilde{A}_{\infty }}$ are
given by the vector $\mu $. Finally assume $\Lambda _{\omega _{0}}F_{\tilde{A%
}_{j}}\rightarrow \Lambda _{\omega _{0}}F_{\tilde{A}_{\infty }}$ in $%
L^{1}(\omega _{0})$. Then there is a subsequence such that for almost all $%
t>0$, $\tilde{A}_{j,t}\rightarrow \tilde{A}_{\infty }$ in $L_{1,loc}^{p}$
away from $\tilde{Z}_{\limfunc{an}}\cup \mathbf{E}$ where $\tilde{A}_{j,t}$
is the time $t$ degenerate $YM$ flow with initial condition $\tilde{A}_{j}$.
\end{proposition}

This will follow from a sequence of lemmas.

\begin{lemma}
\label{Lemma19}For any $t>0$, $\bigl\Vert \Lambda _{\omega _{0}}F_{\tilde{A}%
_{j,t}}\bigr\Vert _{L^{\infty }(\tilde{X}-\mathbf{E})}$ is uniformly bounded
in $j$. Moreover, for almost all $t>0$, $\lim_{j\to \infty }HYM^{\omega
_{0}}\left( \tilde{A}_{j,t}\right) =HYM(\mu ) $.
\end{lemma}

\begin{proof}
The first statement follows from Lemma \ref{Lemma18}. By assumption, we have 
$\Lambda _{\omega _{0}}F_{\tilde{A}_{j}}\rightarrow \Lambda _{\omega _{0}}F_{%
\tilde{A}_{\infty }}$ in $L^{1}$, and $\Lambda _{\omega _{0}}F_{\tilde{A}%
_{\infty }}$ has constant eigenvalues $\mu _{1},\cdots ,\mu _{K}$. Set $%
M^{2}=\sum_{i=1}^{K}\mu _{i}^{2}=\frac{HYM(\mu )}{2\pi }$. Also let $\mu
_{1,\varepsilon },\cdots ,\mu _{K,\varepsilon }$ be the $HN$ type of $(E,%
\bar{\partial}_{\tilde{A}_{0}})$ with respect to $\omega _{\varepsilon }$,
and set $\tilde{M}_{\varepsilon }^{2}=\sum_{i=1}^{K}$ $\mu _{i,\varepsilon
}^{2}$. By Corollary$\ $\ref{Cor3} we know:%
\begin{equation*}
\tilde{M}_{\varepsilon }\leq \frac{1}{2\pi }\int_{\tilde{X}}\left\vert
\Lambda _{\omega _{\varepsilon }}F_{\tilde{A}_{j,t}^{\varepsilon
}}\right\vert dvol_{\omega _{\varepsilon }}.
\end{equation*}%
By $\limfunc{Proposition}\ $\ref{Prop22} , for almost all $t$, we can find a
sequence $\varepsilon _{i}=\varepsilon _{i}(t)\rightarrow 0$ such that $%
\Lambda _{\omega _{\varepsilon _{i}}}F_{\tilde{A}_{j,t}^{\varepsilon
_{i}}}\rightarrow \Lambda _{\omega _{0}}F_{\tilde{A}_{j,t}}$ in any $%
L^{p}(\omega _{0})$. Let $\varepsilon _{i}\rightarrow 0$ and using the
convergence of the $HN$ type:%
\begin{equation*}
M\leq \frac{1}{2\pi }\int_{\tilde{X}}\left\vert \Lambda _{\omega _{0}}F_{%
\tilde{A}_{j,t}}\right\vert dvol_{\omega _{0}}
\end{equation*}%
for all $j$ and almost all $t\geq 0$. We also have:%
\begin{eqnarray*}
\left\vert \Lambda _{\omega _{\varepsilon }}F_{\tilde{A}_{j,t}^{\varepsilon
}}\right\vert (x) &\leq &\int_{\tilde{X}}K_{t}^{\varepsilon }(x,y)\left\vert
\Lambda _{\omega _{\varepsilon }}F_{\tilde{A}_{j}}\right\vert
(y)dvol_{\omega _{\varepsilon }}(y) \\
&=&M+\int_{\tilde{X}}K_{t}^{\varepsilon }(x,y)\left( \left\vert \Lambda
_{\omega _{\varepsilon }}F_{\tilde{A}_{j}}\right\vert -M\right) dvol_{\omega
_{\varepsilon }}
\end{eqnarray*}%
where $K_{t}^{\varepsilon }(x,y)$ is the heat kernel on $(\tilde{X},\omega
_{\varepsilon })$ (since $K_{t}^{\varepsilon }(x,y)$ has integral equal to $%
1 $). Since we have the bound: $K_{t}^{\varepsilon }(x,y)\leq C(1+1/t^{n})$,
there is a constant $C(t)$ independent of $\varepsilon $ such that:%
\begin{equation*}
\left\vert \Lambda _{\omega _{\varepsilon }}F_{\tilde{A}_{j,t}^{\varepsilon
}}\right\vert (x)\leq M+C\left\Vert \left\vert \Lambda _{\omega
_{\varepsilon }}F_{\tilde{A}_{j}}\right\vert -M\right\Vert _{L^{1}(\tilde{X}%
,\omega _{\varepsilon })}.
\end{equation*}%
Then just as above we have:%
\begin{equation*}
\left\vert \Lambda _{\omega _{0}}F_{\tilde{A}_{j,t}}\right\vert (x)\leq
M+C\left\Vert \left\vert \Lambda _{\omega _{0}}F_{\tilde{A}_{j}}\right\vert
-M\right\Vert _{L^{1}(\tilde{X},\omega _{0})}
\end{equation*}%
for almost all $x\in \tilde{X}-\mathbf{E}$ and almost all $t>0$. Since $%
\left\vert \Lambda _{\omega _{0}}F_{\tilde{A}_{j}}\right\vert \rightarrow
\left\vert \Lambda _{\omega _{0}}F_{\tilde{A}_{\infty }}\right\vert =M$ in $%
L^{1}$, we have%
\begin{equation*}
\lim_{j\rightarrow \infty }\sup \left\vert \Lambda _{\omega _{0}}F_{\tilde{A}%
_{j,t}}\right\vert (x)\leq M
\end{equation*}%
for almost all $x\in \tilde{X}-\mathbf{E}$ and almost all $t>0$. On the
other hand since $\Lambda _{\omega _{0}}F_{\tilde{A}_{j,t}}$ is uniformly
bounded in $j$, we can use the lower bound for 
\begin{equation*}
\frac{1}{2\pi }\int_{\tilde{X}}\left\vert \Lambda _{\omega _{0}}F_{\tilde{A}%
_{j,t}}\right\vert dvol_{\omega _{0}}
\end{equation*}%
and Fatou's Lemma to show:%
\begin{equation*}
M\leq \int_{\tilde{X}}\lim_{j\rightarrow \infty }\sup \left\vert \Lambda
_{\omega _{0}}F_{\tilde{A}_{j,t}}\right\vert dvol_{\omega _{0}}.
\end{equation*}%
It follows that $\lim_{j\rightarrow \infty }\sup \bigl\vert\Lambda _{\omega
_{0}}F_{\tilde{A}_{j,t}}\bigr\vert^{2}=M^{2}$ almost everywhere. By Fatou's
lemma we therefore have:%
\begin{eqnarray*}
HYM(\mu ) &\leq &\lim_{j\rightarrow \infty }\inf HYM^{\omega _{0}}(\tilde{A}%
_{j,t})\leq \lim_{j\rightarrow \infty }\sup HYM^{\omega _{0}}(\tilde{A}%
_{j,t}) \\
&=&\lim_{j\rightarrow \infty }\sup \int_{\tilde{X}}\left\vert \Lambda
_{\omega _{0}}F_{\tilde{A}_{j,t}}\right\vert dvol_{\omega _{0}}\leq \int_{%
\tilde{X}}\lim_{j\rightarrow \infty }\sup \left\vert \Lambda _{\omega
_{0}}F_{\tilde{A}_{j,t}}\right\vert dvol_{\omega _{0}} \\
&=&2\pi M^{2}=HYM(\mu ).
\end{eqnarray*}
\end{proof}

\begin{lemma}
\label{Lemma20}For almost all $t_{0}>0$, $\left\Vert \tilde{A}_{j,t}-\tilde{A%
}_{j,t_{0}}\right\Vert _{L^{2}(\tilde{X},\omega _{0})}\to 0$, uniformly for
almost all $t\geq t_{0}$.
\end{lemma}

\begin{proof}
\ As before let $\varepsilon _{i}\rightarrow 0$ be a sequence such that $%
\tilde{A}_{j,t}^{\varepsilon _{i}}\rightarrow $ $\tilde{A}_{j,t}$ and $%
\tilde{A}_{j,t_{0}}^{\varepsilon _{i}}\rightarrow \tilde{A}_{j,t_{0}}$ in $%
C_{loc}^{0}$. Then we again have: 
\begin{eqnarray*}
&&\left\Vert \tilde{A}_{j,t}^{\varepsilon _{i}}-\tilde{A}_{j,t_{0}}^{%
\varepsilon _{i}}\right\Vert _{L^{2}}\overset{Minkowski}{\leq }%
\int_{t_{0}}^{t}\left\Vert \frac{\partial \tilde{A}_{j,s}^{\varepsilon _{i}}%
}{\partial s}\right\Vert _{L^{2}} \\
&\leq &\sqrt{t}\left( \int_{t_{0}}^{t}\left\Vert d_{A_{j,s}}^{\ast }F_{%
\tilde{A}_{j,s}^{\varepsilon _{i}}}\right\Vert _{L^{2}}^{2}\right) ^{\frac{1%
}{2}}=\sqrt{t}\left( \int_{t_{0}}^{t}-\frac{1}{2}\frac{d}{ds}\left\Vert F_{%
\tilde{A}_{j,s}^{\varepsilon _{i}}}\right\Vert _{L^{2}}^{2}ds\right) ^{\frac{%
1}{2}} \\
&=&\sqrt{\frac{t}{2}}\left( YM(\tilde{A}_{j,t_{0}}^{\varepsilon _{i}})-YM(%
\tilde{A}_{j,t}^{\varepsilon _{i}})\right) ^{\frac{1}{2}}=\sqrt{\frac{t}{2}}%
\left( HYM(\tilde{A}_{j,t_{0}}^{\varepsilon _{i}})-HYM(\tilde{A}%
_{j,t}^{\varepsilon _{i}})\right) ^{\frac{1}{2}} \\
&\leq &\sqrt{\frac{t}{2}}\left( HYM(\tilde{A}_{j,t_{0}}^{\varepsilon
_{i}})-HYM(\mu _{\varepsilon _{i}})\right) ^{\frac{1}{2}}\text{.}
\end{eqnarray*}

Using $\limfunc{Proposition}\ $\ref{Prop22}$\,$and $\limfunc{Proposition}$ %
\ref{Prop17} this yields:%
\begin{equation*}
\left\Vert \tilde{A}_{j,t}-\tilde{A}_{j,t_{0}}\right\Vert _{L^{2}(\tilde{X}%
,\omega _{0})}\leq \sqrt{\frac{t}{2}}\left( HYM(\tilde{A}_{j,t_{0}})-HYM(\mu
)\right) ^{\frac{1}{2}}
\end{equation*}

The result follows by applying the previous lemma.
\end{proof}

\begin{lemma}
\label{Lemma21}There is a $YM$ connection $\tilde{A}_{\infty ,\ast }$ on a
reflexive sheaf $\tilde{E}_{\infty ,\ast }\rightarrow \tilde{X}-\mathbf{E}$
with the following property: for almost all $t>0$ there is a subsequence and
a closed subset $\tilde{Z}_{\limfunc{an}}^{t}\subset \tilde{X}-\mathbf{E}$,
possibly depending on $t$ and the choice of subsequence, such that $\tilde{A}%
_{j,t}\rightarrow \tilde{A}_{\infty ,\ast }$ in $L_{1,loc}^{p}$ ($p>n$) away
from $\limfunc{Sing}(\tilde{E}_{\infty ,\ast })\cup \tilde{Z}_{\limfunc{an}%
}^{t}\cup \mathbf{E}$.
\end{lemma}

\begin{proof}
As in $\limfunc{Proposition}$ \ref{Prop21} and using $\limfunc{Proposition}$ %
\ref{Prop22} we have:%
\begin{equation*}
HYM(\tilde{A}_{j,t_{1}})-HYM(\tilde{A}_{j,t_{2}})\geq
2\int_{t_{1}}^{t_{2}}\left\Vert d_{A_{j,s}}^{\ast }F_{\tilde{A}%
_{j,s}}\right\Vert _{L^{2}(\omega _{0})}ds
\end{equation*}%
for almost all $t_{2}\geq t_{1}>0$. It follows from Lemma \ref{Lemma19} and
Fatou's lemma that:%
\begin{equation*}
\lim \inf_{j\rightarrow \infty }\left\Vert d_{A_{j,s}}^{\ast }F_{\tilde{A}%
_{j,s}}\right\Vert _{L^{2}(\omega _{0})}^{2}=0,
\end{equation*}%
for almost all $t$. Choose a sequence $t_{k}$ of such $t$ with $%
t_{k}\rightarrow 0$. For each $k$ there is a subsequence $j_{m}(t_{k})$, a $%
YM$ connection $\tilde{A}_{\infty ,t_{k}}$, and a set of Hausdorff
codimension at least $4$ which we denote by $\tilde{Z}_{\limfunc{an}%
}^{t_{k}} $, depending on the choice of subsequence such that $\tilde{A}%
_{j_{m},t_{k}}\rightarrow \tilde{A}_{\infty ,t_{k}}$ in $L_{1,loc}^{p}$ away
from $\tilde{Z}_{\limfunc{an}}^{t_{k}}$. By a diagonalisation argument,
assume without loss of generality that the original sequence satisfies $%
\tilde{A}_{j,t_{k}}\rightarrow \tilde{A}_{\infty ,t_{k}}$ for all $t_{k}$.
On the other hand, by Lemma \ref{Lemma20}, $\tilde{A}_{\infty ,t_{k}}=\tilde{%
A}_{\infty ,\ast }$ is independent of $t_{k}$. For any $t$, there is a $k$
with $t\geq t_{k}$, so Lemma \ref{Lemma20} also implies $\tilde{A}%
_{j,t}\rightarrow \tilde{A}_{\infty ,\ast }$ in $L_{loc}^{2}$ for almost all 
$t>0$. Hence, any Uhlenbeck limit of $\tilde{A}_{j,t}$ coincides with $%
\tilde{A}_{\infty ,\ast }$.
\end{proof}

The proof of $\limfunc{Proposition}$ \ref{Prop25} will be complete if we can
show $\tilde{A}_{\infty }=\tilde{A}_{\infty ,\ast }$. First we will need:

\begin{lemma}
\label{Lemma22}$\Lambda _{\omega _{\varepsilon }}F_{\tilde{A}_{j,t}}$ is
bounded on compact subsets of $\tilde{X}-\mathbf{E}$, uniformly for all $j$,
all $t\geq 0$, and all $\varepsilon >0$.
\end{lemma}

\begin{proof}
By our assumptions it follows that $\Lambda _{\omega _{\varepsilon }}F_{%
\tilde{A}_{j}}$ are uniformly bounded in $L^{1}$ and that they are uniformly
locally bounded. The result now follows just as in the proof of Lemma \ref%
{Lemma16}$(2)$.
\end{proof}

\begin{corollary}
\label{Cor8}$\left\vert \tilde{A}_{j,t}-\tilde{A}_{\infty }\right\vert $ is
bounded in any $L_{1,loc}^{p}$ away from $\tilde{Z}_{\limfunc{an}}\cup 
\mathbf{E}$, uniformly for all $j$ and all $0\leq t\leq t_{0}$. In
particular, the singular set $\tilde{Z}_{\limfunc{an}}^{t}$ is independent
of $t$ and is equal to $\tilde{Z}_{\limfunc{an}}$.
\end{corollary}

\begin{proof}
Since $\tilde{A}_{j}\rightarrow \tilde{A}_{\infty }$ in $L_{1,loc}^{p}$, it
suffices to prove that $\left\vert \tilde{A}_{j,t}-\tilde{A}_{j}\right\vert $
is bounded in $C_{loc}^{1}$. Choose a sequence $\varepsilon _{i}$ such that $%
\tilde{A}_{j,t}^{\varepsilon _{i}}\rightarrow \tilde{A}_{j,t}$ in $%
C_{loc}^{1}$. It suffices to prove $\left\vert \tilde{A}_{j,t}^{\varepsilon
_{i}}-\tilde{A}_{j}\right\vert $ is bounded in $C_{loc}^{1}$ uniformly in $%
\varepsilon _{i}$. Write $\tilde{A}_{j,t}^{\varepsilon _{i}}=\tilde{g}%
_{j,t}^{\varepsilon _{i}}(\tilde{A}_{j})$ and $\tilde{k}_{j,t}^{\varepsilon
_{i}}=(\tilde{g}_{j,t}^{\varepsilon _{i}})^{\ast }\tilde{g}%
_{j,t}^{\varepsilon _{i}}$. It suffices to show that $(\tilde{k}%
_{j,t}^{\varepsilon _{i}})^{-1}$ is bounded and $\tilde{k}%
_{j,t}^{\varepsilon _{i}}$ has bounded derivatives, locally with respect to
a trivialisation of $\tilde{E}$. The local boundedness of $\tilde{k}%
_{j,t}^{\varepsilon _{i}}$ and $(\tilde{k}_{j,t}^{\varepsilon _{i}})^{-1}$
follows from the flow equations and the preceding lemma. Namely, it is easy
to see that the determinant and trace of these endomorphisms are bounded,
and this easily implies the boundedness of the endomorphisms themselves. The
boundedness of the derivatives follows from \cite{BS} $\limfunc{Proposition}$
$1\ $applied to the equation%
\begin{equation*}
\triangle _{(\bar{\partial}_{A_{0}},\omega _{\varepsilon })}\tilde{k}%
_{\varepsilon ,t}-\sqrt{-1}\Lambda _{\omega _{\varepsilon }}\left( \bar{%
\partial}_{A_{0}}\tilde{k}_{\varepsilon ,t}\right) \tilde{k}_{\varepsilon
,t}^{-1}\left( \partial _{A_{0}}\tilde{k}_{\varepsilon ,t}\right) =\tilde{k}%
_{\varepsilon ,t}f_{\varepsilon ,t}.
\end{equation*}
\end{proof}

Now we can complete the proof of $\limfunc{Proposition}$ \ref{Prop25}. Fix a
smooth test form $\phi \in \Omega ^{1}(\tilde{X},\mathfrak{u}(E))$,
compactly supported away from $\tilde{Z}_{\limfunc{an}}\cup \mathbf{E}$.
Choose $0<\delta \leq 1$. For $\varepsilon >0$ we have:%
\begin{eqnarray*}
\int_{\tilde{X}}\left\langle \phi ,\tilde{A}_{j,\delta }^{\varepsilon
}-A_{j}\right\rangle dvol_{\omega _{\varepsilon }} &=&\int_{0}^{\delta
}dt\int_{\tilde{X}}\left\langle \phi ,\frac{\partial \tilde{A}%
_{j,t}^{\varepsilon }}{\partial t}\right\rangle dvol_{\omega _{\varepsilon }}
\\
(flow\text{ }equations) &=&-\int_{0}^{\delta }dt\int_{\tilde{X}}\left\langle
\phi ,\left( d_{\tilde{A}_{j,t}^{\varepsilon }}\right) ^{\ast }F_{\tilde{A}%
_{j,t}^{\varepsilon }}\right\rangle dvol_{\omega _{\varepsilon }} \\
(K\ddot{a}hler\text{ }identities) &=&\sqrt{-1}\int_{0}^{\delta }dt\int_{%
\tilde{X}}\left\langle \phi ,\left( \partial _{\tilde{A}_{j,t}^{\varepsilon
}}-\bar{\partial}_{\tilde{A}_{j,t}^{\varepsilon }}\right) \Lambda _{\omega
_{\varepsilon }}F_{\tilde{A}_{j,t}^{\varepsilon }}\right\rangle dvol_{\omega
_{\varepsilon }}\text{ } \\
&=&\sqrt{-1}\int_{0}^{\delta }dt\int_{\tilde{X}}\left\langle \left( \partial
_{\tilde{A}_{j,t}^{\varepsilon }}-\bar{\partial}_{\tilde{A}%
_{j,t}^{\varepsilon }}\right) ^{\ast }\phi ,\Lambda _{\omega _{\varepsilon
}}F_{\tilde{A}_{j,t}^{\varepsilon }}\right\rangle dvol_{\omega _{\varepsilon
}}.
\end{eqnarray*}%
By Lemma \ref{Lemma22}, $\Lambda _{\omega _{\varepsilon }}F_{\tilde{A}%
_{j,t}^{\varepsilon }}$ is bounded on the support of $\phi $ for all $j$,
all $\varepsilon >0$, and all $0\leq t\leq \delta $, and the bound may be
taken to be independent of $\delta $. Therefore:%
\begin{equation*}
\int_{\tilde{X}}\left\langle \phi ,\tilde{A}_{j,\delta }^{\varepsilon
}-A_{j}\right\rangle dvol_{\omega _{\varepsilon }}\leq C\int_{0}^{\delta
}dt\left\Vert \left( \partial _{\tilde{A}_{j,t}^{\varepsilon }}-\bar{\partial%
}_{\tilde{A}_{j,t}^{\varepsilon }}\right) ^{\ast }\phi \right\Vert
_{L^{1}(\omega _{0})}.
\end{equation*}%
Applying this inequality to a sequence, $\tilde{A}_{j,t}^{\varepsilon
_{i}}\rightarrow \tilde{A}_{j,t}$ in $C_{loc}^{1}$, 
\begin{equation*}
\left\vert \int_{\tilde{X}}\left\langle \phi ,\tilde{A}_{j,\delta
}-A_{j}\right\rangle dvol_{\omega _{0}}\right\vert \leq C\int_{0}^{\delta
}dt\left\Vert \left( \partial _{\tilde{A}_{j,t}}-\bar{\partial}_{\tilde{A}%
_{j,t}}\right) ^{\ast }\phi \right\Vert _{L^{1}(\omega _{0})}.
\end{equation*}%
By the Corollary \ref{Cor8}, $\left\vert \tilde{A}_{j,t}-\tilde{A}_{\infty
}\right\vert $ is locally bounded in any $L^{p}$ independently of $j$. Then 
\begin{equation*}
\left\vert \int_{\tilde{X}}\left\langle \phi ,\tilde{A}_{j,\delta
}-A_{j}\right\rangle dvol_{\omega _{0}}\right\vert \leq C\delta \text{ }
\end{equation*}%
where $C$ depends only on the $L^{1}$ norm of $\partial _{\tilde{A}_{\infty
}}\phi $, $\bar{\partial}_{\tilde{A}_{\infty }}\phi $ and the bounds on $%
\Lambda _{\omega _{\varepsilon }}F_{\tilde{A}_{j,t}^{\varepsilon }}$ and $%
\left\vert \tilde{A}_{j,t}-\tilde{A}_{\infty }\right\vert $. In particular $%
C $ is independent of $j$. Taking limits as $j\rightarrow \infty $ we have:%
\begin{equation*}
\left\vert \int_{\tilde{X}}\left\langle \phi ,\tilde{A}_{\infty ,\delta
}-A_{\infty }\right\rangle dvol_{\omega _{0}}\right\vert \leq C\delta
\end{equation*}%
and since $\delta $ and was arbitrary and $\tilde{A}_{\infty ,\delta }=%
\tilde{A}_{\infty ,\ast }$ for almost all small $\delta $, this implies $%
\tilde{A}_{\infty ,\ast }=A_{\infty }$. This concludes the proof of $%
\limfunc{Proposition}\ $\ref{Prop25}$\,$and hence the proof of the main
theorem.


\bigskip

\begin{thebibliography}{CHLI}
\bibitem[AB]{AB} M.F. Atiyah and R. Bott, \textit{The Yang Mills equations
over Riemann surfaces, }Phil. Trans. R. Soc. Lond. A \textbf{308} (1986),
523-615.

\bibitem[B]{B} S. Bando, \textquotedblleft Removable singularities for
holomorphic vector bundles\textquotedblright , Tohoku Math. J. (2) \textbf{43%
} (1991), no. 1, 61-67.

\bibitem[BS]{BS} S. Bando and Y.-T.Siu, \textit{Stable sheaves and
Einstein-Hermitian metrics}, in \textquotedblleft Geometry and Analysis on
Complex Manifolds\textquotedblright , World Scientific, 1994, 39-50.

\bibitem[BU1]{BU1} N. Buchdahl, \textit{Hermitian-Einstein connections and
stable vector bundles over compact complex surfaces}, Math. Ann. \textbf{280 
}(1988\textbf{), }625-648 .

\bibitem[BU2]{BU2} N. Buchdahl, \textit{Sequences of stable bundles over
compact complex surfaces,} J. Geom. Anal. \textbf{9}(3), (1999), 391-427.

\bibitem[CHLI]{CHLI} S.-Y. Cheng and P. Li, \textit{Heat kernel estimates
and lower bound of eigenvalues}, Comment. Math. Helvetici \textbf{56 }%
(1981), 327-338.

\bibitem[C]{C} S.D. Cutkosky, \textquotedblleft Resolution of
Singularities\textquotedblright , American Mathematical Society, Graduate
Studies in Mathematics v. 63, 2004.

\bibitem[D]{D} G. Daskalopoulos, \textit{The topology of the space of stable
bundles on a Riemann surface}, J. Diff. Geom. \textbf{36} (1992), 699-742.

\bibitem[DW1]{DW1} G. Daskalopoulos and R.A. Wentworth, \textit{Convergence
properties of the Yang-Mills flow on K\"{a}hler surfaces}, J.Reine Angew.
Math, \textbf{575} (2004), 69-99.

\bibitem[DW2]{DW2} G. Daskalopoulos and R.A. Wentworth, \textit{On the
blow-up set of the Yang-Mills flow on K\"{a}hler surfaces}, Mathematische
Zeitschrift, \textbf{256} (2007), 301-310.

\bibitem[DW3]{DW3} G. Daskalopoulos and R.A. Wentworth, \textit{Convergence
properties of the Yang-Mills flow on Algebraic Surfaces (unpublished
preprint).}

\bibitem[DO1]{DO1} S.K. Donaldson, \textit{Anti self-dual Yang-Mills
connections over complex algebraic surfaces and stable vector bundles},
Proc. London. Math. Soc. \textbf{50} (1985), 1-26.

\bibitem[DO2]{DO2} S.K. Donaldson, \textit{Infinite determinants, stable
bundles, and curvature}, Duke Math. J. \textbf{54} (1987), 231--247.

\bibitem[DO3]{DO3} S.K. Donaldson, \textit{A new proof of a theorem of
Narasimhan and Seshadri}, J. Differential Geom. \textbf{18} (1983), 269-277

\bibitem[DOKR]{DOKR} S.K. Donaldson and P.B. Kronheimer, \textquotedblleft
The Geometry of Four-Manifolds\textquotedblright , Oxford Science, Clarendon
Press, Oxford, 1990.

\bibitem[GR]{GR} A. Grigor'yan, \textit{Gaussian upper bounds for the heat
kernel on arbitary manifolds}, J. Differential Geom. \textbf{45} (1997),
33-52.

\bibitem[H1]{H1} H. Hironaka, \textit{Resolution of singularities of an
algebraic variety over a field of characteristic zero, }Ann. of Math. (2) 
\textbf{79} (1), 109--203.

\bibitem[H2]{H2} H. Hironaka, \textit{Resolution of singularities of an
algebraic variety over a field of characteristic zero part II, }Ann. of
Math. (2) \textbf{79} (1), 205--326.

\bibitem[HT]{HT} M. -C. Hong and G. Tian, \textit{Asymptotical behaviour of
the Yang-Mills flow and singular Yang-Mills connections}, Math. Ann. \textbf{%
330} (2004), no. 3, 441--472.

\bibitem[J1]{J1} A. Jacob, \textit{Existence of approximate
Hermitian-Einstein structures on semi-stable bundles, }arXiv:1012.1888v2.

\bibitem[J2]{J2} A. Jacob, \textit{The limit of the Yang-Mills flow on
semi-stable bundles}, arXiv:1104.4767.

\bibitem[J3]{J3} A. Jacob, \textit{The Yang-Mills flow and the Atiyah-Bott
formula on compact K\"{a}hler manifolds, }arXiv:1109.1550.

\bibitem[KOB]{KOB} S. Kobayashi, \textit{Differential Geometry of Complex
Vector Bundles}, Princeton University Press, 1987.

\bibitem[KO]{KO} J. Koll\'{a}r, \textquotedblleft Lectures on Resolution of
Singularities\textquotedblright , Princeton University Press, 2007.

\bibitem[NA]{NA} H. Nakajima, \textit{Compactness of the moduli space of
Yang-Mills connections in higher dimensions}, J. Math. Soc. Japan \textbf{40}%
, (1988), 383-392.

\bibitem[OSS]{OSS} C. Okonek, M. Scheider, and H. Spindler,
\textquotedblleft Vector Bundles Over Complex Projective
Space\textquotedblright , Birkhauser, Boston, 1980.

\bibitem[R]{R} J. R\aa de, \textit{On the Yang-Mills heat equation in two
and three dimensions}, J. Reine Angew. Math. 431 (1992), 123--163.

\bibitem[SH]{SH} S. Shatz, \textit{The decomposition and specialization of
algebraic families of vector bundles}, Compositio Math. \textbf{35} (1977),
163-187.

\bibitem[SHI]{SHI} B. Shiffman, \textit{On the removal of singularities of
analytic sets}, Michigan Math. J. \textbf{15} (1968), 111-120.

\bibitem[SI]{SI} C. Simpson, \textit{Constructing variations of Hodge
structure using Yang-Mills theory and applications to uniformization}, J.
Amer. Math. Soc. \textbf{1} (1988), 867--918.

\bibitem[SIU]{SIU} Siu, Y.-T., \textit{A Hartogs type extension theorem for
coherent analytic sheaves}, Ann. of Math. (2) \textbf{93} (1971), no. 1,
166-188

\bibitem[U1]{U1} K. Uhlenbeck, \textit{Removable singularities in Yang-Mills
fields, Comm. Math. Phys.} \textbf{83} (1982), no. 1, 11--29.

\bibitem[U2]{U2} K. Uhlenbeck, A priori estimates for Yang-Mills fields,
unpublished.

\bibitem[UY]{UY} K. Uhlenbeck and S.-T. Yau, \textit{On the existence of
Hermitian-Yang-Mills connections in stable vector bundles}, Comm. Pure Appl.
Math. \textbf{39} (1986), S257--S293.

\bibitem[VO]{VO} C. Voisin, Hodge Theory and complex algebraic geometry I
and II, Cambridge University Press 2002-3

\bibitem[WIL]{WIL} G. Wilkin, \textit{Morse theory for the space of Higgs
bundles}. Comm. Anal.Geom., \textbf{16}(2):283--332, 2008.
\end{thebibliography}
\end{document}